\documentclass[12pt]{article}
\usepackage[french, english]{babel}
\usepackage{graphicx}
\usepackage{amsmath}
\usepackage{amsfonts}
\usepackage{amssymb}
\usepackage{amsthm}
\usepackage{hyperref}
\usepackage{mathtools}
\usepackage{comment}

\input xy
\xyoption{all}

\numberwithin{equation}{section}

\newtheorem{thm}{Th\'eor\`eme}[section]
\newtheorem{cor}[thm]{Corollaire}
\newtheorem{lem}[thm]{Lemme}
\newtheorem{pro}[thm]{Proposition}

\theoremstyle{definition}
\newtheorem{defi}[thm]{D\'efinition}

\newtheoremstyle{remarque}{}{}{}{}{\it}{.}{\newline}{}
\theoremstyle{remarque}
\newtheorem*{rem}{Remarque}
\newtheorem*{rems}{Remarques}

\newcommand{\asd}[5]{%
\setbox1=\hbox{\ensuremath{^{#1}}}%
\setbox2=\hbox{\ensuremath{_{#2}}}%
\setbox5=\hbox{\ensuremath{#5}}%
\hspace{\ifnum\wd1>\wd2\wd1\else\wd2\fi}%
\ensuremath{\copy5^{\hspace{-\wd1}\hspace{-\wd5}#1\hspace{\wd5}#3}%
_{\hspace{-\wd2}\hspace{-\wd5}#2\hspace{\wd5}#4}%
}}

\usepackage[OT2,T1]{fontenc}
\DeclareSymbolFont{cyrletters}{OT2}{wncyr}{m}{n}
\DeclareMathSymbol{\Sha}{\mathalpha}{cyrletters}{"58}
\DeclareMathSymbol{\Brusse}{\mathalpha}{cyrletters}{"42}

\newcommand{\n}{\mathbb{N}}
\newcommand{\z}{\mathbb{Z}}
\newcommand{\q}{\mathbb{Q}}
\newcommand{\f}{\mathbb{F}}
\renewcommand{\a}{\mathbb{A}}

\renewcommand{\ker}{\mathrm{Ker }}
\newcommand{\im}{\mathrm{Im}}
\renewcommand{\hom}{\mathrm{Hom}}
\newcommand{\ext}{\mathrm{Ext}}
\renewcommand{\inf}{\mathrm{Inf}}
\newcommand{\res}{\mathrm{Res}}
\newcommand{\aut}{\mathrm{Aut}\,}
\newcommand{\out}{\mathrm{Out}\,}
\renewcommand{\int}{\mathrm{Int}\,}
\newcommand{\gm}{\mathbb{G}_{\mathrm{m}}}
\newcommand{\br}{\mathrm{Br}\,}
\newcommand{\nr}{\mathrm{nr}}
\newcommand{\brun}{\mathrm{Br}_1}
\newcommand{\brnr}{{\mathrm{Br}_{\nr}}}
\newcommand{\al}{\mathrm{al}}
\newcommand{\bral}{\mathrm{Br}_{\al}}
\newcommand{\brnral}{\mathrm{Br}_{\nr,\al}}
\newcommand{\gal}{\mathrm{Gal}}
\newcommand{\sln}{\mathrm{SL}_n}

\newcommand{\pic}{\mathrm{Pic}\,}
\newcommand{\Pic}{\underline{\mathrm{Pic}}\,}
\newcommand{\ns}{\mathrm{NS}\,}
\renewcommand{\div}{\mathrm{div}}
\newcommand{\Div}{\mathrm{Div}}
\newcommand{\ab}{\mathrm{ab}}
\newcommand{\der}{\mathrm{der}}
\newcommand{\spec}{\mathrm{Spec}\,}
\renewcommand{\cal}[1]{\mathcal{#1}}
\newcommand{\bb}[1]{\mathbb{#1}}
\renewcommand{\O}{\cal{O}}

\renewcommand{\a}{\mathfrak{a}}
\renewcommand{\b}{\mathfrak{b}}

\title{Groupe de Brauer non ramifi\'e des espaces homog\`enes \`a stabilisateur fini}
\author{Giancarlo Lucchini Arteche\\[5mm]
{\it\small D\'epartement de math\'ematique, universit\'e Paris-Sud}\\
{\it\small b\^atiment 425, 91405 Orsay cedex, France}\\
{\small giancarlo.lucchini@math.u-psud.fr}
}

\date{}

\begin{document}

\selectlanguage{french}
\maketitle

\begin{abstract}
On d\'eveloppe diff\'erentes formules de nature alg\'ebrique et/ou arithm\'etique permettant de calculer explicitement, tant sur un corps fini que sur
un corps de caract\'eristique $0$, la partie alg\'ebrique du groupe de Brauer non ramifi\'e d'un espace homog\`ene $G\backslash G'$ sous un groupe
lin\'eaire $G'$ semi-simple simplement connexe \`a stabilisateur fini $G$. On donne aussi des exemples de calculs que l'on peut faire avec ces formules.\\

Mots cl\'es : groupe de Brauer, espaces homog\`enes, cohomologie galoisienne, groupes finis.
\end{abstract}

\selectlanguage{english}
\begin{abstract}
{\bf Unramified Brauer group of homogeneous spaces with finite stabilizer.} We develop different formulas of algebraic and/or arithmetic nature allowing
an explicit calculation, both over a finite field and over a field of characteristic $0$, of the algebraic part of the unramified Brauer group of a homogeneous
space $G\backslash G'$ under a semisimple simply connected linear group $G'$ with finite stabilizer $G$. We also give examples of the calculations that can be done
with these formulas.\\

Key words: Brauer group, homogeneous spaces, Galois cohomology, finite groups.
\end{abstract}

\selectlanguage{french}

\section{Introduction}
Le groupe de Brauer $\br V$ d'une vari\'et\'e $V$ sur un corps $k$ est un invariant cohomologique qui permet, via certains de ses sous-quotients, de
r\'epondre parfois par la n\'egative \`a la question de la rationalit\'e de la vari\'et\'e $V$, ainsi que, dans le cas o\`u $k$ est un corps global, \`a des
questions de nature arithm\'etique telles que le principe de Hasse et l'approximation faible via l'obstruction de Brauer-Manin. Ces deux raisons suffisent
en elles m\^emes pour justifier une \'etude approfondie de ce groupe et de ses sous-quotients, notamment de sa partie \emph{non ramifi\'ee} $\brnr V$,
de sa partie \emph{alg\'ebrique} $\bral V$, ou encore de la partie \emph{non ramifi\'ee alg\'ebrique} $\brnral V$.

Dans le cadre des vari\'et\'es $V$ qui sont des \emph{espaces homog\`enes} sous un groupe alg\'ebrique lin\'eaire connexe $G'$, les travaux les
plus approfondis sont ceux pr\'esent\'es par Borovoi, Demarche et Harari dans \cite{BDH}. Ils ont d\'evelopp\'e notamment, en suivant une m\'ethode
pr\'esent\'ee par Colliot-Th\'el\`ene et Kunyavski\u\i\, dans \cite{ColliotKunyavskii} pour les espaces principaux homog\`enes, des formules de nature alg\'ebrique
pour le groupe de Brauer non ramifi\'e alg\'ebrique $\brnral V$ lorsque le stabilisateur de la vari\'et\'e est de type ``ssumult'', cas qui contient notamment les
stabilisateurs connexes et les stabilisateurs de type multiplicatif. Or, malgr\'e la vaste g\'en\'eralit\'e de ces r\'esultats, le cas des espaces homog\`enes
\`a stabilisateur fini semble \'echapper \`a leurs m\'ethodes d\`es que le stabilisateur en question est un groupe non ab\'elien.
Le but de cet article est donc de traiter ce cas particulier. En nous inspirant de la m\^eme m\'ethode de Colliot-Th\'el\`ene et Kunyavski\u\i\, et en suivant le
m\^eme chemin que Borovoi, Demarche et Harari (qui consiste \`a traiter d'abord le cas des corps finis pour ensuite traiter le cas d'un corps
quelconque de caract\'eristique 0), on obtient des formules explicites pour le groupe de Brauer non ramifi\'e alg\'ebrique $\brnral V$ d'un espace homog\`ene
$V=G\backslash G'$ sous un groupe $G'$ semi-simple et simplement connexe \`a stabilisateur fini $G$.\\

Le chemin pour arriver \`a ces formules est relativement d\'etourn\'e. Il faut d'abord traiter le cas des vari\'et\'es sur un corps fini, ce qu'on fait dans
la section \ref{section corps fini}. Le r\'esultat de cette section (th\'eor\`eme \ref{theoreme sans compactification lisse corps fini}) caract\'erise les
\'el\'ements non ramifi\'es de $\br V$, o\`u $V$ est une vari\'et\'e lisse et g\'eom\'etriquement int\`egre sur un corps fini, via l'\'evaluation de ces
\'el\'ements sur les ``points locaux'' de $V$. L'outil principal de la preuve est un th\'eor\`eme r\'ecent de Gabber qui pr\'ecise le th\'eor\`eme de de Jong
sur les alt\'erations, cf. \cite[Expos\'e X, Theorem 2.1]{ILO}. Gr\^ace \`a ce r\'esultat, Borovoi, Demarche et Harari ont pu montrer que les \'el\'ements de
$\brnr V$ v\'erifient le crit\`ere en question, tant sur un corps global de caract\'eristique positive \cite[Proposition 4.2]{BDH} que sur un corps fini
\cite[Proposition 4.3]{BDH}. La r\'eciproque de ce r\'esultat (\`a vrai dire, une version plus faible) n'a pourtant \'et\'e donn\'ee que r\'ecemment pour les corps
globaux, toujours en utilisant le r\'esultat de Gabber, cf. \cite[Th\'eor\`eme 2.1]{GLABrnr}. Le th\'eor\`eme
\ref{theoreme sans compactification lisse corps fini} ci-dessous donne alors cette r\'eciproque dans le cas d'un corps fini, donnant ainsi un crit\`ere pour
l'appartenance au groupe de Brauer non ramifi\'e $\brnr V$ aussi dans ce cas.\\

La section \ref{section formules} contient les r\'esultats principaux de ce texte. On se concentre sur les espaces homog\`enes $V$ sous un groupe $G'$
semi-simple simplement connexe \`a stabilisateur fini sur diff\'erents corps de base $k$ et on donne dans l'ordre :
\begin{itemize}
\item[\S\ref{section formule cohomologique} :] une formule cohomologique pour $\brnral V$ sur un corps fini ;
\item[\S\ref{section formule algebrique} :] une formule alg\'ebrique pour $\brnral V$ sur un corps fini ;
\item[\S\ref{section formule algebrique car 0} :] une formule alg\'ebrique pour $\brnral V$ sur un corps de caract\'eristique 0 ;
\item[\S\ref{section formule cohomologique corps local} :] une formule cohomologique pour $\brnral V$ sur un corps local, sous quelques hypoth\`eses suppl\'ementaires.
\end{itemize}
Ici, et dans toute la suite du texte, on entend par formule cohomologique (resp. alg\'ebrique) une formule d\'ependant des th\'eor\`emes de dualit\'e
en cohomologie galoisienne pour des corps locaux (resp. une formule d\'ependant de la structure de groupe alg\'ebrique du stabilisteur).\\

Enfin, la section \ref{section applications} est consacr\'ee aux applications des diff\'erents r\'esultats. On donne d'abord, via les formules
alg\'ebriques, des r\'esultats de trivialit\'e du groupe $\brnral V$ sur un corps fini, puis sur un corps de caract\'eristique 0
(propositions \ref{proposition nullite du brnral si G abelien}, \ref{proposition brnr nul si exp G divise q-1}, \ref{proposition brnral nul si exp G premier a q-1},
\ref{proposition brnr nul si exp G divise mu_k} et \ref{proposition brnral nul si exp G premier a mu_k}), dont certains \'etaient d\'ej\`a connus auparavant.
On montre aussi avec un exemple que ces formules se pr\^etent \`a des calculs totalement explicites et faisables \`a la main pour des petits groupes (propositions
\ref{proposition exemple Demarche sur fq} et \ref{proposition exemple Demarche en car 0}). Ces calculs permettent de montrer notamment que $\brnr V$ est en
g\'en\'eral plus grand que les groupes de type ``$\Sha^1_{\mathrm{cyc}}$'', lesquels suffissent pour d\'ecrire le groupe de Brauer non ramifi\'e dans le cas des
stabilisateurs connexes ou ab\'eliens (cf. \cite{BDH}). Cette diff\'erence essentielle par rapport aux travaux de Colliot-Th\'el\`ene, Kunyavski\u\i, Borovoi,
Demarche et Harari est due au fait que, sur un corps fini et dans notre contexte des stabilisateurs finis, le groupe de Brauer non ramifi\'e alg\'ebrique
$\brnral V$ n'a aucune raison d'\^etre trivial (par oppposition aux autres cas, cf. \cite[Th\'eor\`eme 7.5]{BDH} et \cite[Proposition 1]{ColliotKunyavskii}
et comparer avec la proposition \ref{proposition exemple Demarche sur fq} ci-dessous).

En un deuxi\`eme temps, on donne une application arithm\'etique de ces formules, notamment de leur relation avec l'obstruction de Brauer-Manin. Dans
ce sens on a pu montrer que cette obstruction ``ne voit pas'' les places r\'eelles d'un corps de nombres (proposition
\ref{proposition BM ne voit pas les places reelles}).\\

Ce texte contient enfin deux sections en appendice. La section \ref{section changement de base} se concentre sur les vari\'et\'es $V$ lisses et
s\'eparablement rationnellement connexes (donc en particulier sur $V$ s\'eparablement unirationnelle, comme c'est le cas pour les espaces homog\`enes
\'etudi\'es dans ce texte). Pour ces vari\'et\'es, on montre un r\'esultat de ``rigidit\'e'' du groupe de Brauer non ramifi\'e alg\'ebrique $\brnral V$.
Ce r\'esultat, qui suppose l'existence d'une compactification lisse de $V$, \'etait initialement pr\'evu pour obtenir une formule comme celle de la section
\ref{section corps fini} via changement de base \`a un corps global (corps pour lequel des telles formules existent d\'ej\`a, cf. \cite[Th\'eor\`emes 2.1 et 3.1]{GLABrnr}),
mais le th\'eor\`eme de Gabber mentionn\'e ci-dessus nous a permis de nous d\'ebarasser de cette hypoth\`ese. On pr\'esente n\'eanmoins ce r\'esultat
dans l'espoir qu'il puisse \^etre utile pour d'autres recherches, car il ne semble pas sans int\'er\^et.

La section \ref{section inseparable} contient pour sa part quelques r\'esultats sur des extensions ins\'eparables de corps et d'anneaux de valuation discr\`ete.
Ces r\'esultats \'el\'ementaires, pour lesquels on n'a pas pu trouver de bonne r\'ef\'erence, sont utilis\'es notamment dans la preuve du th\'eor\`eme
\ref{theoreme sans compactification lisse corps fini}.

\paragraph*{Remerciements}
L'auteur tient \`a remercier David Harari, Cyril Demarche et le rapporteur pour la lecture soigneuse qu'ils ont fait de ce texte, ainsi que Jean-Louis Colliot-Th\'el\`ene
et Philippe Gille pour leurs remarques.

\section{Notations et conventions}\label{section notations}
Dans tout ce texte, $k,K,L$ repr\'esentent des corps de caract\'eristique $p\geq 0$. Pour un corps $k$, on note toujours $\bar k$ une cl\^oture s\'eparable de
$k$ et $\Gamma_k:=\gal(\bar k/k)$ le groupe de Galois absolu. Dans le cas des corps finis, $q$ repr\'esente toujours le cardinal du corps
et il s'agit donc d'une puissance de $p$. Dans le cas des corps globaux, i.e. des corps de nombres et des corps de fonctions d'une courbe sur un corps fini,
on note $\Omega_k$ l'ensemble des places de $k$, lesquelles sont not\'ees $v,w$, et l'on note $k_v,k_w$, etc. les compl\'et\'es de ces corps par rapport \`a
ces places. Par ``corps local'', on entend pr\'ecis\'ement une telle compl\'etion d\`es qu'elle est non archim\'edienne, i.e. une extension finie du corps
$\q_\ell$ dans le cas de caract\'eristique 0 ou bien un corps de la forme $k((t))$ avec $k$ un corps fini dans le cas de caract\'eristique $p>0$.

On note $V$ une $k$-vari\'et\'e, laquelle, sauf mention du contraire, est toujours suppos\'ee lisse et g\'eom\'etriquement int\`egre. On notera $X$ une
$k$-compactification de $V$ (que l'on supposera lisse aussi, \`a d\'efaut de la mention du contraire), i.e. une $k$-vari\'et\'e propre munie
d'une $k$-immersion ouverte $V\hookrightarrow X$. \footnote{On rappelle que l'existence d'une compactification lisse est assur\'ee en caract\'eristique
0 via le th\'eor\`eme d'Hironaka, tandis qu'en caract\'ersitique positive la dite existence est une question toujours ouverte.}  On r\'eserve la notation
$G'$ pour un $k$-groupe semi-simple simplement connexe (par exemple $\sln$) et $G$ pour un $k$-groupe fini,
souvent plong\'e dans $G'$. On supposera dans tout l'expos\'e que l'ordre du groupe $G$ (\`a savoir, le degr\'e du morphisme structurel $G\to\spec k$)
est premier \`a la caract\'eristique $p$ du corps $k$. Ainsi, le groupe $G$ est toujours suppos\'e lisse et en particulier le cardinal de $G(\bar k)$ est toujours premier
\`a $p$. La $k$-vari\'et\'e $V:=G\backslash G'$ est lisse et g\'eom\'etriquement int\`egre et correspond \`a un espace homog\`ene sous $G'$ (cf.
\cite[VI$_\text{B}$, Prop. 9.2]{SGA3I}).

Pour une extension de corps $K/k$ et pour $V$ une $k$-vari\'et\'e, on note $V_K:=V\times_k K$ la $K$-vari\'et\'e obtenue par changement de base.
Pour toute vari\'et\'e $V$, on note $H^i(V,\cdot)$ les groupes de cohomologie \'etale et $H^i(k,\cdot)$ les groupes de cohomologie galoisienne 
classiques (qui co\"\i ncident avec les groupes de cohomologie \'etale pour $V=\spec k$). Pour $i=1$, on peut aussi d\'efinir des ensembles de
cohomologie non ab\'elienne, lesquels seront toujours not\'es $H^1(V,\cdot)$ et $H^1(k,\cdot)$ comme dans le cas ab\'elien (ces deux ensembles
co\"\i ncident aussi pour $V=\spec k$). Si l'on \'ecrit $\gm$ pour le groupe multiplicatif, le groupe de Brauer $\br V$ est d\'efini comme le groupe
$H^2(V,\gm)$. Pour $n> 1$ premier \`a $p$, on note $\mu_n$ le groupe alg\'ebrique (fini et
lisse) des racines $n$-i\`emes de l'unit\'e. La $n$-torsion du groupe de Brauer $\br V$ est alors donn\'ee par un quotient de $H^2(V,\mu_n)$. On note
aussi $\mu$ la limite inductive des $\mu_n$ pour $n\in\n$ et l'on note donc $\mu\{p'\}$ le groupe des racines de l'unit\'e d'ordre premier \`a $p$, cf.
ci-dessous.

On utilise la lettre $\ell$ pour des nombres premiers diff\'erents de la caract\'eristique $p$. Pour un groupe ab\'elien $A$, on note $A\{\ell\}$ les \'el\'ements
de torsion $\ell$-primaire de $A$ et $A\{p'\}$ les \'el\'ements de torsion premi\`ere \`a $p$. Pour la d\'efinition du groupe de Brauer non ramifi\'e $\brnr V$
d'une $k$-vari\'et\'e $V$, on renvoie \`a \cite{ColliotSantaBarbara} ou encore \`a \cite{ColliotSansucChili}. On se limite \`a rappeler que, lorsque l'on
dispose d'une compactification lisse $X$ de $V$, on a l'\'egalit\'e $\brnr V\{p'\}=\br X\{p'\}$, o\`u $p$ correspond \`a la caract\'eristique du corps $k$,
(ce sont des applications du th\'eor\`eme de puret\'e de Grothendieck, cf. \cite[Prop. 4.2.3]{ColliotSantaBarbara}).
Le groupe de Brauer alg\'ebrique est d\'efini comme
\[\bral V:=\frac{\ker[\br V\to\br V_{\bar k}]}{\im[\br k\to\br V]}.\]
On note aussi $\brun V:=\ker[\br V\to\br V_{\bar k}]$ d'o\`u $\bral V=\brun V/\br k$ lorsqu'on note par abus $\br k$ sa propre image dans $\br V$.
Le groupe de Brauer non ramifi\'e alg\'ebrique $\brnral V$ est alors tout simplement l'image de $\brnr V\cap \brun V$ dans $\bral V$ par la
projection naturelle.

On fait enfin les conventions de notation suivantes. Pour $G$ un $k$-groupe fini et pour tout groupe profini $\Delta$ agissant sur $G(\bar k)$ (en particulier
pour $\Gamma_k$ ou $\Gamma_{k_v}$ pour $v$ une place de $k$ si $k$ est un corps global), on note
\[H^i(\Delta,G):=H^i(\Delta,G(\bar k))\quad\text{et}\quad \hom(\Delta,G):=\hom_{\mathrm{cont}}(\Delta,G(\bar k)).\]
De plus, la notation $b\in G$ veut toujours dire $b\in G(\bar k)$. On rappelle aussi que les groupes $\Sha^i_{\mathrm{cyc}}(k,G)$ pour $G$ ab\'elien
sont d\'efinis comme les \'el\'ements $\alpha\in H^i(k,G)$ tels que, pour tout sous-groupe
pro-cyclique $\gamma$ de $\Gamma_k$, on a que la restriction $\alpha_\gamma\in H^i(\gamma,G)$ de $\alpha$ est triviale.\footnote{Pour un corps global on peut aussi
d\'efinir les groupes (ou ensembles) $\Sha^i_\omega(k,G)$, lesquels sont d\'efinis comme les \'el\'ements $\alpha\in H^i(k,G)$
dont l'image dans $H^i(k_v,G)$ est triviale pour presque toute place $v\in\Omega_k$ (i.e. pour toute place sauf un nombre fini). Le th\'eor\`eme de \v Cebotarev
permet de d\'emontrer facilement que les groupes (resp. ensembles) $\Sha^i_\omega$ et $\Sha^i_{\mathrm{cyc}}$ co\"\i ncident, cf. \cite[\S 1]{Sansuc81}.}

\section{Groupe de Brauer non ramifi\'e sur un corps fini}\label{section corps fini}
Comme il a \'et\'e dit dans l'introduction, dans \cite{GLABrnr}, on a utilis\'e un r\'esultat de Gabber qui pr\'ecise le th\'eor\`eme de de Jong
sur les alt\'erations (cf. \cite{ILO}, Expos\'e X, Theorem 2.1) pour \'etablir une caract\'erisation du groupe de Brauer non ramifi\'e $\brnr V$
d'une vari\'et\'e $V$ lisse et g\'eom\'etriquement int\`egre sur un corps global de caract\'eristique positive (cf. \cite{GLABrnr}, Th\'eor\`eme 2.1).
En appliquant les m\^emes techniques, on peut obtenir un r\'esultat analogue pour les corps finis :

\begin{thm}\label{theoreme sans compactification lisse corps fini}
Soient $k=\f_q$ avec $q$ une puissance de $p$ et $V$ une $k$-vari\'et\'e lisse et g\'eom\'etriquement int\`egre. Soit $\alpha\in\br V$
un \'el\'ement d'ordre $n$ premier \`a $p$. Alors $\alpha$ appartient \`a $\brnr V$ si et seulement si, pour toute extension finie $L$ de $k$,
l'application $V(L((t)))\to\br L((t))$ induite par $\alpha$ est nulle.
\end{thm}

L'un des sens de cette proposition, \`a savoir que tout \'el\'ement $\alpha\in\brnr V$ v\'erifie la propri\'et\'e d'annulation ci-dessus, avait d\'ej\`a
\'et\'e d\'emontr\'e par Borovoi, Demarche et Harari (cf. \cite[Proposition 4.3]{BDH}) via le changement de base de $k=\f_q$ \`a $k[t]$ (l'argument est,
outre l'application du r\'esultat de Gabber, totalement analogue au raisonnement classique rappel\'e par exemple dans l'introduction de \cite{GLABrnr}).
Une analogie similaire avec la preuve de \cite[Th\'eor\`eme 2.1]{GLABrnr} nous permet maintenant de d\'emontrer le sens inverse.

\begin{proof}
Puisque le groupe de Brauer est de torsion, on sait qu'il est la somme directe de ses composantes $\ell$-primaires avec $\ell$ un nombre premier.
Il suffit alors de d\'emontrer l'\'enonc\'e pour un \'el\'ement $\alpha\in\br V\{\ell\}$ pour un nombre premier $\ell\neq p$ donn\'e.\\

Le th\'eor\`eme de Nagata (cf. \cite{ConradNagata}, \cite{DeligneNagata}) nous dit qu'il existe une $k$-compactification $X$ de $V$ qui n'est pas forc\'ement lisse.
Le th\'eor\`eme de Gabber nous dit alors qu'il existe une extension finie de corps $k'/k$ de degr\'e premier \`a $\ell$ et une $\ell'$-alt\'eration
$X'\to X_{k'}$, i.e. un morphisme propre, surjectif et g\'en\'eriquement fini de degr\'e premier \`a $\ell$, avec $X'$ une $k'$-vari\'et\'e lisse. On en
d\'eduit que $X'$ est propre et lisse sur $k'$. De plus, on peut supposer que $X'$ est g\'eom\'etriquement int\`egre. En effet, quitte \`a prendre une
composante connexe de $X'$ dominant $X$ (bien choisie pour que le degr\'e g\'en\'erique reste premier \`a $\ell$), on peut supposer qu'elle est connexe.
Il suffit alors de noter que la factorisation de Stein assure l'existence
d'une factorisation $X'\to \spec k'' \to \spec k'$ telle que $X'$ est g\'eom\'etriquement connexe sur $k''$, donc quitte \`a changer $k'$ par $k''$
(extension qui reste de degr\'e premier \`a $\ell$ car $[k'':k']$ divise $[k'(X'):k'(X)]$), on a que $X'$ est g\'eom\'etriquement int\`egre sur $k'$ car elle
est lisse et g\'eom\'etriquement connexe. On a alors le diagramme commutatif suivant, o\`u les carr\'es sont cart\'esiens :
\begin{equation}\label{equation diagramme Gabber 1}
\xymatrix{
V' \ar[r] \ar[d] & X' \ar[d] \\
V_{k'} \ar[r] \ar[d] & X_{k'} \ar[d] \\
V \ar[r] & X.
}
\end{equation}
En notant alors $K=k(t)$, $\O_K=k[t]$, $K'=k'(t)$ et $\O_{K'}=k'[t]$, le diagramme \eqref{equation diagramme Gabber 1} induit par changement de base
le diagramme commutatif suivant
\[\xymatrix{
\cal V' \ar[r] \ar[d] & \cal X' \ar[d] \\
\cal V_{\O_{K'}} \ar[r] \ar[d] & \cal X_{\O_{K'}} \ar[d] \\
\cal V \ar[r] & \cal X
}\]
o\`u $\cal V:=V\times_{k}\O_K$, $\cal X:=X\times_k \O_K$, $\cal V':=V'\times_{k'}\O_{K'}$ et $\cal X':=X'\times_{k'}\O_{K'}$. En particulier, $\cal X'$
est propre et lisse sur $\O_{K'}$.\\

Supposons que $\alpha\in\brnr V$. Dans ce cas, la proposition 4.3 de \cite{BDH} nous dit pr\'ecis\'ement que l'application $V(k((t)))\to\br k((t))$
induite par $\alpha$ est nulle. Pour avoir le m\^eme r\'esultat pour toute extension finie, il suffit de consid\'erer la vari\'et\'e $V_L=V\times_k L$
et d'y appliquer la m\^eme proposition. En effet, le groupe de Brauer non ramifi\'e est fonctoriel par rapport aux morphismes dominants (cela d\'ecoule
facilement de la d\'efinition de $\brnr V=\brnr(k(V)/k)$ et de \cite[Proposition 3.3.1 et \S 3.6]{ColliotSantaBarbara}, voir aussi
\cite[Lemma 5.5]{ColliotSansucChili} pour une version plus explicite en caract\'eristique 0). On peut alors
regarder la restriction $\alpha_L\in\br V_L$ de $\alpha$ comme un \'el\'ement de $\brnr V_L$. Il est \'evident alors que pour tout $L((t))$-point
$P$ de $V$ (qui est en m\^eme temps un point de $V_L$) on a $\alpha(P)=\alpha_L(P)=0$.

Supposons donc que $\alpha\not\in\brnr V$. Il existe alors un anneau de valuation discr\`ete $A$ de corps de fractions $k(V)$ et corps r\'esiduel
$\kappa\supset k$ tel que l'image de $\alpha$ par l'application r\'esidu
\[H^2(k(V),\mu_{\ell^m})\to H^1(\kappa,\z/\ell^m\z),\]
est non nulle. Soit $B=A(t)$ : c'est un anneau de valuation discr\`ete de corps de fractions $K(V)$ et corps r\'esiduel $\kappa(t)$ (la valuation de $t$ \'etant
bien \'evidemment prise comme \'egale \`a $0$) et il est \'evident que l'indice de ramification de l'extension $B/A$ est \'egal \`a $1$. En effet, on rappelle
que si l'on note $v_B$ la valuation associ\'ee \`a $B$, cet indice est d\'efini comme l'indice du groupe $v_B(k(V)^*)$ dans $v_B(K(V)^*)\cong\z$. Le corollaire
\ref{corollaire extensions avd}, d\'emontr\'e dans l'appendice, nous dit qu'il
existe un anneau de valuation discr\`ete $B'\supset B$ de corps de fractions $K'(V')$ et corps r\'esiduel $\kappa'$ tel que, si l'on note
$f=[\kappa':\kappa(t)]$ et $e$ l'indice de ramification de $B'/B$, on a $(ef,\ell)=1$. De plus, $B'$ contient la fermeture int\'egrale de $B$
dans $K'(V')$ et $K\subset B$, donc on a $\kappa'\supset K'$. Alors, d'apr\`es \cite[Proposition 3.3.1]{ColliotSantaBarbara}, on a le diagramme commutatif
\[\xymatrix{
H^2(k(V),\mu_{\ell^m}) \ar[r] \ar[d]^{\inf} & H^1(\kappa,\z/\ell^m\z) \ar[d]^{\inf}\\
H^2(K(V),\mu_{\ell^m}) \ar[r] \ar[d]^{\res} & H^1(\kappa(t),\z/\ell^m\z) \ar[d]^{e\cdot\res} \\
H^2(K'(V'),\mu_{\ell^m}) \ar[r] & H^1(\kappa',\z/\ell^m\z). \\
}\]
On remarque que les fl\`eches verticales en haut sont bien des applications d'inflation. En effet, on a $K(V)=k(V)(t)$, d'o\`u $\Gamma_{k(V)}$ s'identifie avec
le groupe $\Gamma_{\bar k(V)(t)/K(V)}$, lequel correspond \`a un quotient de $\Gamma_{K(V)}$. On en a de m\^eme pour $\kappa$ et $\kappa(t)$.

Soit alors $\kappa_1$ la sous-extension s\'eparable maximale de $\kappa'/\kappa(t)$. On a alors la factorisation suivante des fl\`eches verticales \`a droite :
\[H^1(\kappa,\z/\ell^m\z) \xrightarrow{\inf} H^1(\kappa(t),\z/\ell^m\z)\xrightarrow{e\cdot\res} H^1(\kappa_1,\z/\ell^m\z)\xrightarrow{\res} H^1(\kappa',\z/\ell^m\z).\]
Les applications d'inflation \'etant toujours injectives au niveau du $H^1$, on sait que la fl\`eche de gauche est injective. Puis, le corollaire
\ref{corollaire isomorphisme extensions purement inseparables}, aussi d\'emontr\'e dans l'appendice, nous dit que la fl\`eche de droite est un isomorphisme.
Enfin, pour la fl\`eche du milieu on a que $e[\kappa_1:\kappa]$ divise $ef$ et alors le morphisme est injectif par l'argument classique de restriction-corestriction
(lequel a bien un sens vu que $\kappa_1/\kappa$ est s\'eparable). En particulier, on voit que l'image de $\alpha$ dans
$H^1(\kappa',\z/\ell^m\z)$ est non nulle et alors que l'image $\alpha'\in \br V'$ de $\alpha$ n'est pas dans $\brnr V'=\br X'$.\\

Enfin, d'apr\`es \cite[Th\'eor\`eme 1.2]{GLABrnr}, il existe un ensemble infini $I'$ de places $w$ de $K'$ telles qu'il existe un $K'_w$-point $P'_w$
de $V'$ tel que $\alpha'(P'_w)\neq 0$. En prenant une quelconque de ces places, on a $K'_w\cong L((t'))$ pour une certaine extension finie $L/k$, d'o\`u
l'on voit que, si $Q$ est le $L((t'))$-point de $V$ induit par $P'_w$, alors $\alpha(Q)=\alpha'(P'_w)\neq 0\in\br L((t'))$, ce qui conclut.
\end{proof}

\section{Formules pour le groupe $\brnral$ sur les espaces homog\`enes}\label{section formules}\numberwithin{equation}{subsection}
On pr\'esente dans cette section, \`a l'aide des r\'esultats pr\'ec\'edents, diff\'erentes formules pour calculer
le groupe $\brnral V$, o\`u $V$ est un espace homog\`ene \`a stabilisateur fini sur plusieurs corps de base.\\

Soit $k$ un corps dont on note $p\geq 0$ la caract\'eristique et $\bar k$ sa cl\^oture s\'eparable. On note $\Gamma_k=\gal(\bar k/k)$.
On rappelle que, pour une $k$-vari\'et\'e $V$, la partie alg\'ebrique du groupe de Brauer $\br V$ est d\'efinie comme
\[\bral V:=\frac{\ker[\br V\to\br V_{\bar k}]}{\im[\br k\to\br V]}=\brun V/\br k.\]
Si de plus la vari\'et\'e v\'erifie $V(k)\neq\emptyset$ et $\bar k[V]^*/\bar k^*=1$, alors ce groupe admet une description en termes du groupe de Picard
g\'eom\'etrique $\pic V_{\bar k}$. En effet, le lemme 6.3 de \cite{Sansuc81} nous dit que l'on a un isomorphisme
\begin{equation}\label{equation brnr pour V}
\bral V\xrightarrow{\sim} H^1(k,\pic V_{\bar{ k}}).
\end{equation}
Supposons maintenant qu'il existe une $k$-compactification lisse $X$ de $V$. On a un isomorphisme similaire pour la vari\'et\'e $X$. En effet, il
suffit de noter que $X$ est propre et g\'eom\'etriquement int\`egre pour avoir $\bar k[X]^*/\bar k^*=1$, et il est \'evident que $X(k)\neq\emptyset$
d\`es que $V(k)\neq\emptyset$, donc on a aussi
\begin{equation}\label{equation brnr pour X}
\bral X\xrightarrow{\sim} H^1(k,\pic X_{\bar{ k}}).
\end{equation}
On rappelle que, d'apr\`es les th\'eor\`emes de puret\'e de Grothendieck, le groupe $\bral X\{p'\}$ correspond \`a la partie non ramifi\'ee (et
de torsion premi\`ere \`a $p$) du groupe de Brauer alg\'ebrique de $V$ (cf. \cite[Proposition 4.2.3]{ColliotSantaBarbara}), que l'on notera
$\brnral V\{p'\}$.\\

Soit maintenant $G$ un $k$-groupe fini lisse que l'on plonge dans $G'$ un $k$-groupe alg\'ebrique lin\'eaire semi-simple simplement connexe,
par exemple $G'=\sln$. Soit $V:=G\backslash G'$ la $k$-vari\'et\'e quotient.
Le groupe $G'$ \'etant semi-simple, on a que $\bar k[G']^*/\bar k^*=0$ (c'est le lemme de Rosenlicht, cf. \cite[Lemme 6.5]{Sansuc81}). Ceci
nous dit que $\bar k[V]^*/\bar k^*=0$, puisque toute fonction inversible sur $V_{\bar k}$ donne une fonction inversible sur $G'_{\bar k}$. On en d\'eduit
que l'on a la description du groupe $\bral V$ donn\'ee par l'isomorphisme \eqref{equation brnr pour V} ci-dessus.

Il est int\'eressant de remarquer que cette description ne d\'epend point du plongement $G\hookrightarrow G'$ choisi. En effet, dans cette situation,
la cohomologie de \v Cech non ab\'elienne associ\'ee au torseur $G'_{\bar k}\to V_{\bar k}$ nous donne une suite exacte d'ensembles point\'es
(cf. \cite[Lemma 2.2.2]{Skor})
\begin{equation}\label{equation suite cech non abelien}
0\to\check H^1(G'_{\bar k}/V_{\bar k},\gm)\to\check H^1(V_{\bar k},\gm)\to\check H^0(G'_{\bar k}/V_{\bar k},\check{\mathcal{H}}^1(\gm)),
\end{equation}
o\`u $\check{\mathcal{H}}^1(\gm)$ correspond au pr\'efaisceau (\'etale) $U\mapsto \check H^1(U,\gm)$. Le m\^eme lemme nous dit que
$\check H^1(G'_{\bar k}/V_{\bar k},\gm)$ correspond \`a l'ensemble des classes d'\'equivalence de morphismes
\[f:G'\times G\to\gm\]
satisfaisant la condition de cocycle $f(y,s)f(ys,s')=f(y,ss')$ ; $f$ \'etant \'equivalent \`a $f'$ si
\[f'(y,s)=g(y)f(y,s)g(ys)^{-1},\]
pour un morphisme $g:G'_{\bar k}\to\gm$. Or, puisque $\bar k[G']^*/\bar k^*=0$, il est facile de voir que $f(y,s)$ ne d\'epend que de $s$ et qu'alors
$\check H^1(G'_{\bar k}/V_{\bar k},\gm)$ s'identifie avec l'ensemble des morphismes $g:G\to\gm$, i.e. au groupe des caract\`eres de $G$, que l'on
notera $M$ dans la suite. On rappelle que c'est le dual de Cartier de $G^\ab$, l'ab\'elianis\'e de $G$.\\
De plus, en rappelant que $\check H^1(V_{\bar k},\gm)=H^1(V_{\bar k},\gm)$ (cf. \cite[III, Corollary 2.10]{Milne}), on peut r\'e\'ecrire la suite
exacte \eqref{equation suite cech non abelien} comme
\[0\to M \to \pic V_{\bar k} \to \pic G'_{\bar k}.\]
Or, $G'$ \'etant en plus simplement connexe, on a $\pic G'_{\bar k}=0$ (cf. \cite[Lemme 6.9(iv)]{Sansuc81}).  On a alors trouv\'e un isomorphisme
\begin{equation}\label{equation type du V-torseur G'}
\lambda:M\xrightarrow{\sim}\pic V_{\bar k}.
\end{equation}
Pour v\'erifier qu'il s'agit bien d'un morphisme de groupes (et de $\Gamma_k$-modules), il suffit de regarder la m\^eme situation pour le
$G^{\ab}$-torseur $Z:=G^\der\backslash G'\to V$, o\`u $G^\der$ est le sous-groupe d\'eriv\'e de $G$. On y retrouve la m\^eme application, mais dans le cas
de la cohomologie de \v Cech ab\'elienne (cf. \cite[2.2]{Skor}). On remarque en plus que cet isomorphisme correspond au \emph{type} du torseur
$Z\to V$. En effet, ceci d\'ecoule de la d\'efinition du type dans \cite[Definition 2.3.2]{Skor} et de ce qui est remarqu\'e juste apr\`es.
On en d\'eduit que $Z\to V$ est un torseur universel.\\

On voit alors que $\bral V=H^1(k,\pic V_{\bar k})$ ne d\'epend que de $G$, et m\^eme plus pr\'ecis\'ement de son ab\'elianis\'e.
De plus, puisque dans tout ce texte on ne consid\'erera que des groupes $G$ d'ordre premier \`a $p$ (qui sont d'ailleurs toujours lisses),
$M$ et par suite $\bral V$ n'auront pas non plus de $p$-torsion, ce qui nous permet d'utiliser la notation $\brnral V$ au lieu de
$\brnral V\{p'\}$.

\subsection{Une formule cohomologique sur un corps fini}\label{section formule cohomologique}
Soit maintenant $k$ un corps fini de caract\'eristique $p>0$. On veut \'etudier les espaces homog\`enes $V=G\backslash G'$ comme ci-dessus sur ce
corps de base particulier. Comme on vient de l'anoncer, on supposera d\'esormais que l'ordre de $G$ est premier \`a $p$. Sous cette hypoth\`ese et
gr\^ace au th\'eor\`eme \ref{theoreme sans compactification lisse corps fini} on peut donner une formule cohomologique permettant de calculer le
groupe $\brnral V$.\\

Rappelons que, lorsque l'ordre de $G$ est premier \`a $p$, le cup-produit nous donne l'accouplement parfait suivant
(cf. \cite[Theorem 7.2.6]{NSW})
\[H^1(k((t)),M)\times H^1(k((t)),G^\ab)\to H^2(k((t)),\gm)=\br k((t))\cong\q/\z.\]

En utilisant cet accouplement, le groupe $\brnral V$ admet la description cohomologique suivante :

\begin{thm}\label{theoreme cohomologique}
Soit $G$ un $k$-groupe fini d'ordre premier \`a $q$ plong\'e dans $G'$ semi-simple simplement connexe et soit $V=G\backslash G'$. Soit $M$ le groupe des
caract\`eres de $G$. En identifiant $\bral V$ avec $H^1(k,M)$ comme ci-dessus, le groupe de Brauer
non ramifi\'e alg\'ebrique $\brnral V$ de $V$ est donn\'e par les \'el\'ements $\alpha\in H^1(k,M)$ v\'erifiant la propri\'et\'e
suivante :

Pour toute extension finie $k'$ de $k$, l'image $\alpha'$ de $\alpha$ dans $H^1(k'((t)),M)$ est orthogonale au sous-ensemble
$\im[H^1(k'((t)),G)\to H^1(k'((t)),G^\ab)]$.
\end{thm}

Ce r\'esultat d\'ecoule de fa\c con \'evidente de la proposition suivante et du th\'eor\`eme \ref{theoreme sans compactification lisse corps fini}.

\begin{pro}\label{proposition lemme cohomologique}
Soient $K$ un corps local de caract\'eristique $p\geq 0$, $G$ un $K$-groupe fini d'ordre premier \`a $p$ plong\'e dans $G'$semi-simple simplement
connexe et soit $V=G\backslash G'$. Soit $M$ le groupe des caract\`eres de $G$. Soit enfin $\alpha\in\bral V$.
En identifiant $\bral V$ avec $H^1(K,M)$, il existe une pr\'eimage $\alpha_1\in\brun V$ de $\alpha$ telle que l'application $V(K)\to \br K$ induite
par $\alpha_1$ est triviale si et seulement si $\alpha$ est orthogonale au sous-ensemble $\im[H^1(K,G)\to H^1(K,G^\ab)]$.
\end{pro}

Pour d\'emontrer cet \'enonc\'e, on utilise la proposition suivante.

\begin{pro}\label{proposition Skor}
Soient $K$ un corps quelconque, $V$ une $K$-vari\'et\'e, $\pi:V\to K$ le morphisme structural, $M$ un $\Gamma_K$-module de type fini dont la
torsion est d'ordre premier \`a la caract\'eristique de $K$, $S$ son $K$-groupe dual (qui est donc de type multiplicatif et lisse sur $K$),
$\lambda\in\hom_K(M,\pic V_{\bar K})$ et $r:\brun V\to H^1(K,\pic V_{\bar K})$ le morphisme issu de la suite spectrale de Hochschild-Serre
\[H^p(\Gamma_K,H^q(V_{\bar K},\gm))\Rightarrow H^{p+q}(V,\gm).\]
Supposons qu'il existe un $V$-torseur $T$ sous $S$ de type $\lambda$. Alors pour tout $\alpha\in H^1(K,M)$ on a
\[r(\pi^*(\alpha)\cup [T])=\lambda_*(\alpha)\quad\in H^1(K,\pic V_{\bar K}).\]
\end{pro}

\begin{proof}
Cette proposition correspond au th\'eor\`eme 4.1.1 de \cite{Skor}, \`a cela pr\`es qu'il consid\`ere un corps $K$ de caract\'eristique $0$, ce qui lui
donne imm\'ediatement la lissit\'e de $S$ sans aucune hypoth\`ese sur l'ordre de la torsion de $M$. Suivant alors la preuve dans \cite{Skor} on note
que, puisque $S$ est lisse sur $K$, on a $\ext_V^n(\pi^*M,\gm)=H^n(V,S)$ (cf. \cite[Lemma 2.3.7]{Skor}).
D'autre part, on a des isomorphismes canoniques $\ext_V^n(\z,\pi^*M)=H^n(V,\pi^*M)$ et des accouplements de Yoneda
\[H^1(V,\pi^*M)\times\ext_V^n(\pi^*M,\gm)\to H^{n+1}(V,\gm),\]
et
\[H^1(K,M)\times\ext_K^n(M,\pic V_{\bar K})\to H^{n+1}(K,\pic V_{\bar K}).\]
On d\'emontre d'abord que la fl\`eche
\[\ext_V^1(\pi^*M,\gm)\xrightarrow{\pi^*(\alpha)\cup\cdot} H^2(V,\gm)=\br V,\]
est \`a valeurs dans $\brun V$, puis que l'on a un diagramme commutatif
\[\xymatrix{
H^1(V,S) \ar@{}[r] |{\hspace{-0.5cm}=} & \ext_V^1(\pi^*M,\gm) \ar[d]_{\pi^*(\alpha)\cup\cdot} \ar[r]^{\texttt{type}} & \hom_K(M,\pic V_{\bar K})
\ar[d]_{\alpha\cup\cdot} \\
& \brun V \ar[r]^{r} & H^1(K,\pic V_{\bar K}).
}\]
En notant alors que la deuxi\`eme fl\`eche verticale envoie $\lambda$ en $\lambda_*(\alpha)$, la proposition est d\'emontr\'ee. Pour la
d\'emonstration des deux affirmations manquantes, on renvoie \`a la preuve dans \cite[pages 64-68]{Skor}, car elle est assez technique et la
caract\'eristique du corps $K$ n'y intervient pas.
\end{proof}

\begin{proof}[D\'emonstration de la Proposition \ref{proposition lemme cohomologique}]
Remarquons d'abord que l'isomorphisme \eqref{equation brnr pour V} vient pr\'ecis\'ement du morphisme $r:\brun V\to H^1(K,\pic V_{\bar K})$
de la suite spectrale de Hochschild-Serre, cf. \cite[Lemme 6.3]{Sansuc81}. D'autre part, on sait que $Z\to V$ est universel, i.e. le morphisme
\eqref{equation type du V-torseur G'} $\lambda:M\to\pic V_{\bar K}$ est aussi un isomorphisme, ce qui fait que l'on a
\[\xymatrix{\brun V \ar@{->>}[r] \ar@/_1pc/[rr]_{r} & \bral V \ar[r]^{\hspace{-7mm}\sim} & H^1(K,\pic V_{\bar K})
\ar@{<-}[r]^{\hspace{5mm}\sim}_{\hspace{5mm}\lambda_*}& H^1(K,M). }\]
La proposition \ref{proposition Skor} nous dit alors que l'image de l'\'el\'ement $\pi^*(\alpha)\cup[Z]\in\brun V$ dans $\bral V$ correspond
\`a $\alpha\in H^1(K,M)$.\\

Pour un $K$-point $P:\spec K\to V$, notons $[Z](P)\in H^1(K,G^\ab)$ la classe du $K$-torseur $Z\times_V K$ sous $G^\ab$. De m\^eme,
notons $[G'](P)$ la classe du $K$-torseur $G'\times_V K$ sous $G$. On a le diagramme commutatif suivant
\[\xymatrix{
([G'],P) \ar@{}[r] |{\hspace{-8mm}\in} \ar@{|->}[d] & H^1(V,G)\times V(K) \ar[r] \ar[d] & H^1(K,G) \ar[d] \ar@{}[r] |{\hspace{2mm}\ni} &
\hspace{-4mm}[G'](P) \ar@{|->}[d] \\
([Z],P) \ar@{}[r] |{\hspace{-8mm}\in} \ar@{|->}[d] & H^1(V,G^\ab)\times V(K) \ar[r] \ar[d]_{\pi^*(\alpha)\cup\cdot} & H^1(K,G^\ab) \ar[d]^{\alpha\cup\cdot}
\ar@{}[r] |{\hspace{2mm}\ni} & \hspace{-2mm}[Z](P) \ar@{|->}[d] \\
(\pi^*(\alpha)\cup [Z],P) \ar@{}[r] |{\in} & \brun V\times V(K) \ar[r] & \br K \ar@{}[r] |{\hspace{-4mm}\ni} & \alpha\cup [Z](P). \\
}\]
Si l'on consid\`ere alors la suite exacte de cohomologie galoisienne non ab\'elienne suivante (voir par exemple
\cite[I, \S5.4, Proposition 36]{SerreCohGal})
\[G'(K)\to V(K) \xrightarrow{\delta} H^1(K,G) \to H^1(K,G'),\]
on sait par la d\'efinition de $\delta$ que $\delta(P)=[G'](P)$. Comme $G'$ est simplement connexe, on sait que $H^1(K,G')=0$ (cf.
\cite[4.7]{BruhatTits}), ce qui nous dit que l'application $\delta$ est surjective. On voit alors que l'ensemble des $[Z](P)$ pour $P$
parcourant $V(K)$ est $\im[H^1(K,G)\to H^1(K,G^\ab)]$ tout entier. On en d\'eduit que, si $\alpha$ est orthogonale \`a
$\im[H^1(K,G)\to H^1(K,G^\ab)]$, alors l'application $V(K)\to \br K$ induite par $\pi^*(\alpha)\cup [Z]\in\brun V$ est triviale. Dans le sens
inverse, si $\alpha$ n'est pas orthogonale \`a $\im[H^1(K,G)\to H^1(K,G^\ab)]$, alors l'application
induite par $\pi^*(\alpha)\cup [Z]$ n'est clairement pas triviale, donc en particulier pas constante (car la classe triviale appartient toujours \`a
$\im[H^1(K,G)\to H^1(K,G^\ab)]$ et donc l'\'el\'ement trivial de $\br K$ est toujours atteint). On voit aussit\^ot que celle induite
par une autre pr\'eimage $\alpha_1\in\brun V$ de $\alpha$ ne l'est pas non plus (ni constante, ni triviale), car il s'agit d'une translation de
la premi\`ere par $\beta\in\br K$, o\`u $\beta$ est l'\'el\'ement tel que $\alpha_1=\pi^*(\alpha)\cup [Z]+\beta$.
\end{proof}

\begin{rem}
Vu l'\'enonc\'e de la proposition \ref{proposition Skor}, il est facile de voir que le th\'eor\`eme \ref{theoreme cohomologique} s'\'etend \`a
un stabilisateur \emph{quelconque} \`a composante connexe r\'eductive, \`a cela pr\`es que l'on doit toujours demander que la torsion du groupe $G^\ab$
soit d'ordre premier \`a $q$. En effet, la preuve de la proposition \ref{proposition lemme cohomologique} ne fait pas du tout mention de la finitude de $G$ et,
dans ce cas plus g\'en\'eral, le groupe $G^\ab$ est un groupe de type multiplicatif lisse dont le dual n'a pas de $p$-torsion, donc la proposition
\ref{proposition Skor} reste valable.
\end{rem}

\subsection{Une formule alg\'ebrique sur un corps fini}\label{section formule algebrique}
On garde les notations de la section pr\'ec\'edente, donc $k$ est toujours un corps fini de cardinal $q$. On rappelle aussi les conventions
faites dans la section \ref{section notations} : pour tout groupe profini $\Delta$ agissant sur $G(\bar k)$ (par exemple $\Delta=\Gamma_{k((t))}$) on note
\[H^i(\Delta,G):=H^i(\Delta,G(\bar k))\quad\text{et}\quad \hom(\Delta,G):=\hom_{\mathrm{cont}}(\Delta,G(\bar k)).\]
De plus, la notation $b\in G$ veut toujours dire $b\in G(\bar k)$. On fait de m\^eme pour $G^\ab,$ $M$, etc.

Les th\'eor\`emes \ref{theoreme sans compactification lisse corps fini} et \ref{theoreme cohomologique} nous donnent des formules pour le groupe
$\brnral V$ correspondant \`a la $k$-vari\'et\'e $V=G\backslash G'$ qui regardent ce qui se passe au niveau des corps locaux $k'((t))$ avec $k'$ une extension
finie de $k$ et profitent de leurs propri\'et\'es arithm\'etiques. Dans cette section on en d\'eduit une formule alg\'ebrique qui ne d\'epend que du corps
de base $k$ en d\'eformant peu \`a peu la formule cohomologique (th\'eor\`eme \ref{theoreme cohomologique}).

\begin{defi}
Soit $G$ un groupe fini de cardinal premier \`a $q$. Soit $s\in\Gamma_k$ l'\'el\'ement correspondant au $q$-Frobenius (i.e. l'application
$\bar k\to \bar k:x\mapsto x^q$). On d\'efinit une application bijective $\varphi_q:G(\bar k)\to G(\bar k)$ par la formule
\[\varphi_q(b)=\asd{s^{-1}}{}{q}{}{b}.\]
L'application induite par $\varphi_q$ sur $G^\ab(\bar k)$ est un automorphisme que l'on note toujours $\varphi_q$ par abus. Pour $n\in\n$, on
note $\varphi_q^n$ le $n$-i\`eme it\'er\'e de $\varphi_q$.

Pour tout $b\in G$, on d\'efinit $n_{b}$ comme le plus petit entier $n_b>0$ tel que $\varphi_q^{n_b}(b)$ soit conjugu\'e \`a $b$. On remarque qu'un tel
$n_b$ existe toujours d'apr\`es la finitude de $G(\bar k)$. On d\'efinit alors l'application \emph{$q$-norme} $N_q:G(\bar k)\to G^\ab(\bar k)$ comme le
produit
\[N_q(b)=\prod_{i=0}^{n_{b}-1}\overline{\varphi_q^i(b)}=\prod_{i=0}^{n_{b}-1}\varphi_q^i(\bar b),\]
o\`u ``$\overline{\phantom{(b)}}$'' repr\'esente la projection naturelle de $G(\bar k)$ sur $G^\ab(\bar k)$.
\end{defi}

Avec ces d\'efinitions, on peut \'enoncer la formule alg\'ebrique de la fa\c con suivante.

\begin{thm}\label{theoreme formule sur fq}
Soit $G$ un $k$-groupe fini d'ordre premier \`a $q$ plong\'e dans $G'$ semi-simple simplement connexe et soit $V=G\backslash G'$. Soit $M$ le groupe des
caract\`eres de $G$. En identifiant $\bral V$ avec $H^1(k,M)$, le groupe de Brauer non ramifi\'e alg\'ebrique $\brnral V$ de $V$ est donn\'e par les
\'el\'ements $\alpha\in H^1(k,M)$ tels que pour $a\in Z^1(k,M)$ un cocycle (quelconque) repr\'esentant $\alpha$, on a
\[a_{s}(N_q(b))=1\quad \forall\, b\in G,\]
\end{thm}

On remarque que la formule a bien un sens. En effet, $a_s$ est un \'el\'ement de $M$, donc on peut l'\'evaluer en $N_q(b)\in G^\ab$, ce qui
donne un \'el\'ement dans $\mu\{q'\}$.

\subsubsection{D\'emonstration du th\'eor\`eme \ref{theoreme formule sur fq}}
D'apr\`es le th\'eor\`eme \ref{theoreme cohomologique}, on a int\'er\^et \`a \'etudier l'application $H^1(K_1,G)\to H^1(K_1,G^\ab)$ pour
$K_1=k((t))=\f_q((t))$, ainsi que les applications $H^1(K_n,G)\to H^1(K_n,G^\ab)$ pour $K_n=\f_{q^n}((t))$. Soit alors
$S\subset\Gamma_{K_n}\subset\Gamma_{K_1}$ le sous-groupe de ramification sauvage et consid\'erons le quotient $D_n:=\Gamma_{K_n}/S$ pour
tout $n\geq 1$. On a une suite exacte d'ensembles point\'es
\[0\to H^1(D_n,G)\to H^1(\Gamma_{K_n},G)\to H^1(S,G),\]
mais puisque $G$ est un $k$-groupe fini, on a $G(\bar k)=G(\bar K_n)$ et alors $S$ agit trivialement sur $G(\bar K_n)$, donc $H^1(S,G)$
correspond aux morphismes $S\to G$ \`a conjugaison pr\`es. On voit alors que $H^1(S,G)=0$, car $S$ est un pro-$p$-groupe et $G$ n'a
pas de $p$-sous-groupe non trivial. On a exactement la m\^eme situation pour $G^\ab$, donc on voit que l'on a $H^1(D_n,G)=H^1(K_n,G)$ et de
m\^eme pour $G^\ab$ et $M$ et on peut alors se restreindre \`a \'etudier les applications $H^1(D_n,G)\to H^1(D_n,G^\ab)$ pour $n\geq 1$.\\

On rappelle que $D_n$ est un groupe profini engendr\'e par deux g\'en\'erateurs $\sigma_n$ et $\tau$ avec la relation
$\sigma_n\tau\sigma_n^{-1}=\tau^{q^n}$ (cf. \cite[Thm. 7.5.3]{NSW}). On remarque par ailleurs que l'inclusion naturelle
$\Gamma_{K_n}\subset\Gamma_{K_1}$ donne $\sigma_n=\sigma_1^n$ et que $D_n$ correspond au groupe de Galois de l'extension mod\'erement
ramifi\'ee maximale de $K_n=\f_{q^n}((t))$. Si de plus on consid\`ere le sous-groupe ferm\'e
\[T:=\langle\tau\rangle\subset D_n\]
et les quotients
\[\Gamma_n:=D_n/T\cong\langle\sigma^n\rangle,\]
on voit que $\Gamma_n$ correspond au groupe de Galois de $\bar{\f}_{q^n}((t))/\f_{q^n}((t))$ (i.e. de l'extension non ramifi\'ee maximale de $K_n$)
lequel est isomorphe au groupe de Galois de $\f_{q^n}$. En particulier on a un isomorphisme canonique $\Gamma_1\xrightarrow{\sim}\Gamma_k$ qui
envoie la classe de $\sigma_1$ dans $\Gamma_1$ sur le Frobenius $s$ et alors on a canoniquement $H^1(k,M)\cong H^1(\Gamma_1,M)$. Il est facile
aussi de voir \`a partir de la pr\'esentation de $D_n$ que $\Gamma_n\cong\hat\z$, tandis que $T\cong\z_{p'}:=\prod_{\ell\neq p}\z_\ell$ et que
l'on a une suite exacte scind\'ee
\[1\to T\to D_n\to\Gamma_n\to 1.\]
De ce fait, on se permet de noter d\'esormais abusivement $\sigma$ tant le g\'en\'erateur de $\Gamma_1$ que l'\'el\'ement $\sigma_1\in D_1$. On a
alors que $\sigma^n$ est le g\'en\'erateur de $\Gamma_n$ pour tout $n$. Enfin, comme $G$ est d\'efini sur $k$, on sait que $\tau$ agit trivialement sur
$G(\bar k)=G(\bar K_1)$, donc en particulier on a une action de $\Gamma_1$ sur $G(\bar k)$ qui co\"\i ncide avec l'action de $\Gamma_k$, i.e. on a
$\asd{\sigma}{}{}{}{b}=\asd{s}{}{}{}{b}$ pour tout $b\in G$ et en particulier,
\[\varphi_q(b)=\asd{s^{-1}}{}{q}{}{b}=\asd{\sigma^{-1}}{}{q}{}{b},\quad\text{pour tout }b\in G.\]

\vspace{3mm}

Soit maintenant $\alpha\in H^1(k,M)=H^1(\Gamma_1,M)$. Le th\'eor\`eme \ref{theoreme cohomologique} nous dit qu'il faut \'etudier le comportement
de l'image $\alpha_n$ de $\alpha$ dans $H^1(\Gamma_n,M)\subset H^1(D_n,M)$ par rapport au groupe $H^1(D_n,G^\ab)$. En effet, d'apr\`es les
\'egalit\'es ci-dessus, on voit que l'on a un accouplement parfait
\begin{equation}\label{equation accouplement pour H1(Dn)}
H^1(D_n,M)\times H^1(D_n,G^\ab)\to H^2(D_n,\mu\{q'\})\hookrightarrow\q/\z,
\end{equation}
issu de l'accouplement local classique. Soit alors $\beta\in H^1(D_n,G^\ab)$ que l'on suppose d'abord comme provenant de $H^1(\Gamma_n,G^\ab)$.
On peut donc consid\'erer le cup-produit $\alpha_n\cup\beta\in H^2(\Gamma_n,\mu\{q'\})$. Or, $\Gamma_n$ est un groupe de dimension cohomologique 1
car il est isomorphe \`a $\hat\z$, donc on a $H^2(\Gamma_n,\mu\{q'\})=0$. La compatibilit\'e du cup-produit avec l'inflation nous dit alors que $\alpha_n$
est orthogonal \`a tout le sous-groupe $H^1(\Gamma_n,G^\ab)$ de $H^1(D_n,G^\ab)$. De plus, la suite exacte des premiers termes de la suite spectrale
de Hochschild-Serre
\[H^p(\Gamma_n,H^q(T,G^\ab))\Rightarrow H^{p+q}(D_n,G^\ab),\]
qui est
\[0\to H^1(\Gamma_n,(G^\ab)^T) \xrightarrow{\inf} H^1(D_n,G^\ab) \xrightarrow{\res} H^1(T,G^\ab)^{\Gamma_n}\to H^2(\Gamma_n,(G^\ab)^T),\]
peut \^etre r\'e\'ecrite comme
\[0\to H^1(\Gamma_n,G^\ab) \xrightarrow{\inf} H^1(D_n,G^\ab) \xrightarrow{\res} \hom(T,G^\ab)^{\Gamma_n}\to 0,\]
car $T$ agit trivialement sur $G^\ab(\bar k)$ et $\mathrm{cd}(\Gamma_n)=1$. De cette suite, et du fait que $H^1(\Gamma_n,M)$ est orthogonal
\`a $H^1(\Gamma_n,G^\ab)$, on d\'eduit que l'on a un accouplement
\begin{equation}\label{equation accouplement 1}
H^1(\Gamma_n,M)\times \hom(T,G^\ab)^{\Gamma_n}\to H^2(D_n,\mu\{q'\})\hookrightarrow\q/\z,
\end{equation}
qui nous dit que c'est l'image de $H^1(D_n,G)$ dans $\hom(T,G^\ab)^{\Gamma_n}$ que l'on doit \'etudier. On donne alors quelques d\'efinitions
qui nous aideront \`a d\'ecrire cette image.

\begin{defi}\label{definition (q,n)-relevable}
Soit $n\in\n^*$. On dit qu'un \'el\'ement $\b\in G^\ab$ est $(q,n)$-relevable s'il existe $b\in G$ relevant $\b$ tel que $\varphi_q^n(b)$ soit
conjugu\'e \`a $b$. On remarque que ceci entra\^\i ne en particulier que $\varphi_q^n(\b)=\b$.\\
On note $(G^\ab)^{\varphi_q^n}$ le sous-groupe de $G^\ab(\bar k)$ des \'el\'ements invariants par $\varphi_q^n$ et
$G^\ab_{q,n}\subset(G^\ab)^{\varphi_q^n}$ l'ensemble des \'el\'ements $(q,n)$-relevables. On remarque que $\b\in G^\ab$ est
$\varphi_q^n$-invariant si et seulement s'il v\'erifie $\asd{\sigma^n}{}{}{}{\b}=\b^{q^n}$.
\end{defi}

On a un isomorphisme \'evident $\hom(T,G^\ab)\xrightarrow{\sim}G^\ab(\bar k)$ qui \`a un morphisme lui associe l'image de $\tau$. On peut alors
v\'erifier facilement que l'image de $\hom(T,G^\ab)^{\Gamma_n}$ dans $G^\ab(\bar k)$ pour cet isomorphisme correspond au sous-groupe
$(G^\ab)^{\varphi_q^n}$. On voit alors que pour tout $n$ l'accouplement \eqref{equation accouplement 1} devient
\begin{equation}\label{equation accouplement 2}
H^1(\Gamma_n,M)\times (G^\ab)^{\varphi_q^n}\to H^2(D_n,\mu\{q'\})\hookrightarrow\q/\z.
\end{equation}
On peut alors \'enoncer un r\'esultat \'equivalent au th\'eor\`eme \ref{theoreme cohomologique} qui utilise cet accouplement.

\begin{pro}\label{proposition traduction du theoreme cohomologique}
Soit $G$ un $k$-groupe fini d'ordre premier \`a $q$ plong\'e dans $G'$ semi-simple simplement connexe et soit $V=G\backslash G'$. Soit $M$ le groupe des
caract\`eres de $G$. En identifiant $\bral V$ avec $H^1(\Gamma_1,M)$, le groupe de Brauer non ramifi\'e alg\'ebrique $\brnral V$ de $V$ est donn\'e
par les \'el\'ements $\alpha\in H^1(\Gamma_1,M)$ v\'erifiant la propri\'et\'e suivante :

Pour tout $n\geq 1$, l'image $\alpha_n$ de $\alpha$ dans $H^1(\Gamma_n,M)$ est orthogonale \`a $G^\ab_{q,n}$
pour l'accouplement \eqref{equation accouplement 2} donn\'e ci-dessus.
\end{pro}

\begin{proof}
Il est facile de voir qu'un cocycle $b\in Z^1(D_n,G)$ repr\'esentant une classe de $H^1(D_n,G)$ est d\'efini par ses images $b_{\sigma^n}$
et $b_{\tau}$ et qu'elles sont soumises \`a la seule relation
\[b_{\sigma^n\tau\sigma^{-n}}=b_{\tau^{q^n}},\]
ce qui, sachant que $\tau$ agit trivialement sur $G(\bar k)$, se traduit par
\[b_{\sigma^n}\asd{\sigma^n}{}{}{\tau}{b}b_{\sigma^n}^{-1}=b_{\tau}^{q^n}.\]
On en d\'eduit qu'un morphisme dans $\hom(T,G^\ab)$ appartient \`a l'image de $H^1(D_n,G)$ si et seulement si l'image de $\tau$ est
$(q,n)$-relevable, ce qui nous dit que $\alpha$ est orthogonale \`a $G^\ab_{q,n}$ si et seulement si elle est orthogonale \`a
$\im[H^1(D_n,G)\to \hom(T,G^\ab)^{\Gamma_n}]$ pour l'accouplement \eqref{equation accouplement 1}, ce qui est le cas si et seulement
si elle est orthogonale \`a $\im[H^1(D_n,G)\to H^1(D_n,G^\ab)]$ pour l'accouplement \eqref{equation accouplement pour H1(Dn)}.

Il suffit alors de noter que, pour toute extension finie $k'/k$ il existe $n\geq 1$ tel que $k'((t))\cong K_n$ et que l'accouplement
\eqref{equation accouplement pour H1(Dn)} provient bien de l'accouplement classique
\[H^1(K_n,M)\times H^1(K_n,G^\ab)\to H^2(K_n,\gm)=\br K_n\cong\q/\z.\]
Le th\'eor\`eme \ref{theoreme cohomologique} nous permet alors de conclure.
\end{proof}

Soient maintenant $\alpha\in H^1(\Gamma_1,M)$, $a\in Z^1(\Gamma_1,M)$ un cocycle repr\'esentant $\alpha$, $\b\in(G^\ab)^{\varphi_q^n}$ et, par
abus de notation, on note $\b$ aussi le cocycle dans $Z^1(D_n,G^\ab)$ tel que $\b_\tau=\b$ et $\b_{\sigma^n}=1$. Ce cocycle repr\'esente une
classe $\beta\in H^1(D_n,G^\ab)$ dont l'image dans $\hom(T,G^\ab)^{\Gamma_n}$ est pr\'ecis\'ement le morphisme envoyant $\tau$ en $\b$, ce qui justifie l'abus.
On veut analyser le 2-cocycle $a\cup \b\in Z^2(D_n,\mu\{q'\})$ obtenu par l'accouplement \eqref{equation accouplement 2} ci-dessus.

\begin{lem}\label{lemme calcul de a cup b}
On a la formule
\[(a\cup \b)_{\sigma^{ns_1}\tau^{t_1},\sigma^{ns_2}\tau^{t_2}}=[a_{\sigma^n}(\b)]^{q^{n(s_1+s_2)}s_1t_2}\in\mu\{q'\}.\]
\end{lem}

\begin{proof}
En tant qu'\'el\'ement de $Z^1(D_n,M)$, on sait que $a$ v\'erifie
\begin{align*}
a_{\sigma^{ns}} &=\prod_{i=0}^{s-1}\asd{\sigma^{ni}}{}{}{\sigma^n}{a}, &s\geq 1,\\
a_{\sigma^{-ns}} &=\prod_{i=1}^{s}\asd{\sigma^{-ni}}{}{-1}{\sigma^n}{a}, & s\geq 1,\\
a_{\tau^t} &=1, &t\in\z,\\
a_{\sigma^{ns}\tau^t} &=a_{\sigma^{ns}}\asd{\sigma^{ns}}{}{}{\tau^t}{a}=a_{\sigma^{ns}}, &s,t\in\z. 
\end{align*}
D'autre part, on sait par d\'efinition que $\b$ v\'erifie
\begin{align*}
\b_{\sigma^{ns}} &=1, & s\in\z,\\
\b_{\tau^t} &=\b_\tau^t=\b^t, & t\in\z,\\
\b_{\sigma^{ns}\tau^t}&=\b_{\sigma^{ns}}\asd{\sigma^{ns}}{}{}{\tau^t}{\b}=\asd{\sigma^{ns}}{}{t}{}{\b}=\b^{q^{ns}t}, & s,t\in\z.
\end{align*}
En rappelant alors que la formule g\'en\'erale du cup-produit est
\[(a\cup \b)_{x,y}=a_x\otimes\asd{x}{}{}{y}{\b},\]
il est facile de voir que le 2-cocycle $a\cup \b\in Z^2(D_n,\mu\{q'\})$ est d\'efini par la formule
\begin{equation*}
\begin{split}
(a\cup \b)_{\sigma^{ns_1}\tau^{t_1},\sigma^{ns_2}\tau^{t_2}}&=a_{\sigma^{ns_1}}(\asd{\sigma^{ns_1}\tau^{t_1}}{}{q^{ns_2}t_2}{}{\b})\\
&=[a_{\sigma^{ns_1}}(\asd{\sigma^{ns_1}}{}{}{}{\b})]^{q^{ns_2}t_2}\\
&=\prod_{i=0}^{s_1-1}[(\asd{\sigma^{ni}}{}{}{\sigma^n}{a})(\asd{\sigma^{ns_1}}{}{}{}{\b})]^{q^{ns_2}t_2}\\
&=\prod_{i=0}^{s_1-1}[\asd{\sigma^{ni}}{}{}{}{(}a_{\sigma^n}(\asd{\sigma^{n(s_1-i)}}{}{}{}{\b}))]^{q^{ns_2}t_2}\\
&=\prod_{i=0}^{s_1-1}[a_{\sigma^n}(\b^{q^{n(s_1-i)}})]^{q^{n(i+s_2)}t_2}\\
&=\prod_{i=0}^{s_1-1}[a_{\sigma^n}(\b)]^{q^{n(s_1+s_2)}t_2}\\
&=[a_{\sigma^n}(\b)]^{q^{n(s_1+s_2)}s_1t_2}\in\mu\{q'\},
\end{split}
\end{equation*}
o\`u l'on rappelle que $\sigma$ agit sur $M=\hom(G^\ab,\mu\{q'\})$ via $(\asd{\sigma}{}{}{}{a})(\b)=\asd{\sigma}{}{}{}{(}a(\asd{\sigma^{-1}}{}{}{}{\b}))$
pour $a\in M$ et $\b\in G^\ab$ (cf. par exemple \cite[IX.3]{SerreCorpsLocaux}) et que $\sigma$ agit sur $\mu\{q'\}=\bar\f_q^*$ via le Frobenius, i.e.
par \'el\'evation \`a la $q$-i\`eme puissance. On a par ailleurs utilis\'e dans le calcul la formule de $a_{\sigma^{ns_1}}$ pour $s_1\geq 1$, mais un
calcul similaire donne la m\^eme formule pour tout $s_1\in\z$.
\end{proof}

On voit alors que la classe de $a\cup \b$ dans $H^2(D_n,\mu\{q'\})$ est d\'ecrite par l'\'el\'ement $a_{\sigma^n}(\b)\in\mu\{q'\}$, et l'inverse est vrai
dans le sens suivant.

\begin{lem}\label{lemme changement de H2 pour mu}
Soient $\alpha$, $a$, $\b$ et $\beta$ comme ci-dessus. Alors la classe $\alpha_n\cup\beta\in H^2(D_n,\mu\{q'\})$ est nulle si et
seulement si $a_{\sigma^n}(\b)=1$.
\end{lem}

\begin{proof}
Supposons que $a_{\sigma^n}(\b)=1$. La formule du lemme \ref{lemme calcul de a cup b} nous dit alors
imm\'ediatement que l'on a $(a\cup \b)_{x,y}=1$ pour toute paire $(x,y)\in D_n^2$, donc on a bien $\alpha_n\cup \beta=0$.

Supposons maintenant que $\alpha_n\cup \beta=0$. Il existe alors une fonction continue $c:D_n\to\mu\{q'\}$ telle que pour toute paire
$(x,y)\in D_n^2$, on a l'\'egalit\'e
\[(a\cup \b)_{x,y}=c_{x}\asd{x}{}{}{y}{c}c^{-1}_{xy}.\]
Montrons que $(a\cup \b)_{\sigma^n,\tau}=1$. On a d'abord, d'apr\`es le calcul de $a\cup \b$ dans le lemme \ref{lemme calcul de a cup b},
\begin{align*}
c_{\tau}c_{\tau^t}c_{\tau^{t+1}}^{-1}=(a\cup \b)_{\tau,\tau^t}=1\quad &\Rightarrow\quad c_{\tau^{t+1}}=c_{\tau}c_{\tau^t}=\cdots=c_{\tau}^{t+1},\\
c_{\tau^{q^n}}c_{\sigma^n}c^{-1}_{\tau^{q^n}\sigma^n}=(a\cup \b)_{\tau^{q^n},\sigma^n}=1\quad
&\Rightarrow\quad c_{\sigma^n\tau}=c_{\tau^{q^n}\sigma^n}=c_{\tau^{q^n}}c_{\sigma^n},
\end{align*}
car $\tau$ agit trivialement et $\sigma^n\tau\sigma^{-n}=\tau^{q^n}$.
On calcule alors $(a\cup \b)_{\sigma^n,\tau}\in\mu\{q'\}$ :
\[(a\cup \b)_{\sigma^n,\tau}=c_{\sigma^n}\asd{\sigma^n}{}{}{\tau}{c}c_{\sigma^n\tau}^{-1}=
c_{\sigma^n}c_{\tau}^{q^n}(c_{\tau^{q^n}}c_{\sigma^n})^{-1}=c_{\sigma^n}c_{\tau}^{q^n}(c_{\tau}^{q^n}c_{\sigma^n})^{-1}=1.\]
On conclut que $a_{\sigma^n}(\b)=1$, car on a $(a\cup \b)_{\sigma^n,\tau}=a_{\sigma^n}(\b)^{q^n}$ et $\mu\{q'\}$ est $q$-divisible.
\end{proof}

Ce lemme nous dit que l'on peut oublier l'accouplement plut\^ot artificiel
\[H^1(\Gamma_n,M)\times (G^\ab)^{\varphi^n_q}\to H^2(D_n,\mu\{q'\})\hookrightarrow\q/\z,\]
et nous concentrer sur l'accouplement plus naturel
\[M\times (G^\ab)^{\varphi^n_q}\to\mu\{q'\}\subset \bar k^*.\]

\begin{lem}\label{lemme calcul de a sigma n}
Pour $\alpha\in H^1(\Gamma_1,M)$, $a\in Z^1(\Gamma_1,M)$ repr\'esentant $\alpha$ et $\b\in G^\ab$ on a
\[a_{\sigma^{n}}(\b)=a_{\sigma}\left(\prod_{i=0}^{n-1}\varphi_q^{i}(\b)\right),\]
pour tout $n\in\n$.
\end{lem}

\begin{proof}
C'est un calcul explicite :
\begin{multline*}\label{equation calcul de a sigma n}
a_{\sigma^{n}}(\b)=\prod_{i=0}^{n-1}(\asd{\sigma^{i}}{}{}{\sigma}{a})(\b)=\prod_{i=0}^{n-1}\asd{\sigma^{i}}{}{}{}{(}
a_{\sigma}(\asd{\sigma^{-i}}{}{}{}{\b}))=\prod_{i=0}^{n-1}[a_{\sigma}(\asd{\sigma^{-i}}{}{}{}{\b})]^{q^{i}}\\
=\prod_{i=0}^{n-1}a_{\sigma}(\asd{\sigma^{-i}}{}{q^{i}}{}{\b})=\prod_{i=0}^{n-1}a_\sigma(\varphi_q^{i}(\b))=a_{\sigma}
\left(\prod_{i=0}^{n-1}\varphi_q^{i}(\b)\right).
\end{multline*}
\end{proof}

C'est ce dernier lemme qui nous permet d'introduire la $q$-norme dans les calculs et de terminer la d\'emonstration du th\'eor\`eme.

\begin{proof}[Fin de la preuve du th\'eor\`eme \ref{theoreme formule sur fq}]
Soit $\alpha\in H^1(\Gamma_1,M)$ correspondant \`a un \'el\'ement non ramifi\'e. La proposition \ref{proposition traduction du theoreme cohomologique}
nous dit alors que, pour tout $n\in\n$, $\alpha_n$ est orthogonal \`a $G^\ab_{q,n}$. Soit $b\in G$. On a alors de fa\c con \'evidente que
$\bar b\in G^\ab_{q,n_b}$. Les lemmes \ref{lemme calcul de a sigma n} et \ref{lemme changement de H2 pour mu} nous disent alors que
$a_{\sigma}(N_q(b))=a_{\sigma^{n_b}}(\bar b)=1$.

En sens inverse, soit $a\in Z^1(\Gamma_1,M)$ tel que $a_{\sigma}(N_q(b))=1$ pour tout $b\in G$ et soit $\alpha$ la classe de $a$ dans
$H^1(\Gamma_1,M)$. Soient $n\in\n$, $\b\in G^\ab_{q,n}$ et $b\in G$ une pr\'eimage de $\b$. Il est facile de voir que $n_b$ divise $n$ et alors,
d'apr\`es le lemme \ref{lemme calcul de a sigma n}, on a
\begin{multline*}
a_{\sigma^{n}}(\b)=a_{\sigma^{ln_b}}(\b)=a_{\sigma}\left(\prod_{i=0}^{ln_b-1}\varphi_q^{i}(\b)\right)=
a_{\sigma}\left(\prod_{i=0}^{n_b-1}\prod_{j=0}^{l}\varphi_q^{jn_b+i}(\b)\right)=\\
a_{\sigma}\left(\prod_{i=0}^{n_b-1}\prod_{j=0}^{l}\varphi_q^{i}(\b)\right)=a_{\sigma}\left(\prod_{i=0}^{n_b-1}
\varphi_q^{i}(\b)^l\right)=a_\sigma(N_q(b))^l=1,
\end{multline*}
car on rappelle que $\varphi_q^{n_b}(\b)=\b$. Le lemme \ref{lemme changement de H2 pour mu} nous
dit alors que $\alpha_n$ est orthogonal \`a $\b$ pour l'accouplement \eqref{equation accouplement 2} et la proposition
\ref{proposition traduction du theoreme cohomologique} nous dit donc que $\alpha$ correspond \`a un \'el\'ement non ramifi\'e.

Il suffit alors pour conclure de rappeler que $\Gamma_k$ (resp. $H^1(k,M)$) est canoniquement isomorphe \`a $\Gamma_1$ (resp. $H^1(\Gamma_1,M)$) et
que $\sigma$ est envoy\'e sur $s$ (le $q$-Frobenius) par cet isomorphisme.
\end{proof}

\subsubsection{Une formulation plus pratique du th\'eor\`eme \ref{theoreme formule sur fq}}
On conclut cette section en donnant encore un accouplement qui rendra plus simples quelques calculs que l'on se propose de faire plus tard et qui
aidera peut-\^etre le lecteur \`a voir pourquoi on a la libert\'e de choisir le cocycle $a$ repr\'esentant la classe $\alpha\in H^1(k,M)$ dans le
th\'eor\`eme \ref{theoreme formule sur fq}.\\

On a d\'ej\`a remarqu\'e qu'un cocycle $a\in Z^1(k,M)$ est uniquement d\'efini par l'image du Frobenius. De plus, il est facile de voir que tout
\'el\'ement de $M$ peut \^etre une telle image. On a alors un isomorphisme $M\xrightarrow{\sim} Z^1(k,M)$ qui induit par
composition un morphisme surjectif $M\to H^1(k,M)$.\\
\'Etudions le noyau de ce morphisme. Un cocycle $a\in Z^1(k,M)$ est cohomologue au cocycle trivial si et seulement s'il existe $c\in M$ tel que
\[a_x=c\,\asd{x}{}{-1}{}{c}\quad\forall\,x\in\Gamma_k.\]
On voit alors qu'en particulier on doit avoir $a_s=c\,\asd{s}{}{-1}{}{c}$. On en d\'eduit que le noyau du morphisme $M\to H^1(k,M)$ est
le sous-groupe $N_0$ des \'el\'ements de la forme $c\,\asd{s}{}{-1}{}{c}$ avec $c\in M$. On a alors un isomorphisme
\[M_0:=M/N_0\xrightarrow{\sim}H^1(k,M).\]
En notant $N_q(G)$ l'image de $G$ dans $G^\ab$ par l'application $q$-norme (qui n'a aucune raison d'en \^etre un sous-groupe), on voit que
l'accouplement utilis\'e dans le th\'eor\`eme \ref{theoreme formule sur fq} revient alors \`a l'accouplement
\[M_0\times N_q(G)\to\mu\{q'\},\]
induit par l'accouplement \'evident $M\times G^\ab\to\mu\{q'\}$ et que le groupe de Brauer non ramifi\'e alg\'ebrique correspond au noyau \`a gauche de
cet accouplement. La libert\'e qu'on a pour choisir un repr\'esentant $a\in M$ d'un \'el\'ement $\a\in M_0$ pour faire cet accouplement-ci explique la
libert\'e de choix que l'on a dans le th\'eor\`eme \ref{theoreme formule sur fq}.\\

D\'emontrons donc que cet accouplement est bien d\'efini : il suffit de montrer que $N_q(G)$ est orthogonal \`a $N_0$, ce qu'on fait en remarquant
que l'on a trivialement $N_q(G)\subset (G^\ab)^{\varphi_q}$ et en notant que l'on a :

\begin{lem}\label{lemme calcul de N perp}
$N_0$ est l'orthogonal de $(G^\ab)^{\varphi_q}$ pour l'accouplement
\[M\times G^\ab\to\mu\{q'\}.\]
\end{lem}

\begin{proof}
Soit $\b\in G^\ab$. Comme l'accouplement est parfait, on sait que l'on a les \'equivalences suivantes
\begin{align*}
\b\perp N_0 &\Leftrightarrow \forall\, c\in M,\quad (c\,\asd{s}{}{-1}{}{c})(\b)=1\\
&\Leftrightarrow \forall\, c\in M,\quad c(\b)[(\asd{s}{}{}{}{c})(\b)]^{-1}=1\\
&\Leftrightarrow \forall\, c\in M,\quad c(\b)\asd{s}{}{}{}{[}c(\asd{s^{-1}}{}{}{}{\b})]^{-1}=1\\
&\Leftrightarrow \forall\, c\in M,\quad c(\b)c(\asd{s^{-1}}{}{}{}{\b})^{-q}=1\\
&\Leftrightarrow \forall\, c\in M,\quad c(\b\asd{s^{-1}}{}{-q}{}{\b})=1\\
&\Leftrightarrow \forall\, c\in M,\quad c(\b\varphi_q(\b)^{-1})=1\\
&\Leftrightarrow \b\varphi_q(\b)^{-1}=1\\
&\Leftrightarrow \b\in (G^\ab)^{\varphi_q}.
\end{align*}
\end{proof}

Puisque l'accouplement entre $G^\ab$ et $M$ est parfait, on peut d\'eduire du lemme que l'accouplement entre $M_0$ et $(G^\ab)^{\varphi_q}$ l'est
aussi. Le groupe de Brauer non ramifi\'e alg\'ebrique correspond alors dans cette formulation au sous-groupe de $M_0$ qui est orthogonal \`a $N_q(G)$,
ce qui est la m\^eme chose que l'orthogonal du sous-groupe $\mathcal{N}_q(G)$ de $(G^\ab)^{\varphi_q}$ engendr\'e par $N_q(G)$. On voit aussit\^ot
que, avec cette formulation,

\begin{pro}\label{proposition calcul de brnral}
Le groupe $\brnral V$ est isomorphe \`a $(G^\ab)^{\varphi_q}/\mathcal{N}_q(G)$ en tant que groupe ab\'elien.
\end{pro}

\begin{proof}
En effet, il suffit de noter que $(G^\ab)^{\varphi_q}/\mathcal{N}_q(G)$ correspond au dual de l'orthogonal de $\mathcal{N}_q(G)$ pour l'accouplement
parfait $M_0\times (G^\ab)^{\varphi_q}\to\mu\{q'\}$ et que ce dernier groupe est isomorphe au groupe $\brnral V$ d'apr\`es le th\'eor\`eme
\ref{theoreme formule sur fq}. Ces trois groupes sont donc isomorphes en tant que groupes ab\'eliens.
\end{proof}

En particulier, on a le corollaire suivant.

\begin{cor}\label{corollaire calcul de brnral}
Soit $G$ un $k$-groupe tel que $(G^\ab)^{\varphi_q}$ est engendr\'e par les $q$-normes. Alors on a $\brnral V=0$ pour $V=G\backslash G'$.
\qed
\end{cor}

\subsection{Une formule alg\'ebrique sur un corps de caract\'eristique 0}\label{section formule algebrique car 0}
Dans toute cette section, $k$ repr\'esente un corps de caract\'eristique $0$. On retrouve ici une formule pour le groupe $\brnral V$ o\`u $V=G\backslash G'$ en nous
ramenant au cas des corps finis et en utilisant la formule d\'evelopp\'ee dans la section pr\'ec\'edente. Cette m\'ethode a d\'ej\`a \'et\'e utilis\'ee
par Colliot-Th\'el\`ene et Kunyavski\u\i\, (cf. \cite{ColliotKunyavskii}) dans le cas des espaces homog\`enes principaux, i.e. lorsque $G$ est trivial,
et par Borovoi, Demarche et Harari (cf. \cite{BDH}) dans le cas o\`u $G$ est connexe ou ab\'elien (ou encore plus g\'en\'eralement
de type ``ssumult'') et avec des groupes $G'$ plus g\'en\'eraux.

Notre formule sera d'une nature un peu diff\'erente car, dans les deux cas cit\'es, il se trouve que le groupe de Brauer non ramifi\'e alg\'ebrique
est toujours nul sur un corps fini, ce qui rend les calculs beaucoup plus simples. Dans les sections qui suivent, on montrera qu'il n'en est pas du tout
ainsi dans le cas o\`u $G$ est fini et non ab\'elien. La formule qu'on obtient ressemblera donc \`a celle que l'on a donn\'ee pour les corps finis. Elle
s'appuie notamment sur l'accouplement naturel $M\times G^\ab\to\mu_n$, o\`u $n$ est l'exposant de $G$.\\

On commence avec quelques d\'efinitions analogues \`a celles donn\'ees dans la section \ref{section formule algebrique}. Dans toute cette section,
$n$ repr\'esente l'exposant du $k$-groupe fini $G$. On fixe aussi d\'esormais une racine primitive $n$-i\`eme de l'unit\'e $\zeta_n\in\bar k$.

\begin{defi}
On d\'efinit le morphisme $q:\Gamma_k\to(\z/n\z)^*$ par la formule suivante :
\[\forall\,\sigma\in\Gamma_k,\quad\asd{\sigma}{}{}{n}{\zeta}=\zeta_n^{q(\sigma)}.\]
On d\'efinit ensuite, pour tout $\sigma\in\Gamma_k$, l'application $\varphi_\sigma:G(\bar k)\to G(\bar k)$ par la formule
\[\varphi_\sigma(b)=\asd{\sigma^{-1}}{}{q(\sigma)}{}{b},\]
et on note abusivement $\varphi_\sigma$ aussi l'application induite sur $G^\ab(\bar k)$ qui en est un automorphisme.\\
Enfin, pour $\sigma\in\Gamma_k$ et $b\in G$, on d\'efinit $n_{\sigma,b}$ comme le plus petit entier $>0$ tel que
$\varphi^{n_{\sigma,b}}_\sigma(b)$ soit conjugu\'e \`a $b$. On d\'efinit alors l'application \emph{$\sigma$-norme}
$N_\sigma:G(\bar k)\to G^\ab(\bar k)$ par la formule suivante :
\[N_\sigma(b):=\prod_{i=0}^{n_{\sigma,b}-1}\overline{\varphi_\sigma^i(b)}=\prod_{i=0}^{n_{\sigma,b}-1}\varphi_\sigma^i(\bar b).\]
\end{defi}

Avec ces d\'efinitions, on peut \'enoncer la formule pour $\brnral V$.

\begin{thm}\label{theoreme formule en car 0}
Soit $G$ un $k$-groupe fini plong\'e dans $G'$ semi-simple simplement connexe et soit $V=G\backslash G'$. Soit enfin $M$ le groupe des caract\`eres de $G$.
En identifiant $\bral V$ avec $H^1(k,M)$, le groupe de Brauer non ramifi\'e alg\'ebrique $\brnral V$ de $V$ est donn\'e par les \'el\'ements
$\alpha\in H^1(k,M)$ tels que pour $a\in Z^1(k,M)$ un cocycle (quelconque) repr\'esentant $\alpha$, on a
\begin{equation}\label{equation formule brnral en car 0}
a_\sigma(N_\sigma(b))=1\quad \forall\, b\in G,\forall\, \sigma\in\Gamma_k.
\end{equation}
\end{thm}

\begin{rem}
Il est important de noter que cette formule ne d\'epend que du stabilisateur $G$. En particulier, elle ne d\'epend point du groupe ambiant $G'$ tant
qu'il s'agit d'un groupe semi-simple et simplement connexe. Cette remarque peut par ailleurs \^etre g\'en\'eralis\'ee \`a \emph{tout} le groupe de Brauer non
ramifi\'e $\brnr V$ et pour un stabilisateur $G$ \emph{quelconque}, comme nous l'a fait remarquer Colliot-Th\'el\`ene (cf. \cite{ColliotBrnr}).

Ceci ne veut pas pourtant dire que toutes
ces vari\'et\'es aient des chances d'\^etre birationnelles ou encore stablement birationnelles. En effet, il suffit de penser au cas o\`u $G=\{1\}$ :
il existe des exemples de groupes alg\'ebriques $G'$ semi-simples simplement connexes rationnels, comme $\sln$, et d'autres dans la m\^eme famille
($A_n$) qui ne sont m\^eme pas stablement rationnels (cf. \cite{MerkurjevSL1}).

D'un autre c\^ot\'e, le lemme sans nom (cf. \cite[Corollary 3.9]{ColliotSansucChili}) nous dit par exemple que lorsque le groupe ambiant $G'$ est
$\sln$, alors les vari\'et\'es quotient obtenues pour deux plongements diff\'erents de $G$ sont stablement birationnelles. Il serait int\'eressant
de savoir s'il en est de m\^eme pour d'autres groupes $G'$.
\end{rem}

\subsubsection{D\'emonstration du th\'eor\`eme \ref{theoreme formule en car 0}}
Puisque la caract\'eristique de $k$ est 0, le th\'eor\`eme d'Hironaka nous permet de trouver une $k$-compactification lisse $X$ de $V$ que l'on fixe
pour toute la suite. On rappelle que de la m\^eme fa\c con que l'on identifie $\bral V$ avec $H^1(k,M)$, on peut identifier
$\brnral V=\bral X$ avec $H^1(k,\pic X_{\bar k})$. Soit $\gamma$ un sous-groupe ferm\'e pro-cyclique de $\Gamma_k$. Il est \'evident par fonctorialit\'e
de la restriction que si un \'el\'ement $\alpha\in H^1(k,M)$ provient de $H^1(k,\pic X_{\bar k})$, alors son image dans $H^1(\gamma,M)$ provient de
$H^1(\gamma,\pic X_{\bar k})$. Lorsqu'on consid\`ere tous les sous-groupes pro-cycliques de $\Gamma_k$ ensemble, on a l'affirmation r\'eciproque
au sens suivant.

\begin{lem}\label{lemme brnral <<egal au sha1>>}
Soit $G$ un $k$-groupe fini plong\'e dans $G'$ semi-simple simplement connexe et soit $V=G\backslash G'$. Soit $X$ une $k$-compactification lisse de
$V$. Soit enfin $M$ le groupe des caract\`eres de $G$ et $\alpha\in H^1(k,M)$. Supposons que pour tout sous-groupe pro-cyclique $\gamma$ de
$\Gamma_k$, l'image de $\alpha$ dans $H^1(\gamma,M)$ provient de $H^1(\gamma,\pic X_{\bar k})$. Alors $\alpha$ provient de $H^1(k,\pic X_{\bar k})$.
\end{lem}

\begin{rem}
Ce lemme nous dit en particulier que l'on a toujours l'inclusion $\Sha_\mathrm{cyc}^1(k,M)\subset\brnral V$. Dans les sections qui suivent on montrera
que cette inclusion peut en fait \^etre stricte, par opposition au cas des espaces homog\`enes \`a stabilisateur connexe ou ab\'elien, cf. \cite{BDH}. On
g\'en\'eralise ainsi le r\'esultat de Demarche dans \cite[\S 6]{Demarche}, valable pour $k$ un corps de nombres.
\end{rem}

\begin{proof}
Soit $\Div_\infty X_{\bar k}$ le sous-groupe de $\Div X_{\bar k}$ des diviseurs \`a support dans le compl\'ementaire de $V$ dans $X$. C'est un
$\Gamma_k$-module de permutation. En effet, le compl\'ementaire de $V$ \'etant un ferm\'e strict de $X$, il contient un nombre fini de diviseurs
que le groupe $\Gamma_k$ permute, a priori, de fa\c con non triviale.

Puisque $X_{\bar k}$ et $V_{\bar k}$ sont des vari\'et\'es int\`egres et lisses, on a la suite exacte classique (cf. par exemple \cite[Chapter II, 6.5 et 6.16]{Hartshorne})
\[\Div_\infty X_{\bar{k}}\to\pic X_{\bar{k}}\to\pic V_{\bar{k}}\to 0.\]
En ajoutant un $0$ \`a gauche de cette suite, il se trouve qu'elle est en fait aussi exacte \`a gauche. En effet, soit $D\in \Div_{\infty} X_{\bar k}$
tel que son image dans $\pic X_{\bar k}$ est nulle. Il existe alors une fonction $f\in \bar k(X)$ telle que $\div(f)=D$, o\`u $\div$ est l'application
diviseur $\bar k(X)\to\Div X_{\bar k}$. Comme $D$ est \`a support dans $X_{\bar k}\smallsetminus V_{\bar k}$, on a que $f$ est inversible sur $V_{\bar k}$,
i.e. $f\in \bar k[V]^*=\bar k^*$ (cf. le d\'ebut de la section \ref{section formules}). Donc $f$ est une constante et $D=0$, ce qui montre l'injectivit\'e.

De cette suite on d\'eduit la suite exacte longue suivante :
\[\cdots\to H^1(\Gamma_k,\Div_\infty X_{\bar{k}})\to H^1(\Gamma_k,\pic X_{\bar{k}})\to H^1(\Gamma_k,\pic V_{\bar{k}})
\to H^2(\Gamma_k,\Div_\infty X_{\bar k})\to\cdots.\]
Or, $\Div_\infty X_{\bar k}$ \'etant un module de permutation, on a $H^1(\Gamma_k,\Div_\infty X_{\bar{k}})=0$. En effet, en utilisant le lemme de
Shapiro on peut se restreindre \`a d\'emontrer que $H^1(H_i,\z)=0$ pour certains sous-groupes ouverts $H_i$ de $\Gamma_k$, ce qui est \'evident,
car $\Gamma_k$ (et par cons\'equent, $H_i$) est un groupe profini et $\z$ est sans torsion. La suite devient alors
\[0\to H^1(k,\pic X_{\bar{k}})\to H^1(k,M)\to H^2(k,\Div_\infty X_{\bar k}).\]
Pour tout sous-groupe ferm\'e pro-cyclique $\gamma$ de $\Gamma_k$ on peut consid\'erer alors les applications de restriction et obtenir le
diagramme commutatif \`a lignes exactes suivant :
\[\xymatrix{
0 \ar[r] & H^1(k,\pic X_{\bar k}) \ar[r] \ar[d] & H^1(k,M) \ar[r] \ar[d] & H^2(k,\Div_\infty X_{\bar k}) \ar[d] \\
0 \ar[r] & H^1(\gamma,\pic X_{\bar k}) \ar[r] & H^1(\gamma,M) \ar[r] & H^2(\gamma,\Div_\infty X_{\bar k}).
}\]
Avec ce diagramme, on voit que l'affirmation du lemme d\'ecoule de la nullit\'e du groupe $\Sha_{\mathrm{cyc}}^2(k,\Div_\infty X_{\bar k})$ qui
est d\'efini comme le sous-groupe de $H^2(k,\Div_\infty X_{\bar k})$ des \'el\'ements dont la restriction \`a tout sous-groupe
pro-cyclique de $\Gamma_k$ est triviale. En effet, si l'image de $\alpha$ provient de $H^1(\gamma,\pic X_{\bar k})$ pour tout $\gamma$, alors son
image dans $H^2(\gamma,\Div_\infty X_{\bar k})$ est nulle pour tout $\gamma$, ce qui nous dit que l'image de $\alpha$ dans
$H^2(k,\Div_\infty X_{\bar k})$ est en fait dans $\Sha_{\mathrm{cyc}}^2(k,\Div_\infty X_{\bar k})$ et est donc nulle, ce qui nous dit que $\alpha$ provient
de $H^1(k,\pic X_{\bar k})$.

Montrons donc que $\Sha_{\mathrm{cyc}}^2(k,\Div_\infty X_{\bar k})=0$. Puisque $\Div_\infty X_{\bar k}$ est un module de permutation, on sait
(toujours d'apr\`es le lemme de Shapiro) qu'il existe des sous-groupes ouverts $H_i$ de $\Gamma_k$ tel que l'on a
\[H^2(k,\Div_\infty X_{\bar k})\cong\bigoplus_i H^2(H_i,\z)\cong \bigoplus_i H^1(H_i,\q/\z)=\bigoplus_i \hom(H_i,\q/\z).\]
Si l'on consid\`ere alors un morphisme non trivial $\phi$ dans $H^1(H_i,\q/\z)$ et un \'el\'ement $\sigma\in H_i$ tel que son image par $\phi$ soit
non nulle, on voit que la restriction de $\phi$ au sous-groupe ferm\'e pro-cyclique engendr\'e par $\sigma$ est non nulle. La compatibilit\'e de la
restriction avec les suites exactes longues et avec l'isomorphisme de Shapiro nous dit alors que pour tout \'el\'ement non nul $\beta$ de
$H^2(k,\Div_\infty X_{\bar k})$ il existe un sous-groupe ferm\'e pro-cyclique $\gamma$ de $\Gamma_k$ tel que la restriction de $\beta$ \`a
$\gamma$ est non nulle, ce qui entra\^\i ne la nullit\'e de $\Sha_{\mathrm{cyc}}^2(k,\Div_\infty X_{\bar k})$.
\end{proof}

Soient maintenant $\alpha\in H^1(k,M)$, $\sigma\in\Gamma_k$ et $\gamma$ le sous-groupe ferm\'e pro-cyclique engendr\'e par $\sigma$.
Compte tenu de ce lemme, on voit facilement que pour d\'emontrer le th\'eor\`eme \ref{theoreme formule en car 0} il suffit de d\'emontrer la proposition
suivante :

\begin{pro}\label{proposition intermediaire theoreme formule en car 0}
Avec les notations ci-dessus, la restriction $\alpha'\in H^1(\gamma,M)$ de $\alpha$ provient de $H^1(\gamma,\pic X_{\bar k})$ si et seulement si,
pour $a\in Z^1(k,M)$ un cocycle repr\'esentant $\alpha$, on a
\[a_\sigma(N_\sigma(b))=1\quad \forall\, b\in G.\]
\end{pro}

En effet, on peut voir que la restriction $\alpha'$ est clairement repr\'esent\'ee par le cocycle $a|_{\gamma}\in Z^1(\gamma,M)$ et donc $a_\sigma$
n'est rien d'autre que $(a|_\gamma)_\sigma$. La d\'emarche consistera alors \`a trouver un corps fini $\f$ tel que l'on puisse regarder $G$ et $M$
comme des $\f$-groupes dont l'action du Frobenius co\"\i ncide avec celle de $\sigma$. On aura alors des morphismes canoniques $H^1(\gamma,M)\to H^1(\f,M)$ et
$H^1(\gamma,\pic X_{\bar k})\xrightarrow{\sim} H^1(\f,\pic X_{\bar\f})$ et des \'egalit\'es $\varphi_{q_0}=\varphi_\sigma$ et $N_{q_0}=N_\sigma$,
o\`u $q_0$ est le cardinal de $\f$. Le r\'esultat d\'ecoulera alors du th\'eor\`eme \ref{theoreme formule sur fq}.\\

Pour ce faire, on suit la d\'emarche de Colliot-Th\'el\`ene et Kunyavski\u\i , laquelle est explicit\'ee en d\'etails dans l'article de Borovoi, Demarche et
Harari et que l'on reproduit ici (\'etapes 1 et 2).

\paragraph*{\'Etape 1 : Restriction au cas o\`u $k$ est de type fini sur $\q$}
Soit $\tilde k$ une extension finie galoisienne de $k$ d\'eployant $\pic X_{\bar k}$ et contenant $\zeta_n$ et posons
$\tilde\Gamma=\gal(\tilde k/k)$. Puisque $V$ est clairement $k$-unirationnelle, le groupe $\pic X_{\bar k}$ est sans torsion (cf. par exemple
le th\'eor\`eme \ref{theoreme pic invariant} dans l'appendice) et donc \'egal \`a son groupe de N\'eron-Severi $\ns X_{\bar k}$. De plus, puisque le
groupe (profini) de Galois de $\tilde k$ agit trivialement sur $\pic X_{\bar k}$, on a que
$H^1(\tilde k,\pic X_{\bar k})=0$, ce qui entra\^\i ne que $H^1(k,\pic X_{\bar k})=H^1(\tilde\Gamma,\pic X_{\bar k})$ par la suite de restriction-inflation.
Si l'on note $\tilde\gamma$ l'image de $\gamma$ dans $\tilde\Gamma$, on voit que l'on a aussi $H^1(\gamma,\pic X_{\bar k})=H^1(\tilde\gamma,\pic X_{\bar k})$.

\'Ecrivons $\tilde k=k[t]/P(t)$ pour un certain polyn\^ome irr\'eductible $P\in k[t]$. Comme tous les morphismes consid\'er\'es sont de
type fini, on sait qu'il existe un corps $k_0\subset k$ de type fini sur $\q$ sur lequel $G,G',V,X$ et les plongements $G\hookrightarrow G'$ et
$V\hookrightarrow X$ sont d\'efinis. De plus, on peut supposer que $X$ admet un $k_0$-point : l'image de l'\'el\'ement neutre $e\in G'$ par la projection
$G'\to V$ (le groupe $G'$, et en particulier $e$, sont d\'efinis sur un corps de nombres). Enfin, quitte \`a agrandir $k_0$, on peut supposer que $P\in k_0[t]$ et
que $\tilde k_0:=k_0[t]/P(t)$ est une extension galoisienne qui d\'eploie $\pic X_{\bar k}$. En effet, si $a$ est une racine de $P$ telle que
$\tilde k=k[a]$ et $\tilde k_0=k_0[a]$ et si $a_1$ est un conjugu\'e de $a$, alors on a $a_1=P_1(a)$ pour un certain polyn\^ome $P_1\in k[t]$, car
$a_1\in\tilde k=k[a]$. Il suffit alors de demander que $k_0$ contienne les coefficients de $P_1$ pour tout conjugu\'e $a_1$ de $a$ pour que
$\tilde k_0/k_0$ soit galoisienne. Le fait qu'elle d\'eploie $\pic X_{\bar k_0}$ est \'evident. Avec ces suppositions, on a un isomorphisme canonique
entre $\tilde\Gamma$ et $\gal(\tilde k_0/k_0)$ et des isomorphismes
\[\pic X_{\tilde k_0}\xrightarrow{\sim}\pic X_{\bar k_0}\xrightarrow{\sim}\pic X_{\bar k}\xleftarrow{\sim}\pic X_{\tilde k}.\]
En effet, les isomorphismes \`a gauche et \`a droite d\'ecoulent du lemme 6.3 de \cite{Sansuc81}, tandis que celui du milieu d\'ecoule de l'invariance
du groupe de N\'eron-Severi par changement de base alg\'ebriquement clos, cf. \cite[Proposition 3.1]{MaulikPoonen}
(cf. aussi le th\'eor\`eme \ref{theoreme pic invariant} dans l'appendice). On trouve alors finalement l'isomorphisme
\[H^1(\gal(\tilde k_0/k_0),\pic X_{\tilde k_0})\xrightarrow{\sim} H^1(\tilde\Gamma,\pic X_{\tilde k}),\]
ce qui nous dit que l'on peut supposer $k=k_0$ et $\tilde k=\tilde k_0$.

\paragraph*{\'Etape 2 : R\'eduction au cas d'un corps fini}
Puisque maintenant $k$ est de type fini sur $\q$, on sait que l'on peut trouver :
\begin{enumerate}
\item Un anneau int\`egre et r\'egulier $A$ de type fini sur $\z$ et de corps de fractions $k$, tel que la fermeture int\'egrale $\tilde A$ dans $\tilde k$
soit finie, \'etale et galoisienne sur $A$ de groupe de Galois $\tilde\Gamma$.
\item Des $A$-sch\'emas $\mathcal{G}$, $\cal{G}'$, $\mathcal{V}$ et $\mathcal{X}$ de type fini, de fibres g\'en\'eriques respectives $G$, $G'$, $V$ et $X$ et tels que :
\begin{itemize}
\item $\cal{G}'$ est un $A$-sch\'ema en groupes semi-simple simplement connexe (au sens de \cite[XIX.2.7, XXII.4.3.3]{SGA3III}) et $\mathcal{G}$ est un
sous-$A$-sch\'ema en groupes fini et lisse de $\cal{G}'$ ;
\item $\mathcal{V}=\mathcal{G}\backslash \cal{G}'$ est lisse sur $A$, $\mathcal{X}$ est propre et lisse sur $A$ et on a une $A$-immersion ouverte
$\mathcal{V}\hookrightarrow\mathcal{X}$ \'etendant $V\hookrightarrow X$ ;
\item pour tout point $x$ de $\spec A$ de corps r\'esiduel $\kappa_x$, la fibre $\mathcal{X}_x\subset\mathcal{X}$ est une compactification lisse de
$\mathcal{V}_x\subset\mathcal{V}$ sur $\kappa_x$.
\end{itemize}
\end{enumerate}

En notant que $X$ est unirationnelle et que $k$ est un corps de caract\'eristique $0$, on a $H^1(X,\O_X)=H^2(X,\O_X)=0$. Alors, comme la
dimension des groupes de cohomologie est une fonction semi-continue sup\'erieurement (cf. \cite[III, Theorem 12.8]{Hartshorne}), on peut supposer,
quitte \`a restreindre encore $\spec A$, que $H^1(\mathcal{X}_x,\O_{\mathcal{X}_x})=H^2(\mathcal{X}_x,\O_{\mathcal{X}_x})=0$ pour tout point
$x$ de $\spec A$. Alors, d'apr\`es \cite[Cor. 2.7]{GrothendieckPicard}, on trouve des isomorphismes de sp\'ecialisation
\[\pic X\xrightarrow{\sim}\pic \mathcal{X}_x,\]
pour tout point $x\in\spec A$. De m\^eme, quitte \`a r\'eduire encore $\spec A$, on trouve des isomorphismes 
\[\pic X_{\tilde k}\xrightarrow{\sim}\pic \tilde{\mathcal{X}}_{\tilde x}\]
pour tout point $\tilde x\in\spec\tilde A$, o\`u $\tilde{\mathcal{X}}=\mathcal{X}\times_A\tilde A$, et on en a de m\^eme pour tout sous-anneau
$\tilde A^\Delta$ de $\tilde A$ correspondant aux \'el\'ements invariants par un sous-groupe $\Delta$ de $\tilde\Gamma$ (il suffit de consid\'erer
toute la situation sur le sous-corps $\tilde k^\Delta\subset\tilde k$ des invariants par $\Delta$). De cela on d\'eduit que ces isomorphismes sont
d'ailleurs compatibles avec les actions galoisiennes correspondantes.

Si l'on consid\`ere alors le sous-anneau $A_1:=\tilde A^{\tilde\gamma}$ des \'el\'ements $\tilde\gamma$-invariants (ou $\sigma$-invariants) de
$\tilde A$, on a que le morphisme $\rho:\spec\tilde A\to\spec A_1$ est un rev\^etement \'etale cyclique de groupe de Galois $\tilde\gamma$ avec $\tilde A$
et $A_1$ des anneaux int\`egres, r\'eguliers et de type fini sur $\z$. La version g\'eom\'etrique du th\'eor\`eme de \v Cebotarev (cf.
\cite[Theorem 7]{SerreZeta}) nous dit alors qu'il existe une infinit\'e de points ferm\'es $x_1$ de $\spec A_1$ tels que la fibre en $x_1$ de
$\rho$ soit un point ferm\'e $\tilde x$ de $\spec \tilde A$. Ceci veut dire que si l'on note $\f$ le corps r\'esiduel en $x_1$ et $\mathbb{E}$ le corps
r\'esiduel en $\tilde x$, on a une extension de corps r\'esiduels $\mathbb{E}/\f$ de groupe de Galois $\tilde\gamma$ et par suite, en notant
$\mathcal{X}_1=\mathcal{X}\times_A A_1$ et $Y$ la $\f$-vari\'et\'e $(\mathcal{X}_1)_{x_1}$, un isomorphisme
\[H^1(\tilde\gamma,\pic X_{\tilde k})\xrightarrow{\sim} H^1(\tilde\gamma,\pic Y_{\mathbb{E}}).\]
De plus, toujours d'apr\`es le lemme 6.3 de \cite{Sansuc81}, on a que $\pic Y_{\mathbb{E}}\cong\pic Y_{\bar\f}$ car $\pic Y_{\bar\f}$ est d\'eploy\'e sur
$\mathbb{E}$ d'apr\`es la compatibilit\'e des actions galoisiennes \'evoqu\'ee ci-dessus.\\

Pour r\'esumer, on a les isomorphismes suivants :
\begin{multline*}
H^1(\gamma,\pic X_{\bar k})\cong H^1(\tilde\gamma,\pic X_{\bar k})\cong H^1(\tilde\gamma,\pic X_{\tilde k})\\
\cong H^1(\tilde\gamma, \pic Y_{\mathbb{E}}) \cong H^1(\tilde\gamma,\pic Y_{\bar\f})\cong H^1(\gamma,\pic Y_{\bar\f}).
\end{multline*}
En particulier, on a le diagramme commutatif suivant :
\begin{equation}\label{equation diagramme specialisation des picard}
\xymatrix{
H^1(\gamma,\pic X_{\bar k}) \ar@{^{(}->}[r] \ar[d]^{\sim} & H^1(\gamma,M) \ar@{=}[d] \\
H^1(\gamma,\pic Y_{\bar\f}) \ar@{^{(}->}[r] & H^1(\gamma,M) \\
}
\end{equation}
o\`u $M=M(\bar k)=M(\bar\f)$ repr\'esente le groupe des caract\`eres (g\'eom\'etriques) de $G$ (ou encore $\cal G$, $\cal G_{x_1}$). On rappelle que,
puisque $G$ est fini (donc a fortiori $M$ aussi), les points g\'eom\'etriques de $M$ sont invariants par changement de base. L'action de $\gamma$ sur
$M(\bar k)$ et sur $M(\bar\f)$ \'etant induite par celle sur $\pic X_{\bar k}$ et sur $\pic X_{\bar\f}$ respectivement, on voit qu'il s'agit bien de la
m\^eme action.

Consid\'erons le morphisme surjectif $f:\Gamma_\f\to\gamma$ induit par l'extension $\bb E/\f$ qui envoie le Frobenius $s$ sur $\sigma$ (c'est un
isomorphisme d\`es que l'ordre de $\sigma$ n'est pas fini). Cela nous donne des morphismes canoniques d'inflation
\[H^1(\gamma,M)\hookrightarrow H^1(\f,M)\quad\text{et}\quad H^1(\gamma,\pic Y_{\bar\f})\xrightarrow{\sim}H^1(\f,\pic Y_{\bar\f})=\bral Y,\]
o\`u l'action de $\Gamma_\f$ sur $M$ correspond bien entendu \`a celle induite par $f$ (il en va de m\^eme par ailleurs pour l'action sur
$G(\bar k)=\cal G_{x_1}(\bar \f)$). La premi\`ere fl\`eche est injective (car d'inflation) et la deuxi\`eme fl\`eche est un isomorphisme d'apr\`es la suite
de restriction-inflation car on a d\'ej\`a vu que $\pic Y_{\bar\f}$ est sans torsion et d\'eploy\'e par $\bb E$. En mettant ensemble ces morphismes avec
le diagramme \eqref{equation diagramme specialisation des picard} on obtient le diagramme commutatif
\begin{equation}\label{equation diagramme specialisation des picard 2}
\xymatrix{
H^1(\gamma,\pic X_{\bar k}) \ar@{^{(}->}[r] \ar[d]^{\sim} & H^1(\gamma,M) \ar@{^{(}->}[d] \\
H^1(\f,\pic Y_{\bar\f}) \ar@{^{(}->}[r] & H^1(\f,M) \\
}
\end{equation}
Enfin, puisque $Y$ est une compactification lisse d'un $\f$-espace homog\`ene sous $\cal G'_{x_1}$ \`a stabilisateur fini $\cal G_{x_1}$ et que l'on
peut toujours choisir $x_1$ de fa\c con que l'ordre de $\cal G_{x_1}$ soit premier au cardinal de $\f$, on a enfin le droit d'utiliser le th\'eor\`eme
\ref{theoreme formule sur fq}.

\paragraph*{\'Etape 3 : Fin de la preuve, utilisation du th\'eor\`eme \ref{theoreme formule sur fq}}
Reprenons alors un \'el\'ement $\alpha\in H^1(k,M)$, $a\in Z^1(k,M)$ repr\'esentant $\alpha$, $\sigma\in\Gamma_k$ et $\gamma$ le sous-groupe
pro-cyclique de $\Gamma_k$ engendr\'e par $\sigma$ comme dans la proposition \ref{proposition intermediaire theoreme formule en car 0}. D'apr\`es
le diagramme commutatif \eqref{equation diagramme specialisation des picard 2}, la restriction $\alpha'$ de $\alpha$ dans $H^1(\gamma,M)$ provient
de $H^1(\gamma,\pic X_{\bar k})$ si et seulement si son image $\tilde\alpha\in H^1(\f,M)$ provient de $H^1(\f,\pic Y_{\bar\f})=\bral Y$. Consid\'erons
les cocycles $a'=a|_{\gamma}$ et $\tilde a=a'\circ f$, qui repr\'esentent respectivement $\alpha'$ et $\tilde\alpha$. On voit alors gr\^ace au
th\'eor\`eme \ref{theoreme formule sur fq} que $\tilde\alpha$ provient de $H^1(\f,\pic Y_{\bar\f})$ si et seulement si
\[\tilde a_{s}(N_{q_0}(b))=1\quad\forall\, b\in G,\]
o\`u $q_0$ est le cardinal de $\f$. Or, on a clairement
\[\tilde a_{s}(N_{q_0}(b))=a'_{\sigma}(N_{q_0}(b))=a_{\sigma}(N_{q_0}(b)),\]
d'o\`u l'on voit qu'il suffit de d\'emontrer que $N_{q_0}=N_\sigma$ pour conclure.

D'apr\`es la d\'efinition de $N_{q_0}$ et de $N_\sigma$, il suffit en fait de d\'emontrer l'\'egalit\'e $\varphi_{q_0}=\varphi_\sigma$. Et d'apr\`es la
d\'efinition de ces applications, tout ce qu'il faut d\'emontrer est que
\[q(\sigma)\equiv q_0\mod n,\quad\text{ou encore}\quad\zeta_n^{q_0}=\zeta_n^{q(\sigma)},\]
o\`u l'on rappelle que $n$ est l'exposant de $G$. Or, cela est vrai de fa\c con \'evidente, car $\sigma$ agit sur $\mu_{n}$ de la m\^eme fa\c con que
le Frobenius de $\f$ (on rappelle que $\mu_n\subset\tilde k$ par hypoth\`ese), lequel envoie $\zeta_{n}$ en $\zeta_{n}^{q_0}$ par d\'efinition, d'o\`u
\[\zeta_n^{q(\sigma)}=\asd{\sigma}{}{}{n}{\zeta}=\asd{s}{}{}{n}{\zeta}=\zeta_n^{q_0}.\]
Ceci conclut donc la d\'emonstration de la proposition \ref{proposition intermediaire theoreme formule en car 0} et par cons\'equent celle du
th\'eor\`eme \ref{theoreme formule en car 0}.

\subsubsection{Applications et simplifications du th\'eor\`eme \ref{theoreme formule en car 0}}
La formule, telle qu'elle est donn\'ee dans le th\'eor\`eme \ref{theoreme formule en car 0}, n'est pas pratique pour faire des calculs. En effet, il
suffit de noter tout simplement que $\Gamma_k$ est en g\'en\'eral un groupe infini. Or, il se trouve que, comme on pourrait s'y attendre, il suffit
de se restreindre \`a un quotient de $\Gamma_k$ donn\'e par une extension finie bien pr\'ecise. Le r\'esultat est le suivant :

\begin{pro}
Sous les m\^emes hypoth\`eses que le th\'eor\`eme \ref{theoreme formule en car 0}, soient $L\supset k$ une extension finie galoisienne de $k$
d\'eployant $G$, $n=\exp (G)$ et $\zeta_n\in\bar k$ une racine primitive $n$-i\`eme de l'unit\'e. Alors, pour tout \'el\'ement
$\sigma\in\Gamma_{L(\zeta_n)}\subset\Gamma_k$ et pour tout  $b\in G$ on a $N_\sigma(b)=\bar b$.

En particulier, un cocycle $a\in Z^1(k,M)$ v\'erifiant l'\'egalit\'e \eqref{equation formule brnral en car 0} est forc\'ement trivial sur
$\Gamma_{L(\zeta_n)}$.
\end{pro}

\begin{proof}
Soit $\sigma\in\Gamma_{L(\zeta_n)}$, alors il est \'evident que $\sigma$ fixe $\zeta_n$, ce qui nous dit que $q(\sigma)=1$. De plus, puisque $L$
d\'eploie $G$, $\sigma$ agit de fa\c con triviale sur $G(\bar k)$. On a alors que $\varphi_\sigma$ se r\'eduit \`a l'identit\'e et alors $N_\sigma$
correspond \`a la projection $G\to G^\ab$.

On voit alors que pour un cocycle $a\in Z^1(k,M)$ v\'erifiant l'\'egalit\'e \eqref{equation formule brnral en car 0}, on a $a_{\sigma}(\b)=1$ pour
tout $\b\in G^\ab$, ce qui nous dit que $a_\sigma=1$.
\end{proof}

De cette proposition et du th\'eor\`eme \ref{theoreme formule en car 0} on d\'eduit aussit\^ot que

\begin{cor}\label{corollaire brnral se calcule dans une extension finie}
Sous les m\^emes hypoth\`eses que le th\'eor\`eme \ref{theoreme formule en car 0}, on a 
\[\brnral V\subset H^1(L(\zeta_n)/k,M).\]
\qed
\end{cor}

Ce corollaire nous dit que le groupe de Brauer est au moins calculable en temps fini.

\begin{rem}
Avec le but de r\'eduire le temps de ces calculs (en particulier pour pouvoir les faire \`a la main), il est possible de d\'emontrer que
sous les m\^emes hypoth\`eses du th\'eor\`eme, pour $\sigma,\tau\in\Gamma_k$, on a
\[a_\sigma(N_\sigma(b))=a_\tau(N_\tau(b))=1\quad\forall\,\, b\in G\quad\Rightarrow\quad a_{\sigma\tau}(N_{\sigma\tau}(b))=1\quad\forall\,\, b\in G,\]
d\`es que ces deux \'el\'ements \emph{commutent entre eux} et que leur produit n'est pas d'ordre fini.\footnote{Le th\'eor\`eme d'Artin-Schreier nous dit par
ailleurs que, \`a conjugaison pr\`es, il y a au plus un seul \'el\'ement d'ordre fini dans $\Gamma_k$ dont l'ordre est 2 et qui correspondrait \`a une
``conjugaison complexe''.} De l\`a on d\'eduit par exemple que, lorsque $L(\zeta_n)/k$ est une extension ab\'elienne, il suffit de v\'erifier la propri\'et\'e
\eqref{equation formule brnral en car 0} pour une famille de g\'en\'erateurs de $\Gamma_{L(\zeta_n)/k}$. On fera attention par contre au fait qu'il n'est
pas vrai en g\'en\'eral que, pour $\sigma,\tau\in\Gamma_k$ \emph{quelconques} on ait cette propri\'et\'e. En particulier, on ne peut pas restreindre
les calculs \`a un sous-ensemble engendrant $\Gamma_{L(\zeta_n)/k}$ lorsque ce groupe n'est pas ab\'elien.\\
\end{rem}

Pour terminer cette section, on remarque que les calculs que l'on a faits \`a la fin de la section pr\'ec\'edente sont tout \`a fait applicables
aux applications $\varphi_\sigma$ et $N_\sigma$. En effet, notons toujours $\gamma$ le sous-groupe pro-cyclique engendr\'e par $\sigma\in\Gamma_k$
et soit
\[M_\sigma:=\{c\in M:\, \exists a\in Z^1(\gamma,M),\,a_\sigma=c\}.\]
(Ce groupe co\"\i ncide avec $M$ tout entier lorsque l'ordre de $\sigma$ est infini). Il est facile alors de voir que l'on a une application surjective
$M_\sigma\to H^1(\gamma,M)$ dont le noyau correspond aux \'el\'ements de la forme $c\asd{\sigma}{}{-1}{}{c}$ avec $c\in M$. On peut montrer tout
aussi facilement que ce noyau est l'orthogonal de $(G^\ab)^{\varphi_\sigma}$ (cf. le lemme \ref{lemme calcul de N perp}) et en particulier que,
si l'on note $\mathcal{N}_\sigma(G)$ le sous-groupe de $(G^\ab)^{\varphi_\sigma}$ engendr\'e par l'image de $N_\sigma$, alors le groupe
$H^1(\gamma,\pic X_{\bar k})$ est isomorphe \`a $(G^\ab)^{\varphi_\sigma}/\mathcal{N}_\sigma$ (cf. la proposition \ref{proposition calcul de brnral}).
On obtient alors les corollaires suivants du th\'eor\`eme \ref{theoreme formule en car 0}.

\begin{cor}\label{corollaire calcul de brnral en car 0}
Soit $G$ un $k$-groupe et $L\supset k$ une extension d\'eployant $G$ telle que $L(\zeta_n)/k$ est une extension cyclique (un tel corps $L$ existe
par exemple lorsque $G$ est un $p$-groupe constant avec $p$ impair ou bien lorsque $k=\mathbb{R}$). Alors
$\brnral V\cong (G^\ab)^{\varphi_\sigma}/\mathcal{N}_\sigma$, o\`u $\sigma$ est le g\'en\'erateur du groupe $\Gamma_{L(\zeta_n)/k}$.
\qed
\end{cor}

\begin{cor}\label{corollaire brnral egal au sha1 en car 0}
Soit $G$ un $k$-groupe tel que $(G^\ab)^{\varphi_\sigma}$ est engendr\'e par les $\sigma$-normes pour tout $\sigma\in\Gamma_k$ et soit
$V=G\backslash G'$ avec $G'$ semi-simple simplement connexe. Alors $\brnral V=\Sha^1_{\mathrm{cyc}}(k,M)$.
\end{cor}

On rappelle que $\Sha^1_{\mathrm{cyc}} (k,M)$ est le sous-groupe de $H^1(k,M)$ des \'el\'ements dont la restriction \`a $H^1(\gamma,M)$ est nulle
pour tout sous-groupe ferm\'e pro-cyclique $\gamma$ de $\Gamma_k$. On voit rapidement alors du lemme \ref{lemme brnral <<egal au sha1>>} que l'on a toujours
$\Sha^1_{\mathrm{cyc}}(k,M)\subset\brnral V$. En effet, on sait que $\brnral V=H^1(k,\pic X_{\bar k})$ et il est \'evident que,
pour $\alpha\in\Sha^1_{\mathrm{cyc}}(k,M)$, on a que son image dans $H^1(\gamma,M)$ provient de $H^1(\gamma,\pic X_{\bar k})$ car elle est nulle.

\begin{proof}
En effet, un tel groupe $G$ v\'erifie que $H^1(\gamma,\pic X_{\bar k})=0$ pour tout sous-groupe pro-cyclique de $\Gamma_k$. Le
lemme \ref{lemme brnral <<egal au sha1>>} nous permet alors de conclure.
\end{proof}

\subsection{Une formule cohomologique sur un corps local}\label{section formule cohomologique corps local}
La version du th\'eor\`eme \ref{theoreme sans compactification lisse corps fini} pour un corps de nombres $k$, mise au point par Harari dans
\cite{HarariDuke}, nous dit que pour comprendre le groupe de Brauer non ramifi\'e d'une $k$-vari\'et\'e $V$, il faut comprendre
comment le groupe de Brauer se comporte par rapport aux points locaux. Il est alors int\'eressant de se demander, pour une vari\'et\'e $V$ \emph{d\'efinie
sur un corps local}, quelle est la d\'ependance du groupe $\brnr V$ vis-\`a-vis de ses points.

La proposition \ref{proposition lemme cohomologique} nous a permis de comprendre ce comportement \`a partir de la cohomologie galoisienne des $k$-groupes $G$,
$G^\ab$ et $M$ dans le cas o\`u $V$ est un espace homog\`ene \`a stabilisateur fini, bien que seulement pour la partie alg\'ebrique de $\br V$. Soit donc $K$ un corps
local de caract\'eristique $0$, i.e. une extension finie de $\q_p$. On peut alors essayer d'obtenir une description du ``comportement cohomologique'' des
\'el\'ements du groupe $\brnral V$ pour une $K$-vari\'et\'e $V$ \`a partir de la formule d\'evelopp\'ee dans la section pr\'ec\'edente. Ainsi,
toujours gr\^ace \`a la proposition \ref{proposition lemme cohomologique}, on aura une description de la d\'ependance du groupe $\brnral V$ vis-\`a-vis des
$K$-points de $V$.\\

On dira qu'un $K$-groupe fini $G$ est \emph{non ramifi\'e} s'il est d\'eploy\'e par une extension non ramifi\'ee de $K$. On a alors le r\'esultat suivant :

\begin{thm}\label{theoreme corps local}
Soit $G$ un $K$-groupe fini non ramifi\'e d'ordre premier \`a la caract\'eristique r\'esiduelle, plong\'e dans $G'$ semi-simple simplement connexe et
soit $V=G\backslash G'$. Alors $\alpha\in \bral V$ appartient \`a $\brnral V$ si et seulement s'il existe une pr\'eimage $\alpha_1\in\brun V$ de $\alpha$
telle que, pour toute extension finie non ramifi\'ee $K'$ de $K$, l'application $V(K')\to\br K'$ induite par $\alpha_1$ est nulle.
\end{thm}

Or, comme on vient de l'avancer, la proposition \ref{proposition lemme cohomologique} nous dit que ce r\'esultat est \'equivalent au r\'esultat suivant.

\begin{pro}\label{proposition cohomologique corps local}
Soit $G$ un $K$-groupe fini non ramifi\'e d'ordre premier \`a la caract\'eristique r\'esiduelle, plong\'e dans $G'$ semi-simple simplement connexe
et soit $V=G\backslash G'$. Soit $M$ le groupe des caract\`eres de $G$. En identifiant alors $\bral V$ avec $H^1(K,M)$, le groupe de Brauer non ramifi\'e
alg\'ebrique $\brnral V$ de $V$ est donn\'e par les \'el\'ements $\alpha\in H^1(K,M)$ v\'erifiant la propri\'et\'e suivante.

Propri\'et\'e (*) : Pour toute extension finie non ramifi\'ee $K'$ de $K$, l'image $\alpha'$ de $\alpha$ dans $H^1(K',M)$ est orthogonale au sous-ensemble
$\im[H^1(K',G)\to H^1(K',G^\ab)]$.
\end{pro}

\begin{rem}
Des r\'esultats semblables peuvent \^etre retrouv\'es dans le cas des corps locaux de caract\'eristique positive via le th\'eor\`eme \ref{theoreme cohomologique}
\`a l'aide du th\'eor\`eme \ref{theoreme invariance brnr} donn\'e en appendice, moyennant l'existence d'une certaine compactification lisse $X$ de $V$.
\end{rem}

\begin{proof}
Soit $\alpha\in\bral V=H^1(K,M)$. Il s'agit de montrer que $\alpha\in\brnral V$ si et seulement si cette classe v\'erifie la propri\'et\'e
(*). Il y a deux cas possibles :\\

\noindent{\bf Cas 1 : $\alpha\not\in H^1_\nr(K,M)$}\\
Soit $L/K$ une extension finie non ramifi\'ee d\'eployant $G$ et soit $n=\exp(G)$. Le corollaire \ref{corollaire brnral se calcule dans une extension finie}
nous dit que l'on a $\brnral V\subset H^1(L(\zeta_n)/K,M)$. Or, comme le cardinal de $G$ est premier \`a $p$, on sait que l'extension $L(\zeta_n)/K$ est
non ramifi\'ee, donc on a en particulier que $\brnral V\subset H^1_\nr(K,M)$.

Si alors $\alpha\not\in H^1_\nr(K,M)$, d'apr\`es ce qui pr\'ec\`ede, on a que $\alpha\not\in\brnral V$ et
il faut alors d\'emontrer que $\alpha$ ne v\'erifie pas la propri\'et\'e (*). Or, puisque $H^1_\nr(K,M)^\perp=H^1_\nr(K,G^\ab)$
pour la dualit\'e locale, on voit imm\'ediatement qu'il existe $\beta\in H^1_\nr(K,G^\ab)$ tel que $\alpha\cup\beta\neq 0$. D'autre part, en remarquant que les
ensembles $H^1_\nr$ sont isomorphes aux ensembles de cohomologie sur le corps r\'esiduel de $K$ et qu'un tel corps est de dimension cohomologique 1, il est facile
de voir que l'application $H^1_\nr(K,G)\to H^1_\nr(K,G^\ab)$ est surjective  (cf. par exemple \cite[III, \S2.4 Corollary 2]{SerreCohGal}), d'o\`u l'on tire que
$\alpha$ n'est pas orthogonal \`a $\im[H^1(K,G)\to H^1(K,G^\ab)]$.\\

\noindent{\bf Cas 2 : $\alpha\in H^1_\nr(K,M)$}\\
Pour traiter ce cas on rappelle que, de m\^eme que pour les corps locaux de caract\'eristique $p$ dans la section
\ref{section formule algebrique}, le sous-groupe de ramification sauvage $S\subset \Gamma_K$ ne joue pas de r\^ole dans la cohomologie galoisienne des groupes $G$,
$G^\ab$ et $M$ puisque leur ordre est premier \`a $p$ et $S$ est un pro-$p$-groupe. Autrement dit, si l'on pose $D=\Gamma_K/S$, on a $H^1(K,G)=H^1(D,G)$. De plus,
d'apr\`es \cite[Theorem 7.5.3]{NSW}, on sait que $D$ est un groupe profini \`a deux g\'en\'erateurs $\sigma,\tau$ avec la seule relation
$\sigma\tau\sigma^{-1}=\tau^{q_0}$, o\`u $q_0$ est le cardinal du corps r\'esiduel de $K$. On sait aussi que l'on a une suite exacte
\[1\to T\to D\to\Gamma_1\to 1,\]
o\`u $T$ est le sous-groupe profini engend\'re par $\tau$ et correspond au groupe de ramification mod\'er\'ee, tandis que $\Gamma_1$ correspond au groupe de
Galois de l'extension non ramifi\'ee maximale $K^\nr/K$ ainsi qu'au groupe de Galois absolu du corps r\'esiduel de $K$. Il est engendr\'e par l'image de
$\sigma$, laquelle correspond au $q_0$-Frobenius sur le corps r\'esiduel et que l'on note toujours $\sigma$ par abus. De m\^eme, si l'on note $D_m$ le
sous-groupe ferm\'e d'indice $m$ de $D$ engendr\'e par $\sigma^m$ et $\tau$, et aussi $\Gamma_m$ le quotient $D_m/T$, on voit que l'on a les \'egalit\'es
\[H^1(K_m,G)=H^1(D_m,G)\quad\text{et}\quad H^1_\nr(K_m,G)=H^1(\Gamma_m,G),\]
o\`u $K_m$ est la seule extension non ramifi\'e de $K$ de degr\'e $m$. Il est par ailleurs \'evident qu'il en va de m\^eme pour $M$ et $G^\ab$.

\begin{lem}\label{lemme intermediaire proposition cohomologique locale}
Il existe, pour tout $m\geq 1$, un accouplement parfait
\[H^1(\Gamma_m,M)\times \hom(T,G^\ab)^{\Gamma_m}\to H^2(K_m,\mu_n)\hookrightarrow \q/\z,\]
induit par la dualit\'e locale. De plus, un \'el\'ement $\alpha\in H^1_\nr(K,M)=H^1(\Gamma_1,M)\subset H^1(K,M)$ v\'erifie la propri\'et\'e
(*) si et seulement si, pour tout $m\geq 1$, son image $\alpha_m$ dans $H^1(K_m,M)$ est orthogonale
\`a $\im[H^1(D_m,G)\to\hom(T,G^\ab)^{\Gamma_m}]$.
\end{lem}

\begin{proof}
On rappelle que la suite des premiers termes issue de la suite spectrale de Hochschild-Serre pour ce quotient
est la suivante :
\[1\to H^1(\Gamma_m,G^\ab)\to H^1(D_m,G^\ab)\to \hom(T,G^\ab)^{\Gamma_m}\to 1.\]
En effet, $H^1(T,G^\ab)=\hom(T,G^\ab)$ et on sait que $H^2(\Gamma_m,M)$ (qui correspondrait au quatri\`eme terme de la suite \`a
$5$ termes) est nul car la dimension cohomologique de $\Gamma_m\cong\hat\z$ est $1$. L'\'egalit\'e $H^1_\nr(K_m,M)=H^1_\nr(K_m,G^\ab)^\perp$ pour la dualit\'e locale
nous dit alors que l'on a l'accouplement parfait induit par la dualit\'e locale
\[H^1(\Gamma_m,M)\times \hom(T,G^\ab)^{\Gamma_m}\to H^2(K_m,\mu_n)\hookrightarrow \q/\z,\]
et il est \'evident, par la compatibilit\'e de ces deux accouplements, que $\alpha_m\in H^1(\Gamma_m,M)$ est orthogonale \`a $\im[H^1(K_m,G)\to H^1(K_m,G^\ab)]$ pour la
dualit\'e locale si et seulement s'il est orthogonale \`a $\im[H^1(D_m,G)\to\hom(T,G^\ab)^{\Gamma_m}]$ pour l'accouplement ci-dessus. La deuxi\`eme affirmation
du lemme en d\'ecoule car les $K_m$ correspondent \`a toutes les extensions non ramifi\'ees de $K$.
\end{proof}

On remarque maintenant que l'application $q:\Gamma_K\to\z/n\z$ se factorise par $\Gamma_1$, ce qui nous dit que $q(\sigma)$ a bien un sens. De plus, puisque $\sigma$
agit pr\'ecis\'ement comme le $q_0$-Frobenius, il est facile de voir que l'on a $q(\sigma)=q_0$, d'o\`u l'on se permet de noter tout simplement $q$ au lieu de $q_0$
ou de $q(\sigma)$, pour ne pas alourdir les notations. Il est par ailleurs \'evident aussi que l'on a $q(\sigma^m)=q^m$ et on fait donc le m\^eme raccourci de notation.
Ceci nous permet de faire l'analogue de ce qu'on a fait dans le cas des corps finis. En effet, l'application $\varphi_\sigma$ n'est alors rien d'autre que
l'application $\varphi_q$ utilis\'ee dans le cas des corps finis et de m\^eme on aura $\varphi_{\sigma^m}=\varphi_q^m$ et $N_\sigma=N_q$.

Cela \'etant dit, on peut rappeler la d\'efinition d'un \'el\'ement $(q,m)$-relevable (d\'efinition \ref{definition (q,n)-relevable}) :
c'est un \'el\'ement $\b\in G^\ab$ tel qu'il existe $b\in G$ relevant $\b$ et tel que $\varphi_q^m(b)$ soit conjugu\'e \`a $b$. Puisque ceci
entra\^\i ne en particulier que $\varphi_q^m(\b)=\b$, on note $(G^\ab)^{\varphi_q^m}$ le sous-groupe de $G^\ab(\bar K)$ des \'el\'ements invariants par
$\varphi_q^m$ et $G^\ab_{q,m}\subset(G^\ab)^{\varphi_q^m}$ l'ensemble des \'el\'ements $(q,m)$-relevables.

De m\^eme que dans le cas des corps finis (cf. section \ref{section formule algebrique}), on a un isomorphisme \'evident
$\hom(T,G^\ab)\xrightarrow{\sim}G^\ab(\bar K)$ qui \`a un morphisme lui associe l'image de $\tau$. Ce morphisme identifie $\hom(T,G^\ab)^{\Gamma_m}$ avec
$(G^\ab)^{\varphi_q^m}$. Aussi, en suivant la preuve de la proposition \ref{proposition traduction du theoreme cohomologique}, on voit que l'image
de $H^1(D_m,G)$ dans $\hom(T,G^\ab)^{\Gamma_m}$ est identifi\'ee avec $G^\ab_{q,m}$. Puis, les calculs faits dans les lemmes \ref{lemme calcul de a cup b}
et \ref{lemme changement de H2 pour mu} nous disent que, pour tout $m$, l'orthogonalit\'e de $\alpha_m$ avec $\im[H^1(D_m,G)\to\hom(T,G^\ab)^{\Gamma_m}]$ est
\'equivalente avec l'orthogonalit\'e de $a_{\sigma^m}$ avec $G^\ab_{q,m}$, o\`u $a$ est un cocycle repr\'esentant $\alpha$.\\

Tout cela \'etant dit, on peut enfin terminer la d\'emonstration de la proposition \ref{proposition cohomologique corps local}.\\

Dans un premier temps, si $\alpha\in H^1_\nr(K,M)$ v\'erifie la propri\'et\'e (*), on sait d'apr\`es les lemmes
\ref{lemme intermediaire proposition cohomologique locale}, \ref{lemme calcul de a sigma n} et ce qui pr\'ec\`ede, que pour un cocycle $a$ repr\'esentant
$\alpha$ on a
\[a_{\sigma^m}(\b)=a_{\sigma}\left(\prod_{i=0}^{m-1}\varphi_q^{i}(\b)\right)=1,\]
pour tout $\b\in G^\ab_{q,m}$ et tout $m\in\n$. En particulier, on a $a_\sigma(N_q(b))=1$ pour tout $b\in G$, ce qui suffit pour \'etablir que
$\alpha\in\brnral V$ d'apr\`es la remarque apr\`es le corollaire \ref{corollaire brnral se calcule dans une extension finie}.\\

Dans le sens inverse, si $\alpha\in\brnral V$ est repr\'esent\'e par $a\in Z^1(\Gamma_1,M)$, on a $a_{\sigma}(N_q(b))=1$ pour tout $b\in G$.
Soient $m\in\n$, $\b\in G^\ab_{q,m}$ et $b\in G$ une pr\'eimage de $\b$. Il est facile de voir que $n_{\sigma,b}$ divise $m$ et alors,
d'apr\`es le lemme \ref{lemme calcul de a sigma n}, on a
\begin{multline*}
a_{\sigma^{m}}(\b)=a_{\sigma^{ln_{\sigma,b}}}(\b)=a_{\sigma}\left(\prod_{i=0}^{ln_{\sigma,b}-1}\varphi_q^{i}(\b)\right)=
a_{\sigma}\left(\prod_{i=0}^{n_{\sigma,b}-1}\prod_{j=0}^{l}\varphi_q^{jn_{\sigma,b}+i}(\b)\right)=\\
a_{\sigma}\left(\prod_{i=0}^{n_{\sigma,b}-1}\prod_{j=0}^{l}\varphi_q^{i}(\b)\right)=a_{\sigma}\left(\prod_{i=0}^{n_{\sigma,b}-1}
\varphi_q^{i}(\b)^l\right)=a_\sigma(N_q(b))^l=1,
\end{multline*}
car on rappelle que $\varphi_q^{n_{\sigma,b}}(\b)=\varphi_\sigma^{n_{\sigma,b}}(\b)=\b$. Le lemme \ref{lemme changement de H2 pour mu} nous
dit alors que $\alpha_m$ est orthogonal \`a tout $\b\in G^\ab_{q,m}\cong \hom(T,G^\ab)^{\Gamma_m}$, d'o\`u l'orthogonalit\'e de $\alpha_m$ avec
$\im[H^1(K_m,G)\to H^1(K_m,G^\ab)]$ via le lemme \ref{lemme intermediaire proposition cohomologique locale}.
\end{proof}

\begin{rem}
Il est bien n\'ecessaire de consid\'erer les extensions non ramifi\'ees de $K$ et non seulement $K$ tout seul. En effet, demander seulement
l'orthogonalit\'e de $\alpha\in H^1(K,M)$ avec $\im[H^1(K,G)\to H^1(K,G^\ab)]$ revient \`a demander seulement l'orthogonalit\'e de $a_\sigma$ avec
$G^\ab_{q,1}$, alors que l'ensemble $N_\sigma(G)\subset (G^\ab)^{\varphi_q}$ est \`a priori beaucoup plus grand : pour un \'el\'ement $b\in G$ tel
que $\varphi_q^m(b)$ est conjugu\'e \`a $b$, on ne sait pas \`a priori si le produit $b_0:=\prod_{i=0}^{m-1}\varphi_q^i(b)$ est tel que $\varphi_q(b_0)$
est conjugu\'e \`a $b_0$, d'o\`u le fait que la $\sigma$-norme de $b$ n'a aucune raison \`a priori d'\^etre $(q,1)$-relevable.

Si l'on peut objecter que le fait d'avoir $N_\sigma(G)\supsetneq G^\ab_{q,1}$ n'entra\^\i ne pas pour autant que les groupes qu'ils engendrent
soient diff\'erents, il suffit de remarquer que l'on peut construire
des exemples explicites (comme ceux donn\'es dans la proposition \ref{proposition exemple Demarche en car 0}) o\`u le groupe engendr\'e par $N_\sigma(G)$
est strictement plus grand que celui engendr\'e par $G^\ab_{q,1}$.
\end{rem}

Pour finir, on donne un corollaire de la proposition \ref{proposition cohomologique corps local} dans le cas ab\'elien.

\begin{cor}
Soit $G$ un $K$-groupe fini ab\'elien non ramifi\'e, d'ordre premier \`a la caract\'eristique r\'esiduelle, plong\'e dans $G'$ semi-simple simplement connexe
et soit $V=G\backslash G'$. Alors le groupe de Brauer non ramifi\'e alg\'ebrique $\brnral V$ de $V$ est trivial. 
\end{cor}

\begin{proof}
Puisque dans ce cas $G=G^\ab$, il est \'evident que $\im[H^1(K,G)\to H^1(K,G^\ab)]$ correspond au groupe $H^1(K,G)$ tout entier. Ainsi, on voit qu'un \'el\'ement
$\alpha\in\brnral V\subset H^1(K,M)$ doit \^etre orthogonal \`a tout le groupe $H^1(K,G)$. L'accouplement \'etant parfait, on voit que $\alpha$ doit \^etre nul.
\end{proof}

\begin{rem}
Ce r\'esultat n'est pas, \`a notre connaissance, facile \`a d\'emontrer avec ce qui \'etait d\'ej\`a connu dans le cas ab\'elien. En effet, vue la preuve de la
proposition \ref{proposition cohomologique corps local}, on peut se rendre compte qu'un \'el\'ement cl\'e de cette d\'emonstration est le corollaire
\ref{corollaire brnral se calcule dans une extension finie}, lequel nous permet de d\'eduire l'inclusion $\brnral V\subset H^1_\nr(K,M)$. Une fois cette
inclusion admise, le r\'esultat d\'ecoule facilement de l'\'egalit\'e (d\'ej\`a connue) $\brnral V=\Sha^1_\mathrm{cyc}(K,M)$, mais on ne voit pas une voie
``facile'' pour d\'emontrer cette inclusion, m\^eme dans le cas ab\'elien.
\end{rem}

\section{Application des r\'esultats}\label{section applications}\numberwithin{equation}{section}

\subsection{Calcul explicite de certains $\brnral$}
Il s'agit dans cette section de calculer explicitement quelques groupes de Brauer non ramifi\'es alg\'ebriques pour montrer l'utilit\'e des
formules d\'evelopp\'ees dans les sections pr\'ec\'edentes.\\

Tant sur un corps fini que sur un corps de caract\'eristique $0$, il est int\'eressant de comprendre quand deux espaces homog\`enes a priori
diff\'erents en leur structure peuvent avoir un m\^eme groupe de Brauer non ramifi\'e alg\'ebrique. On a d\'ej\`a remarqu\'e au d\'ebut de la section
\ref{section formules} que ce groupe ne d\'epend que du stabilisateur, mais on peut aller encore un peu plus loin. En effet, on peut
imaginer des espaces homog\`enes qui, vis-\`a-vis de notre formule, se comportent de la m\^eme fa\c con tout en \'etant essentiellement diff\'erents.
La d\'efinition qui suit va donc dans ce sens.

Pour un $k$-groupe fini $G$ avec $k$ un corps quelconque, on peut consid\'erer le morphisme $\kappa_G:\Gamma_k\to\out (G(\bar k))$ induit par sa
structure sous-jacente de $\Gamma_k$-groupe, o\`u l'on rappelle que $\out(G(\bar k)):=\aut(G(\bar k))/\int(G(\bar k))$ est le groupe des
automorphismes ext\'erieurs de $G(\bar k)$. Si l'on a un deuxi\`eme $k$-groupe $H$ et un $\bar k$-isomorphisme $f:G_{\bar k}\to H_{\bar k}$, il
induit les isomorphismes de groupes abstraits
\begin{align*}
f_1 &:G(\bar k)\xrightarrow{\sim} H(\bar k), \\
f_2 &:\aut (G(\bar k))\xrightarrow{\sim} \aut (H(\bar k)) :\phi \mapsto f_1\circ\phi\circ f_1^{-1}, \\
f_3 &:\out (G(\bar k))\xrightarrow{\sim} \out (H(\bar k)) :\bar\phi\mapsto \overline{f_2(\phi)}.
\end{align*}

\begin{defi}\label{definition ext-isomorphe}
Soient $G$ et $H$ des $k$-groupes finis avec $k$ un corps quelconque et soient $\kappa_G$ et $\kappa_H$ les morphismes induits comme ci-dessus.
Le $k$-groupe $H$ est dit \emph{ext-isomorphe} \`a $G$ s'il existe un $\bar k$-isomorphisme $f:G_{\bar k}\to H_{\bar k}$ (induisant les isomorphismes $f_1$, $f_2$
et $f_3$ ci-dessus) tel que $\kappa_H=f_3\circ\kappa_G$.

On dit qu'un groupe fini $G$ est \emph{ext-constant} s'il est ext-isomorphe \`a un groupe constant ou, de fa\c con \'equivalente, si $\Gamma_k$
agit sur $G(\bar k)$ par automorphismes int\'erieurs.
\end{defi}

On d\'emontrera par la suite que deux $k$-groupes ext-isomorphes $G$ et $H$ se comportent effectivement de la m\^eme fa\c con par rapport aux formules
d\'evelopp\'ees dans la section pr\'ec\'edente, tant sur un corps fini que sur un corps de caract\'eristique 0. En particulier, \`a chaque
fois qu'on aura un groupe $G$ ext-constant, on pourra supposer qu'il est constant pour calculer le groupe $\brnral V$ d'un espace homog\`ene $V$ sous $G'$
semi-simple simplement connexe \`a stabilisateur $G$. On en d\'eduit par exemple que si $G(\bar k)$ est le groupe $S_n$ des permutations de $n$
\'el\'ements, avec $n\neq 6$, on pourra toujours supposer qu'il est constant (on rappelle que $S_n$ n'a pas d'automorphisme ext\'erieur pour $n\neq 6$).

\subsubsection{Cas d'un corps fini}\label{section calculs corps finis}
On revient aux notations de la section \ref{section formule algebrique}. En particulier, on note $k=\f_q$. Dans cette section, le groupe
$G$ sera toujours suppos\'e d'ordre premier \`a $q=p^r$, donc il sera toujours lisse.\\

La similitude entre les formules \ref{theoreme formule sur fq} et \ref{theoreme formule en car 0}, ou encore le fait que la formule cohomologique
(th\'eor\`eme \ref{theoreme cohomologique}) soit de nature arithm\'etique, de m\^eme que celle sur un corps de nombres (cf.
\cite{HarariBulletinSMF}, Proposition 4) ou sur un corps global de caract\'eristique positive (cf. \cite[Th\'eor\`eme 3.1]{GLABrnr}), permettent de
s'attendre quelques similitudes au niveau du groupe $\brnral$ entre les $k$-vari\'et\'es et ses analogues sur les corps globaux. Ces similitudes
se montrent en effet avec les r\'esultats qui suivent, lesquels ont \'et\'e inspir\'es des travaux de Demarche dans le cas d'un corps de nombres
(cf. \cite{Demarche}).\\

Il est d\'ej\`a connu depuis quelque temps que le groupe $\brnral V$ s'annule lorsque $V=G\backslash G'$ avec $G$ ab\'elien (cf. par exemple
\cite[Th\'eor\`eme 7.6]{BDH}). On commence donc en retrouvant ce r\'esultat avec notre formule.

\begin{pro}\label{proposition nullite du brnral si G abelien}
Soit $G$ un $k$-groupe ab\'elien d'ordre premier \`a $p$ que l'on plonge dans $G'$ semi-simple simplement connexe. Soit $V=G\backslash G'$. Alors on a
$\brnral V=0$.
\end{pro}

\begin{proof}
D'apr\`es le corollaire \ref{corollaire calcul de brnral}, il suffit de d\'emontrer que $N_q(G)$ engendre tout le groupe $(G^\ab)^{\varphi_q}$.
Or cela est \'evident puisque, comme $G=G^\ab$, la $q$-norme $N_q(b)$ d'un \'el\'ement $b\in G$ qui est d\'ej\`a dans
$G^{\varphi_q}=(G^\ab)^{\varphi_q}$ est tout simplement $b$ lui-m\^eme par d\'efinition de $N_q(b)$.
\end{proof}

On d\'emontre maintenant que deux groupes ext-isomorphes se comportent en effet de la m\^eme fa\c con vis-\`a-vis de notre formule.

\begin{pro}\label{proposition ext-constant = constant}
Soient $G$ et $H$ deux $k$-groupes finis ext-isomorphes d'ordre premier \`a $p$ plong\'es respectivement dans $G'$ et $H'$, des $k$-groupes semi-simples
simplement connexes, et soient $V=G\backslash G'$ et $W=H\backslash H'$. Alors on a $\brnral V\cong\brnral W$.
\end{pro}

\begin{proof}
Soit $f:G_{\bar k}\to H_{\bar k}$ un $\bar k$-isomorphisme tel que $\kappa_H=f_3\circ\kappa_G$. Il induit alors un $\bar k$-isomorphisme $f^\ab:G^\ab_{\bar k}\to H^\ab_{\bar k}$ qui est
en fait d\'efini sur $k$ car $\kappa_H=f_3\circ\kappa_G$ entra\^\i ne $\kappa_{H^\ab}=f_3^\ab\circ\kappa_{G^\ab}=f_2^\ab\circ\kappa_{G^\ab}$ (on rappelle que
$\aut (G^\ab(\bar k))=\out (G^\ab(\bar k))$ et on en a de m\^eme pour $H^\ab$) et alors $f_1^\ab$ est un isomorphisme de $\Gamma_k$-modules .
Si l'on note alors $M$ et $N$ les groupes des caract\`eres de $G$ et $H$ respectivement, on a un $\Gamma_k$-isomorphisme $N\xrightarrow{\sim} M$ induisant
un isomorphisme $f^\ab_*:H^1(k,N)\xrightarrow{\sim} H^1(k,M)$.

Il suffit alors de consid\'erer la formule du th\'eor\`eme \ref{theoreme formule sur fq}. Soient $\varphi_{G,q}$ et $\varphi_{H,q}$ les applications
respectives d\'efinissant les $q$-normes respectives sur $G(\bar k)$ et $H(\bar k)$. On a
\[\varphi_{H,q}=f_1\circ\mathrm{int}(b_1)\circ\varphi_{G,q}\circ f_1^{-1}\]
pour un certain $b_1\in G$, d'o\`u il est facile de voir que, pour tout $m\geq 1$,
\[\varphi_{H,q}^m=f_1\circ\mathrm{int}(b_m)\circ\varphi_{G,q}^m\circ f_1^{-1},\]
pour certain $b_m\in G$. En particulier,
\[\varphi_{H,q}=f^\ab_1\circ\varphi_{G,q}\circ (f_1^{\ab})^{-1},\]
lorsqu'on consid\`ere leurs restrictions aux ab\'elianis\'es. Enfin, de ce qui pr\'ec\`ede,
on voit facilement que pour tout $b\in G$ et pour tout $m\geq 1$, $\varphi_{G,q}^m(b)$ est conjugu\'e \`a $b$ si et seulement si
$\varphi_{H,q}^m(f_1(b))$ est conjugu\'e \`a $f_1(b)$ aussi. On voit alors que l'on a
\[N_{H,q}=f_1^\ab\circ N_{G,q}\circ f_1^{-1},\]
o\`u $N_{H,q}$ et $N_{G,q}$ sont les $q$-normes respectives, et la formule du th\'eor\`eme \ref{theoreme formule sur fq} donne alors le m\^eme groupe de Brauer non
ramifi\'e alg\'ebrique pour les vari\'et\'es $V$ et $W$.
\end{proof}

Concentrons-nous un instant sur les groupes \emph{ext-constants}. La proposition \ref{proposition ext-constant = constant} nous dit que, comme on
l'avait annonc\'e, \`a chaque fois qu'on aura un tel groupe $G$ dans les calculs qui suivent, on pourra supposer qu'il est constant dans la
d\'emonstration. Cela \'etant dit, on montre les deux r\'esultats suivants, inspir\'es des
travaux de Demarche sur des groupes constants sur un corps de nombres (cf. \cite[\S3, Cor. 1 et \S5, Thm. 5]{Demarche}).

Cela \'etant dit, on montre les deux r\'esultats suivants, inspir\'es des travaux de Demarche sur des groupes constants sur un corps de nombres (cf.
\cite[\S3 Corollaire 1, \S5 Th\'eor\`eme 5]{Demarche}).

\begin{pro}\label{proposition brnr nul si exp G divise q-1}
Soit $G$ un $k$-groupe fini ext-constant d'ordre premier \`a $p$ plong\'e dans $G'$ semi-simple simplement connexe et soit $V=G\backslash G'$. On suppose
que $n=\exp (G)$ divise $q-1$, alors $\brnral V=0$.
\end{pro}

\begin{proof}
On se r\'eduit tout de suite au cas o\`u $G$ est constant gr\^ace \`a la proposition \ref{proposition ext-constant = constant}. Si l'exposant de
$G$ divise $q-1$, alors $b^q=b$ pour tout \'el\'ement $b\in G$. On voit en particulier que l'application $\varphi_q$ est l'identit\'e, ce qui nous dit
que la $q$-norme correspond forc\'ement \`a la projection de $G$ sur $G^\ab$. Le r\'esultat en d\'ecoule compte tenu de la formule du th\'eor\`eme
\ref{theoreme formule sur fq}.
\end{proof}

\begin{pro}\label{proposition brnral nul si exp G premier a q-1}
Soit $G$ un $k$-groupe fini d'ordre premier \`a $p$ plong\'e dans $G'$ semi-simple simplement connexe et soit $V=G\backslash G'$. Soit $L/k$ l'extension
correspondant au noyau du morphisme $\kappa:\Gamma_k\to\out (G(\bar k))$ induit par $G$ et soit $m=[L:k]$. On suppose que $n=\exp (G)$ est premier
\`a $q^m-1$. Alors $\brnral V=0$.
\end{pro}

\begin{proof}
Soit $s$ le $q$-Frobenius. Puisque $s^m\in\Gamma_L$, on sait qu'il agit par automorphisme int\'erieur sur $G(\bar k)$, donc trivialement sur $G^\ab(\bar k)$.
Soit alors $\b\in (G^\ab)^{\varphi_q}$. On a
\[\b^{q^m}=\asd{s^m}{}{q^m}{}{\b}=\varphi_q^m(\b)=\b,\]
ce qui nous dit que l'ordre de $\b$ divise $q^m-1$. Or, l'ordre de $\b$ divise aussi forc\'ement $n$. On en d\'eduit que l'ordre de cet \'el\'ement
est forc\'ement $1$ et donc on a $\b=1$. Il est alors \'evident que les $q$-normes engendrent $(G^\ab)^{\varphi_q}$ car il est trivial, d'o\`u
l'on d\'eduit que $\brnral V=0$ d'apr\`es le corollaire
\ref{corollaire calcul de brnral}.
\end{proof}

On a dit au d\'ebut de la section \ref{section formule algebrique car 0} qu'il existe des exemples de groupes $G$ sur $k=\f_q$ tels que le groupe
$\brnral V$ soit non nul. Les propositions \ref{proposition brnral nul si exp G premier a q-1} et \ref{proposition brnr nul si exp G divise q-1}
nous disent que, si l'on s'attend \`a trouver ce genre de groupes parmi les groupes ext-constants, on doit chercher dans le cas o\`u seulement un
diviseur propre (et non trivial) de l'exposant de $G$ divise $q-1$. C'est effectivement ce qu'on fait pour construire enfin de tels groupes,
toujours inspir\'es d'un exemple de Demarche (cf. \cite[\S6 Proposition 1]{Demarche}) : on aura $n=\exp (G)=\ell^{2m}$ et $(n,q-1)=\ell^m$ pour
un certain nombre premier $\ell$.

\begin{pro}\label{proposition exemple Demarche sur fq}
Soit $k=\f_q$ avec $q$ une puissance de $p$. Soit $\ell\neq 2$ un nombre premier tel que $q-1=\ell^m u$ avec $(\ell,u)=1$ et $m\geq 1$. Soit
\[G:=\langle x,y,z : x^{\ell^{2m}}=y^{\ell^{2m}}=z^{\ell^{2m}}=1;[x,y]=z^{-\ell^m},[x,z]=[y,z]=1\rangle,\]
o\`u $[x,y]:=xyx^{-1}y^{-1}$. On regarde $G$ comme un $k$-groupe constant et l'on note $V$ la vari\'et\'e $G\backslash G'$ pour un plongement de
$G$ dans $G'$ semi-simple simplement connexe. Alors on a $\brnral V\cong\z/\ell^{\lceil m/2\rceil}\z$, o\`u $\lceil m/2\rceil$ d\'esigne le plus petit
entier sup\'erieur ou \'egal \`a $m/2$.
\end{pro}

\begin{rems}
{\noindent\bf 1.} Le groupe $G$ donn\'e dans l'\'enonc\'e peut aussi \^etre vu comme le produit semi direct
\[(C_{\ell^{2m}}\times C_{\ell^{2m}})\rtimes C_{\ell^{2m}},\]
o\`u $C_{\ell^{2m}}$ d\'esigne le groupe cyclique de cardinal $\ell^{2m}$, d\'efini comme suit : si l'on note $x,z$ et $y$ des g\'en\'erateurs respectifs des trois groupes
$C_{\ell^{2m}}$ apparaissant dans ce produit, l'action de $y$ sur $\langle x\rangle \times\langle z\rangle$ est donn\'ee par
\[y\cdot x=xz^{\ell^m}\quad\text{et}\quad y\cdot z=z.\]

{\noindent\bf 2.} Lorsque $\ell=2$ (par exemple, lorsque $q=3$), le groupe $G$ donne aussi un espace homog\`ene $V$ \`a groupe de Brauer non
ramifi\'e alg\'ebrique non trivial. Si l'on a \'evit\'e de consid\'erer ce cas, c'est tout simplement parce que les calculs sont un peu plus
compliqu\'es.\\

{\noindent\bf 3.} On rappelle que la proposition \ref{proposition ext-constant = constant} nous permettrait d'avoir une action non triviale
du Frobenius sur $G$ (mais ``ext-triviale'' quand-m\^eme) et d'obtenir le m\^eme r\'esultat.\\
\end{rems}

Avant de commencer la d\'emonstration, on donne tout ce qu'il faut savoir sur $G$. Soient $Z$ le centre de $G$ et $G^\der$ le sous-groupe
d\'eriv\'e de $G$. En rappelant que $\varphi_q(\b)=\b^q=\b^{u\ell^m+1}$ pour $\b\in G^\ab$, il est facile de voir que l'on a
\begin{align*}
Z&=\langle x^{\ell^m},y^{\ell^m},z\rangle,\\
G^\der&=\langle z^{\ell^m}\rangle\cong C_{\ell^m},\\
G^\ab&=\langle x,y,z\rangle/\langle z^{\ell^m}\rangle\cong (C_{\ell^{2m}}\times C_{\ell^{2m}})\times C_{\ell^m},\\
(G^\ab)^{\varphi_q}&=\langle \bar x^{\ell^m},\bar y^{\ell^m},\bar z\rangle\cong (C_{\ell^{2m}}^{\ell^m}\times C_{\ell^{2m}}^{\ell^m})\times C_{\ell^m}.
\end{align*}
On remarque enfin que les seules formules qu'il faut garder \`a l'esprit pour suivre le calcul suivant, sont
\[q=u\ell^m+1\quad\text{et}\quad y^sx^r=x^ry^sz^{rs\ell^m},\]
et que $x^{\ell^m}$, $y^{\ell^m}$ et $z$ appartiennent au centre de $G$.

\begin{proof}
Comme le groupe $G$ est constant, on a $\varphi_q(b)=b^q$. On voit alors que l'on a, pour $b=x^ry^sz^t\in G$,
\begin{align*}
\varphi_q^n(b)&=(x^ry^sz^t)^{q^n}\\
&=x^{q^nr}y^{q^ns}z^{q^nt}z^{\frac{q^n(q^n-1)}{2}rs\ell^m}\\
&=x^{q^nr}y^{q^ns}z^{q^nt}\\
&=x^{(nu\ell^m+1)r}y^{(nu\ell^m+1)s}z^{(nu\ell^m+1)t}\\
&=x^ry^sz^tx^{rnu\ell^m}y^{snu\ell^m}z^{tnu\ell^m}\\
&=b\cdot x^{rnu\ell^m}y^{snu\ell^m}z^{tnu\ell^m}.
\end{align*}
o\`u l'on remarque que dans le passage de la deuxi\`eme \`a la troisi\`eme ligne on a utilis\'e le fait que $\ell^m$ divise $q-1$, donc aussi $q^n-1$, et que
$z^{\ell^{2m}}=1$.

Calculons d'autre part ce qui se passe lorsqu'on conjugue un \'el\'ement de $G$. Soit $c=x^{r'}y^{s'}z^{t'}\in G$, alors
\begin{align*}
cbc^{-1}&=(x^{r'}y^{s'}z^{t'})x^ry^sz^t(x^{r'}y^{s'}z^{t'})^{-1}\\
&=x^{r'}y^{s'}z^{t'}x^ry^sz^tz^{-t'}y^{-s'}x^{-r'}\\
&=x^{r'}y^{s'}x^ry^sy^{-s'}x^{-r'}z^t\\
&=x^ry^sz^tz^{(rs'-r's)\ell^m}\\
&=b\cdot z^{(rs'-r's)\ell^m}.
\end{align*}
On voit alors que, pour avoir $\varphi_q^n(b)=cbc^{-1}$, il faut et il suffit que $rnu$ et $snu$ soient des
multiples de $\ell^m$ et que $tnu$ soit congru \`a $rs'-r's$ modulo $\ell^m$. Or, puisque $(\ell,u)=1$, cela se r\'esume en
\begin{equation}\label{equation petit calcul}
rn\equiv sn\equiv 0\mod\ell^m\quad\text{et}\quad tnu\equiv rs'-r's\mod\ell^{m}.
\end{equation}
Supposant ces \'egalit\'es v\'erifi\'ees pour un $n$ minimal (i.e. pour $n=n_b$), on peut calculer la $q$-norme de $b$. On a
\begin{multline*}
N_q(b)=\prod_{i=0}^{n-1}\varphi_q^{i}(\bar b)=\prod_{i=0}^{n-1}\overline{\varphi_q^i(b)}
=\prod_{i=0}^{n-1}\bar b\bar x^{riu\ell^m}\bar y^{siu\ell^m}\bar z^{tiu\ell^m}\\
=\bar b^n\bar x^{r\frac{n(n-1)}{2}u\ell^m}\bar y^{s\frac{n(n-1)}{2}u\ell^m}\bar z^{t\frac{n(n-1)}{2}u\ell^m}=\bar b^n,
\end{multline*}
d'apr\`es les congruences \eqref{equation petit calcul} et le fait que $\bar z^{\ell^m}=1$ et $\ell\neq 2$.\\

Sachant donc que $N_q(b)=\bar x^{n_br}\bar y^{n_bs}\bar z^{n_bt}$, les congruences \eqref{equation petit calcul} nous disent que, pour tout
$b\in G$, $N_q(b)$ appartient au sous-groupe $(G^\ab)^{\varphi_q}\cong (C_{\ell^{2m}}^{\ell^m}\times C_{\ell^{2m}}^{\ell^m})\times C_{\ell^m}$ de $G^\ab$ (chose \'evidente
\`a priori d'apr\`es la th\'eorie, mais plut\^ot rassurante lorsqu'on fait ces calculs). De plus, si l'on note $e$ la valuation
$\ell$-adique de $n_b$ on voit, toujours gr\^ace aux congruences \eqref{equation petit calcul}, que
\[\ell^e|n_bt,\quad\ell^{m-e}|r\quad\text{et}\quad\ell^{m-e}|s.\]
On en d\'eduit que $\ell^{m-e}|(rs'-r's)$ et donc finalement que $\ell^{\max\{m-e,e\}}|n_bt$. Ceci nous dit que $N_q(b)$ appartient au sous-groupe
$H\cong (C_{\ell^{2m}}^{\ell^m}\times C_{\ell^{2m}}^{\ell^m})\times C_{\ell^m}^{\ell^{\lceil m/2\rceil}}$ de $G^\ab$ pour tout $b\in G$.\\

On d\'emontre maintenant que l'image $N_q(G)$ dans $G^\ab$ engendre $H$. On consid\`ere les \'el\'ements $b_1=x$, $b_2=y$ et
$b_3=x^{\ell^{\lfloor m/2\rfloor}}y^{\ell^{\lfloor m/2\rfloor}}z$ de $G$ et on calcule ses $q$-normes.

Pour $b_1$ on a $r=1$ et $s=t=0$, ce qui entra\^\i ne $\ell^m|n_{b_1}$ et on voit rapidement que $n_{b_1}=\ell^m$ car il v\'erifie bien les
congruences \eqref{equation petit calcul} lorsqu'on pose $s'=0$. On a alors $N_q(b_1)=\bar x^{\ell^m}$. Un calcul similaire donne aussi
$N_q(b_2)=\bar y^{\ell^m}$.

Pour $b_3$, on a $r=s=\ell^{\lfloor m/2\rfloor}$ et $t=1$, ce qui entra\^\i ne $\ell^{\lceil m/2\rceil}|n_{b_3}$ et on voit qu'en posant
$r'=0$, $s'=u\ell^{\lceil m/2\rceil-\lfloor m/2\rfloor}$ et $n=\ell^{\lceil m/2\rceil}$, les congruences
\eqref{equation petit calcul} sont bien v\'erifi\'ees. On a alors $n_{b_3}=\ell^{\lceil m/2\rceil}$ et
$N_q(b_3)=\bar x^{\ell^m}\bar y^{\ell^m}\bar z^{\ell^{\lceil m/2\rceil}}$.\\

On voit alors que $N_q(G)$ engendre bien le sous-groupe $H$ de $(G^\ab)^{\varphi_q}$. Donc, d'apr\`es la proposition
\ref{proposition calcul de brnral}, le groupe de Brauer non ramifi\'e alg\'ebrique de la vari\'et\'e $V=G\backslash G'$ est isomorphe, en tant que groupe ab\'elien,
\`a $(G^\ab)^{\varphi_q}/H\cong C_{\ell^m}/C_{\ell^m}^{\ell^{\lceil m/2\rceil}}\cong\z/\ell^{\lceil m/2\rceil}\z$.
\end{proof}

\subsubsection{Cas d'un corps de caract\'eristique 0}
Soit $k$ un corps de caract\'eristique $0$. On retrouve ici les r\'esultats analogues \`a ceux que l'on vient de d\'emontrer pour un
corps fini. On retrouve en particulier une g\'en\'eralisation des r\'esultats de Demarche dans \cite{Demarche}, car $k$ repr\'esente ici un corps
arbitraire, donc en particulier cela pourrait \^etre un corps de nombres.\\

On commence encore une fois avec le cas des groupes ab\'eliens (cas qui est d\'ej\`a connu, cf. par exemple \cite[Th\'eor\`eme 8.1]{BDH}).
On note toujours $n=\mathrm{exp}(G)$ et $\zeta_n$ est une racine primitive $n$-i\`eme de l'unit\'e.

\begin{pro}\label{proposition brnral=sha1 pour G abelien}
Soit $G$ un $k$-groupe fini ab\'elien que l'on plonge dans $G'$ semi-simple simplement connexe. Soit $V=G\backslash G'$. Alors on a
$\brnral V=\Sha^1_{\mathrm{cyc}}(k,M)$.

De plus, si $L$ est une extension d\'eployant $G$ et telle que $L(\zeta_n)/k$ soit une extension cyclique, alors $\brnral V=0$ (un tel corps $L$ existe
par exemple lorsque $G$ est un $p$-groupe constant avec $p$ impair ou lorsque $k=\mathbb{R}$).
\end{pro}

\begin{proof}
D'apr\`es le corollaire \ref{corollaire brnral egal au sha1 en car 0}, il suffit de d\'emontrer que $N_\sigma(G)$ engendre tout le groupe
$(G^\ab)^{\varphi_\sigma}$ pour tout $\sigma\in\Gamma_k$. Or cela est \'evident encore une fois puisque, comme $G=G^\ab$, la $\sigma$-norme
$N_\sigma(b)$ d'un \'el\'ement $b\in G$ qui est d\'ej\`a dans $G^{\varphi_\sigma}=(G^\ab)^{\varphi_\sigma}$ est tout simplement $b$ lui-m\^eme
par d\'efinition de $N_\sigma$.\\

Pour la deuxi\`eme affirmation, le corollaire \ref{corollaire calcul de brnral en car 0} nous dit qu'il suffit de montrer que, si $\sigma$ est un
g\'en\'erateur du groupe de Galois $\Gamma_{L(\zeta_n)/k}$, alors $(G^\ab)^{\varphi_\sigma}=\mathcal{N}_\sigma(G)$. Or, c'est pr\'ecis\'ement ce que l'on
vient de d\'emontrer pour tout \'el\'ement de $\Gamma_k$.
\end{proof}

On d\'emontre aussi dans ce cas que deux groupes ext-isomorphes au sens de la d\'efinition \ref{definition ext-isomorphe} se comportent effectivement
de la m\^eme fa\c con vis-\`a-vis de notre formule.

\begin{pro}\label{proposition ext-constant = constant en car 0}
Soient $G$ et $H$ deux $k$-groupes finis ext-isomorphes plong\'es respectivement dans $G'$ et $H'$, des $k$-groupes semi-simples simplement connexes,
et soient $V=G\backslash G'$ et $W=H\backslash H'$. Alors on a $\brnral V\cong\brnral W$.
\end{pro}

\begin{proof}
Soit $f:G_{\bar k}\to H_{\bar k}$ un $\bar k$-isomorphisme tel que $\kappa_H=f_3\circ\kappa_G$. Il induit alors un $\bar k$-isomorphisme $f^\ab:G^\ab_{\bar k}\to H^\ab_{\bar k}$ qui est
en fait d\'efini sur $k$ car $\kappa_H=f_3\circ\kappa_G$ entra\^\i ne $\kappa_{H^\ab}=f_3^\ab\circ\kappa_{G^\ab}=f_2^\ab\circ\kappa_{G^\ab}$ (on rappelle que
$\aut (G^\ab(\bar k))=\out (G^\ab(\bar k))$ et on en a de m\^eme pour $H^\ab$) et alors $f_1^\ab$ est un isomorphisme de $\Gamma_k$-modules .
Si l'on note alors $M$ et $N$ les groupes des caract\`eres de $G$ et $H$ respectivement, on a un isomorphisme $N\xrightarrow{\sim} M$ induisant
un isomorphisme $f^\ab_*:H^1(k,N)\xrightarrow{\sim} H^1(k,M)$.

Il suffit alors de consid\'erer la formule du th\'eor\`eme \ref{theoreme formule en car 0}. Soit $\sigma\in \Gamma_k$ et soient $\varphi_{G,\sigma}$ et
$\varphi_{H,\sigma}$ les applications d\'efinissant les $\sigma$-normes respectives sur $G(\bar k)$ et $H(\bar k)$. On a
\[\varphi_{H,\sigma}=f_1\circ\mathrm{int}(b_{\sigma,1})\circ\varphi_{G,\sigma}\circ f_1^{-1},\]
pour un certain $b_{\sigma,1}\in G$, d'o\`u il est facile de voir que, pour tout $m\geq 1$,
\[\varphi_{H,\sigma}^m=f_1\circ\mathrm{int}(b_{\sigma,m})\circ\varphi_{G,\sigma}^m\circ f_1^{-1}.\]
pour certain $b_{\sigma,m}\in G$. En particulier,
\[\varphi_{H,\sigma}=f^\ab_1\circ\varphi_{G,\sigma}\circ (f^\ab_1)^{-1},\]
lorsqu'on consid\`ere leurs restrictions aux ab\'elianis\'es. Enfin, de ce qui pr\'ec\`ede, on voit facilement que pour tout \'el\'ement
$b\in G(\bar k)$ et pour tout $m\geq 1$, $\varphi_{G,\sigma}^m(b)$ est conjugu\'e \`a $b$ si et seulement si $\varphi_{H,\sigma}^m(f_1(b))$ est conjugu\'e
\`a $f_1(b)$ aussi. On voit alors que l'on a
\[N_{H,\sigma}=f_1^\ab\circ N_{G,\sigma}\circ f_1^{-1},\]
o\`u $N_{H,\sigma}$ et $N_{G,\sigma}$ sont les $q$-normes respectives, et la formule du th\'eor\`eme
\ref{theoreme formule en car 0} donne alors le m\^eme groupe de Brauer non ramifi\'e alg\'ebrique pour les vari\'et\'es $V$ et $W$.
\end{proof}

On d\'emontre maintenant les analogues en caract\'eristique $0$ des propositions \ref{proposition brnr nul si exp G divise q-1} et
\ref{proposition brnral nul si exp G premier a q-1}, g\'en\'eralisant donc les r\'esultats de Demarche dans \cite{Demarche}.

\begin{pro}\label{proposition brnr nul si exp G divise mu_k}
Soit $G$ un $k$-groupe fini ext-constant plong\'e dans $G'$ semi-simple simplement connexe, $V=G\backslash G'$. Soit $\mu(k)$ l'ensemble des racines
de l'unit\'e dans $k$ et soit $m(k)$ le cardinal de $\mu(k)$. On suppose que $n=\mathrm{exp}(G)$ divise $m(k)$. Alors $\brnral V=0$.
\end{pro}

\begin{proof}
On se r\'eduit tout de suite au cas o\`u $G$ est constant gr\^ace \`a la proposition \ref{proposition ext-constant = constant en car 0}. Le
corollaire \ref{corollaire brnral se calcule dans une extension finie} nous dit alors que $\brnral V\subset H^1(k(\zeta_n)/k,M)$. Or, puisque
$n$ divise $m(k)$, on a alors $\zeta_n\in k$, ce qui nous dit que $H^1(k(\zeta_n)/k,M)=0$.
\end{proof}

\begin{rem}
Cette proposition a toujours un sens lorsque $\mu(k)$ est infini : il suffit de consid\'erer $m(k)$ comme un nombre \emph{surnaturel} (cf.
\cite[I, \S1.3]{SerreCohGal}). On voit en particulier que si $k$ contient toutes les racines de l'unit\'e, alors pour tout groupe ext-constant
$G$ on a $\brnral V=0$.\\
\end{rem}

\begin{pro}\label{proposition brnral nul si exp G premier a mu_k}
Soit $G$ un $k$-groupe fini. On suppose qu'il existe une extension finie $L/k$ d\'eployant $G$ et telle que, si l'on note $\mu(L)$ l'ensemble des racines
de l'unit\'e dans $L$ et $m(L)$ le cardinal de $\mu(L)$, alors $n=\mathrm{exp}(G)$ est premier \`a $m(L)$. Alors, pour tout $k$-groupe $H$ ext-isomorphe
\`a $G$ plong\'e dans $G'$ semi-simple simplement connexe, si l'on note $V=H\backslash G'$, on a $\brnral V=0$.
\end{pro}

\begin{proof}
Il est clair qu'il suffit de montrer la proposition pour $H=G$. Le corollaire \ref{corollaire brnral se calcule dans une extension finie} nous dit que
$\brnral V$ est contenu dans $H^1(L(\zeta_n)/k,M)$. On consid\`ere alors la suite exacte de restriction-inflation suivante :
\[0\to H^1(L/k,M^{\Gamma_{L(\zeta_n)/L}})\to H^1(L(\zeta_n)/k,M)\to H^1(L(\zeta_n)/L,M).\]
Puisque $L$ d\'eploie $G$, on a que $G^\ab(\bar k)$ est un $\Gamma_{L(\zeta_n)/L}$-module avec action triviale, ce qui fait de $M$ un $\Gamma_{L(\zeta_n)/L}$-module
de la forme $\bigoplus_{i}\mu_{n_i}$, avec $\mu_{n_i}\subset\mu_n$. Il est facile de voir alors que $M^{\Gamma_{L(\zeta_n)/L}}=0$, car
$\mu_n\cap L=\{1\}$. Ainsi, on a que $H^1(L/k,M^{\Gamma_{L(\zeta_n)/L}})=0$ et il suffit alors de d\'emontrer que, pour $\alpha\in\brnral V$, son image
dans $H^1(L(\zeta_n)/L,M)$ est triviale.\\

\'Ecrivons $n=p_1^{\alpha_1}\cdots p_r^{\alpha_r}$ avec $p_i\neq p_j$. Puisque $m(L)$ est forc\'ement pair (car $-1\in L$), on a que tous les $p_i$ sont
impairs car $m(L)$ est premier \`a $n$. Le groupe $\Gamma_{L(\zeta_n)/L}$ est alors isomorphe au produit direct des groupes
$\z/p_i^{\alpha_i-1}(p_i-1)\z$. Soient $\sigma_i$ des g\'en\'erateurs respectifs de ces composantes de $\Gamma_{L(\zeta_n)/L}$. Il est facile de voir
alors que l'on a
\[q(\sigma_i)\equiv 1\mod p_j^{\alpha_j},\quad\forall\, j\neq i\quad\text{et}\quad q(\sigma_i)\equiv s_i\mod p_i^{\alpha_i},\]
o\`u $s_i$ est un g\'en\'erateur du groupe multiplicatif $(\z/p_i^{\alpha_i}\z)^*\cong\z/p_i^{\alpha_i-1}(p_i-1)\z$ et $q$ est le morphisme d\'efini
au d\'ebut de la section \ref{section formule algebrique car 0}.

Soit donc $\alpha\in\brnral V$ et soit $a\in Z^1(\Gamma_{L(\zeta_n)/L},M)$ un cocycle repr\'esentant l'image de $\alpha$ dans $H^1(L(\zeta_n)/L,M)$.
Il s'agit de montrer qu'il existe $c_0\in M$ tel que $a_\sigma=c_0\asd{\sigma}{}{-1}{0}{c}$ pour tout $\sigma\in\Gamma_{L(\zeta_n)/L}$. Or, il est
\'evident qu'il suffit de v\'erifier cela pour les $\sigma_i$ car ils engendrent le groupe $\Gamma_{L(\zeta_n)/L}$.
Consid\'erons alors $a_{\sigma_i}$, qui par hypoth\`ese v\'erifie que
\[a_{\sigma_i}(N_{\sigma_i}(b))=1,\quad\forall b\in G.\]
Consid\'erons d'abord un $b\in G$ d'ordre premier \`a $p_i$. Alors, d'apr\`es les propri\'et\'es de $q(\sigma_i)$ ci-dessus, on a
$\varphi_{\sigma_i}(b)=b^{q(\sigma_i)}=b$, ce qui veut dire que l'on a $N_{\sigma_i}(b)=\bar b$. On voit alors que $a_{\sigma_i}$ est orthogonal
\`a toute la $p_j$-partie de $G^\ab$ pour tout $j\neq i$, i.e. $a_{\sigma_i}\in M\{p_i\}$.

Soit $c\in M\{p_i\}$. Puisque $G$ (et par cons\'equent $G^\ab$) est constant en tant que $\Gamma_{L(\zeta_n)/L}$-groupe, on a
\[\asd{\sigma_i}{}{}{}{c}=c^{q(\sigma_i)}=c^{s_i}\quad\text{et}\quad\asd{\sigma_j}{}{}{}{c}=c^{q(\sigma_j)}=c,\quad\forall\, j\neq i.\]
Puisque $s_i$ est un g\'en\'erateur du groupe cyclique $(\z/p_i^{\alpha_i}\z)^*$, on sait qu'il n'est pas congru \`a $1\mod p_i$, car l'ordre d'un tel
\'el\'ement divise $p_i^{\alpha_i-1}$. On a alors que $s_i-1$ est premier \`a $p_i$, donc en particulier il existe $c_i\in M\{p_i\}$ tel que
$c_i^{s_i-1}=a_{\sigma_i}^{-1}$ et alors
\[c_i\asd{\sigma_i}{}{-1}{i}{c}=c_ic_i^{-s_i}=c_i^{1-s_i}=a_{\sigma_i}.\]
De plus, puisque $c_i\in M\{p_i\}$, on a aussi, pour $j\neq i$,
\[c_i\asd{\sigma_j}{}{-1}{i}{c}=c_ic_i^{-1}=1.\]
Soit alors $c_0=\prod_{i=1}^{r}c_i\in M$. On voit facilement que l'on a, pour tout $1\leq i\leq r$,
\[a_{\sigma_i}=c_0\asd{\sigma_i}{}{-1}{0}{c},\]
ce qui conclut.
\end{proof}

Enfin, on reprend une g\'en\'eralisation de l'exemple de Demarche \cite[\S6 Proposition 1]{Demarche}, qui \'etait originellement un contre-exemple
\`a la formule $\brnral V=\Sha^1_{\mathrm{cyc}}(k,M)$ sur un corps de nombres,
\footnote{Dans son article Demarche compare les groupes $\brnral V$ et $\Sha^1_\omega(k,M)$, mais on a d\'ej\`a remarqu\'e que $\Sha^1_\omega(k,M)$
et $\Sha^1_{\mathrm{cyc}}(k,M)$ d\'enotent le m\^eme sous-groupe de $H^1(k,M)$ pour un corps de nombres.}
on calcule son groupe de Brauer et on d\'emontre que l'on a aussi en g\'en\'eral $\brnral V\supsetneq\Sha^1_{\mathrm{cyc}}(k,M)$.

\begin{pro}\label{proposition exemple Demarche en car 0}
Soit $k$ un corps de caract\'eristique $0$ et soit $\ell\neq 2$ un nombre premier tel que la valuation $\ell$-adique du cardinal de $\mu(k)$ soit
$m\geq 1$. Soit
\[G:=\langle x,y,z : x^{\ell^{2m}}=y^{\ell^{2m}}=z^{\ell^{2m}}=1;[x,y]=z^{-\ell^m},[x,z]=[y,z]=1\rangle,\]
o\`u $[x,y]:=xyx^{-1}y^{-1}$. On regarde $G$ comme un $k$-groupe constant et l'on note $V$ la vari\'et\'e $G\backslash G'$ pour un plongement de
$G$ dans $G'$ semi-simple simplement connexe. Soit enfin $M$ le groupe des caract\`eres de $G$. Alors on a
\[\brnral V\cong\z/\ell^{\lfloor m/2\rfloor}\z\quad\text{et}\quad\Sha^1_{\mathrm{cyc}}(k,M)=0.\]
\end{pro}

\begin{proof}
Puisque $G$ est un $\ell$-groupe constant avec $\ell$ impair, le corollaire \ref{corollaire calcul de brnral en car 0} nous dit que le groupe $\brnral V$
est isomorphe \`a $ (G^\ab)^{\varphi_\sigma}/\mathcal{N}_\sigma$, o\`u $\sigma$ est le g\'en\'erateur du groupe $\Gamma_{k(\zeta_{\ell^{2m}})/k}$ et
$\mathcal{N}_\sigma$ est le sous-groupe de $G^\ab$ engendr\'e par les $\sigma$-normes. Il suffit alors de calculer ce quotient.\\

Puisque $\sigma$ est d'ordre $\ell^m$ et qu'il fixe $\zeta_{\ell^m}$ (car $\zeta_{\ell^m}\in k$), on sait que l'on doit avoir
\[q(\sigma)\equiv u\ell^m+1\mod\ell^{2m},\]
pour un certain $u$ avec $(u,\ell)=1$. On a alors que $\varphi_{\sigma}(b)=b^{q(\sigma)}=b^{u\ell^m+1}$ pour $b\in G$. On remarque alors que l'on
est \emph{exactement} dans la m\^eme situation que dans la proposition \ref{proposition exemple Demarche sur fq}, \`a cela pr\`es que l'on note
$\varphi_\sigma$, $n_{\sigma,b}$ et $N_\sigma$ au lieu de $\varphi_q$, $n_b$ et $N_q$. Le calcul d\'ej\`a fait nous dit alors que
$\brnral V\cong\z/\ell^{\lfloor m/2\rfloor}\z$.\\

Pour d\'emontrer finalement que $\Sha^1_{\mathrm{cyc}}(k,M)=0$, il suffit de noter qu'il est contenu dans $\brnral V$ et donc dans
$H^1(k(\zeta_{\ell^{2m}})/k,M)$ d'apr\`es le corollaire \ref{corollaire brnral se calcule dans une extension finie}. Un \'el\'ement qui s'annule en
toute restriction \`a un sous-groupe pro-cyclique s'annule forc\'ement dans $H^1(k(\zeta_{\ell^{2m}})/k,M)$, puisque $k(\zeta_{\ell^{2m}})/k$
est une extension cyclique.
\end{proof}

\subsection{Obstruction de Brauer-Manin \`a l'approximation faible}
Soit $V$ une vari\'et\'e sur un corps de nombres $k$. On d\'efinit le groupe $\Brusse_\omega V$ comme le sous-groupe de $\bral V$ des \'el\'ements
$\alpha$ tels que sa restriction $\alpha_v$ \`a $\bral V_{k_v}$ est nulle pour presque toute place de $k$, ou bien, si l'on rappelle l'identification
de $\bral V$ avec $H^1(k,M)$ pour un espace homog\`ene \`a stabilisateur fini, comme le sous-groupe $\Sha^1_\omega(k,M)$ de $H^1(k,M)$ des
\'el\'ements qui s'annulent en presque toute localisation $H^1(k_v,M)$.

Plus g\'en\'eralement, pour un $\Gamma_k$-module de type fini $M$ avec $k$ un corps quelconque, on a d\'ej\`a d\'efini
$\Sha^1_{\mathrm{cyc}} (k,M)$ comme le sous-groupe de $H^1(k,M)$ des \'el\'ements dont la restriction \`a $H^1(\gamma,M)$ est nulle pour
tout sous-groupe ferm\'e pro-cyclique $\gamma$ de $\Gamma_k$. Le th\'eor\`eme de \v Cebotarev nous permet de d\'emontrer facilement que,
lorsque $k$ est un corps global, on a $\Sha^1_{\mathrm{cyc}} (k,M)=\Sha^1_{\omega} (k,M)$, cf. par exemple \cite[\S 1]{Sansuc81}.

L'int\'er\^et du sous-groupe $\Brusse_\omega V$ (ou $\Sha^1_\omega(k,M)$ dans ce contexte) est que dans le cas o\`u $V$ est un espace
homog\`ene sur un groupe alg\'ebrique connexe \`a stabilisateur connexe ou ab\'elien sur un corps de nombres, Borovoi a d\'emontr\'e que
l'obstruction de Brauer-Manin \`a l'approximation faible associ\'ee \`a ce sous-groupe \'etait la seule, cf.
\cite[Corollary 2.5]{Borovoi96}. De plus, on peut d\'emontrer que pour ces vari\'et\'es on a l'\'egalit\'e $\Brusse_\omega V=\brnral V$
(c'est ce qu'on a fait dans la proposition \ref{proposition brnral=sha1 pour G abelien} pour le cas o\`u $G$ est ab\'elien, cf. \cite[Th\'eor\`eme 8.1]{BDH}
pour le cas g\'en\'eral). La question de savoir si cette obstruction est la seule reste cependant ouverte pour les espaces homog\`enes \`a stabilisateur
fini non ab\'elien et il n'y a pas d'accord parmi les experts sur ce qu'on devrait s'attendre comme r\'eponse.

Dans cette section on donne un exemple de ce que les formules d\'evelopp\'ees ont \`a nous dire sur l'obstruction de Brauer-Manin
\`a l'approximation faible dans ce cas toujours ouvert.\\

La formule \eqref{equation formule brnral en car 0} devient tr\`es simple lorsque $k=\bb{R}$. Ceci nous permet de d\'emontrer que l'obstruction de
Brauer-Manin \`a l'approximation faible associ\'ee au sous-groupe $\brnral$ ne prend pas en compte ce qui se passe aux places r\'eelles. En effet, on a
le r\'esultat suivant.

\begin{pro}\label{proposition BM ne voit pas les places reelles}
Soient $k=\bb{R}$, $G$ un $k$-groupe fini plong\'e dans $G'$ semi-simple simplement connexe et $V=G\backslash G'$. En identifiant $\bral V$ avec $H^1(k,M)$,
tout \'el\'ement $\alpha\in\brnral V$ est orthogonal au sous-ensemble $\im [H^1(k,G)\to H^1(k,G^\ab)]$ de $H^1(k,G^\ab)$.
\end{pro}

\begin{rem}
Ce re\'sultat nous dit en particulier que pour $k$ un corps de nombres, $G$ un $k$-groupe fini plong\'e dans $G'$ semi-simple simplement connexe et
$V=G\backslash G'$, si l'obstruction de Brauer-Manin associ\'ee au sous-groupe $\brnral$ s'av\`ere \^etre la seule pour $V$,
alors $V$ v\'erifie l'approximation r\'eelle (i.e. l'approximation faible en l'ensemble $S=\Omega_\infty$ des places r\'eelles de $k$), comme il sera explicit\'e
dans le corollaire qui suit. Pris dans l'autre sens, ceci veut dire aussi qu'un espace homog\`ene \`a stabilisateur fini ne v\'erifiant pas l'approximation
r\'eelle fournirait un exemple de vari\'et\'e o\`u l'obstruction de Brauer-Manin alg\'ebrique n'est pas la seule.
\end{rem}

\begin{proof}[D\'emonstration de la proposition \ref{proposition BM ne voit pas les places reelles}]
Soit $\sigma$ l'\'el\'ement non nul de $\Gamma_k$. Il est facile de voir que l'on a les isomorphismes
\begin{align*}
Z^1(k,G) \xrightarrow{\sim}  G^{\varphi_\sigma}, &\quad b\mapsto b_\sigma \\
Z^1(k,G^\ab)\xrightarrow{\sim}  (G^\ab)^{\varphi_\sigma}, &\quad \b\mapsto \b_\sigma.
\end{align*}
Soit $g\in G^{\varphi_\sigma}$. Puisque $\varphi_\sigma(g)=g$, on a que $N_\sigma(g)=\bar g\in G^\ab$. En particulier, on voit que l'image de
$N_\sigma$ contient l'image de $Z^1(k,G)$ dans $Z^1(k,G^\ab)$.

Soit alors $\beta\in H^1(k,G^\ab)$ un \'el\'ement dans $\im [H^1(k,G)\to H^1(k,G^\ab)]$. On sait que l'on peut repr\'esenter $\beta$ par un cocycle
$\b$ tel que $\b_\sigma\in N_\sigma(G)$. Pour $\alpha\in\brnral V$, soit $a\in Z^1(k,M)$ un cocycle repr\'esentant $\alpha$. On sait alors d'apr\`es le
th\'eor\`eme \ref{theoreme formule en car 0} que l'on a $a_\sigma(\b_\sigma)=1$. D'autre part, on sait que
$\alpha\cup\beta\in H^2(k,\mu_n)\subset\br k$ est repr\'esent\'e par un cocycle $c$ tel que
\[c_{\sigma,\sigma}=a_\sigma(\asd{\sigma}{}{}{\sigma}{\b})=a_{\sigma}(\b_\sigma^{-1})=(a_\sigma(\b_\sigma))^{-1}=1.\]
(On rappelle que $\b_\sigma\in (G^\ab)^{\varphi_\sigma}$ et alors $\asd{\sigma}{}{}{\sigma}{\b}=\b_\sigma^{-1}$.) Il est clair alors que
$\alpha\cup\beta$ est nul, ce qui nous dit que $\alpha$ est orthogonal au sous-ensemble $\im [H^1(k,G)\to H^1(k,G^\ab)]$ de $H^1(k,G^\ab)$.
\end{proof}

Soit maintenant $k$ un corps de nombres et $V=G\backslash G'$ comme toujours. Ce dernier r\'esultat, joint au th\'eor\`eme
\ref{theoreme corps local} (ou plut\^ot \`a son \'equivalent cohomologique, proposition \ref{proposition cohomologique corps local}), nous permet de
cerner un peu mieux l'ensemble $(\prod_{\Omega_k}V(k_v))^{\brnral}$ des familles des points dans $\prod_{\Omega_k}V(k_v)$ qui sont orthogonales au
groupe $\brnral V$, comme le montre le r\'esultat suivant :

\begin{cor}
Soit $k$ un corps de nombres, $G$ un $k$-groupe fini d\'eploy\'e par $L/k$ plong\'e dans $G'$ semi-simple simplement connexe et $V=G\backslash G'$. Soit enfin
$S$ la r\'eunion de l'ensemble des places (non archim\'ediennes) ramifi\'ees pour l'extension $L/k$ avec l'ensemble des places divisant le cardinal de $G$ et soit
$T=\Omega_k\smallsetminus S$. Alors la projection
\[\left(\prod_{\Omega_k} V(k_v)\right)^{\brnral}\to \prod_{v\in T} V(k_v),\]
est surjective.
\end{cor}

\begin{proof}
Il suffit de noter que d'apr\`es la formule de compatibilit\'e donn\'ee dans la proposition \ref{proposition Skor}, l'ensemble
$(\prod_{\Omega_k}V(k_v))^{\brnral}$ est l'ensemble de familles de points $(P_v)_{\Omega_k}$ telles que
\[\sum_{v\in\Omega_k}\mathrm{inv}_v(\alpha_v\cup [Z](P_v))=0\in\q/\z,\quad\forall \alpha\in\brnral V\subset H^1(k,M),\]
o\`u $\mathrm{inv}_v$ est l'application canonique $\br k_v\to\q/\z$ donn\'ee par la th\'eorie du corps de classes et $[Z](P_v)\in H^1(k,G^\ab)$ est la classe
du $k_v$-torseur sous $G^\ab$ obtenu \`a partir du torseur $Z=G^\der\backslash G'\to V$ sous $G^\ab$ et du $k_v$-point $P_v\in V(k_v)$ (cf. le d\'ebut de la section
\ref{section formules} pour la d\'efinition de $Z$ et la preuve de la proposition \ref{proposition lemme cohomologique} pour comprendre l'usage qu'on fait ici de
la formule de compatibilit\'e). En effet, on voit alors imm\'ediatement \`a partir des propositions \ref{proposition BM ne voit pas les places reelles} et
\ref{proposition cohomologique corps local} qu'on a $\alpha_v\cup [Z](P_v)=0$ pour toute place $v\in T$,
car pour toute telle place non archim\'edienne on a que $G$ est non ramifi\'e en tant que $k_v$-groupe et son cardinal est premier \`a la caract\'eristique
r\'esiduelle. Il est alors \'evident que pour $(\beta_v)_T\in\prod_{v\in T} V(k_v)$, l'\'el\'ement $(0)_S\times (\beta_v)_T\in\prod_{\Omega_k}V(k_v)$ appartient
\`a $(\prod_{\Omega_k}V(k_v))^{\brnral}$, ce qui conclut.
\end{proof}

\begin{rem}
Si ce r\'esultat nous dit que le groupe $\brnral V_{\bb{R}}$ ne joue pas de r\^ole dans l'obstruction de Brauer-Manin, il est important de remarquer qu'il
n'y a aucune raison pour qu'il soit trivial. En effet, le corollaire \ref{corollaire calcul de brnral en car 0} nous dit que $\brnral V_{\bb{R}}$ est isomorphe \`a
$(G^\ab)^{\varphi_\sigma}/\mathcal{N}_\sigma$, o\`u $\sigma$ est l'\'el\'ement non nul de $\Gamma_{\bb R}$ et $\cal{N}_\sigma$ est le sous-groupe engendr\'e par l'image
de la $\sigma$-norme $N_\sigma$, et l'on peut trouver des exemples analogues \`a ceux donn\'es dans la proposition \ref{proposition exemple Demarche en car 0}
pour lesquels ce quotient serait non trivial. Par ailleurs, on voudrait voir dans l'\'enonc\'e de la proposition \ref{proposition BM ne voit pas les places reelles}
une formule cohomologique pour le groupe $\brnral V_{\bb R}$ analogue \`a celle trouv\'ee pour les corps locaux (proposition
\ref{proposition cohomologique corps local}). Pourtant, si l'on regarde de pr\`es le raisonnement qui m\`ene au corollaire \ref{corollaire calcul de brnral en car 0}
et la preuve de la proposition \ref{proposition BM ne voit pas les places reelles}, on voit que cela serait equivalent \`a affirmer que le sous-groupe
$\cal{N}_\sigma$ de $(G^\ab)^{\varphi_\sigma}$ est le m\^eme que celui engendr\'e par l'image de $G^{\varphi_\sigma}$, ce qui n'est pas forc\'ement le cas en
g\'en\'eral, m\^eme si l'on a toujours de fa\c con \'evidente $G^{\varphi_\sigma}\subset N_\sigma(G)$.
\end{rem}

\appendix

\section{``Rigidit\'e'' du groupe de Brauer non ramifi\'e alg\'ebrique}\label{section changement de base}
On donne dans cette section des \'enonc\'es de ``rigidit\'e'' pour les groupes de Picard et de Brauer non ramifi\'e alg\'ebrique d'une
vari\'et\'e s\'eparablement rationnellement connexe.\\

On suppose ici que $k$ est un corps \emph{parfait} de caract\'eristique $p$, donc par exemple un corps fini. On a le r\'esultat suivant pour
le groupe de Picard.

\begin{thm}\label{theoreme pic invariant}
Soit $X$ une $k$-vari\'et\'e propre, lisse et s\'eparablement rationnellement connexe. Soit $K$ une extension de $k$. Alors le changement
de base de $k$ \`a $K$ induit un isomorphisme
$\pic X_{\bar k} \xrightarrow{\sim} \pic X_{\bar K}$. Ce sont des groupes ab\'eliens de type fini et sans $\ell$-torsion pour tout
nombre premier $\ell\neq p$.
\end{thm}

L'\'egalit\'e de ces groupes de Picard donne un peu d'espoir sur l'existence d'une \'egalit\'e entre les groupes de Brauer non ramifi\'e alg\'ebrique de $V$
et $V_{K}$, compte tenu de la formule \eqref{equation brnr pour X} pour une compactification lisse $X$ de $V$ :
\[\bral X\xrightarrow{\sim} H^1(k,\pic X_{\bar k}).\]
C'est effectivement le r\'esultat que l'on trouve en ajoutant l'hypoth\`ese que
$k$ soit alg\'ebriquement ferm\'e dans $K$. Remarquons que, sans cette hypoth\`ese, il est difficile d'esp\'erer un tel comportement : il suffit
de penser \`a une extension alg\'ebrique s\'eparable $K/k$ d\'eployant $\pic X_{\bar k}$. Le groupe $H^1(K,\pic X_{\bar k})$ est alors trivial
(du moins en dehors de la $p$-partie) alors que $H^1(k,\pic X_{\bar k})$ n'a aucune raison de l'\^etre.

\begin{thm}\label{theoreme invariance brnr}
Soit $V$ une $k$-vari\'et\'e telle que $V(k)\neq\emptyset$. On suppose que $V$ admet une compactification lisse $X$ que
l'on suppose en plus s\'eparablement rationnellement connexe (par exemple, lorsque $V$ est $k$-unirationnelle). Soit $K$ une extension de $k$ telle
que $k$ est alg\'ebriquement ferm\'e dans $K$. Alors le morphisme $\brun V\hookrightarrow\brun V_{K}$ induit par le changement de base de $k$ \`a
$K$ induit un isomorphisme entre les groupes de Brauer non ramifi\'e alg\'ebrique
\[\brnral V\{p'\}=\bral X\{p'\}\subset\bral V\quad\text{et}\quad \brnral V_{K}\{p'\}=\bral X_{K}\{p'\}\subset\bral V_{K}.\]
\end{thm}

\begin{proof}[D\'emonstration du Th\'eor\`eme \ref{theoreme pic invariant}]
Soit $L=\bar k$ ou $\bar K$. Montrons d'abord que le groupe $\pic X_L$ est sans $\ell$-torsion
pour tout $\ell\neq p$. La th\'eorie de Kummer nous donne la suite exacte suivante (voir par exemple \cite[III.4]{Milne})
\[0\to H^1(X_L,\mu_\ell)\to \pic X_L\xrightarrow{\ell}\pic X_L.\]
En notant que $L$ est s\'eparablement clos, on voit que $H^1(X_L,\mu_\ell)=H^1(X_L,\z/\ell\z)$, car les faisceaux $\mu_\ell$ et $\z/\ell\z$
sont isomorphes. De plus, le th\'eor\`eme de changement de base propre (cf. \cite[IV, Corollary 2.3]{Milne}) nous dit que
$H^1(X_L,\z/\ell\z)=H^1(X_{L^\al},\z/\ell\z)$, o\`u $L^\al$ est une cl\^oture alg\'ebrique de $L$. Ce dernier groupe classifie les
$X_{L^\al}$-torseurs sous $\z/\ell\z$ (cf. \cite[III, 4.6 et 4.8]{Milne}).\\

Il suffit alors de montrer que $H^1(X_{L^\al},\z/\ell\z)=0$ pour avoir $(\pic X_L)[\ell]=0$.  En raisonnant par l'absurde, supposons qu'il
existe un torseur non trivial, i.e. non isomorphe au torseur correspondant \`a $\ell$ copies disjointes de $X_{L^\al}$ avec l'action \'evidente de
$\z/\ell\z$. Ce torseur $Y\to X_{L^\al}$ est alors forc\'ement connexe, car la quantit\'e de composantes connexes doit diviser $\ell$ et si
$Y$ \'etait form\'e de $\ell$ composantes, elles seraient toutes forc\'ement isomorphes \`a $X_L$ et $Y$ serait un torseur trivial.
On voit alors que $Y$ est un rev\^etement \'etale connexe cyclique d'ordre $\ell$ de $X_L$, ce qui entra\^\i ne l'existence d'un \'el\'ement non trivial
dans $\pi^1(X_{L^\al})$. Or, ce dernier groupe est nul d'apr\`es un r\'esultat de Koll\'ar (cf. \cite[Corollaire 3.6]{Debarre}) puisque $X$ est
s\'eparablement rationnellement connexe.\\

Int\'eressons-nous maintenant au sch\'ema $\Pic X_L$ (le groupe $\pic X_L$ correspond alors \`a $\Pic X_L(L)$). D'apr\`es
\cite[8.4, Theorem 3]{NeronModels}, on sait que le sous-sch\'ema $\Pic^0X_L$ correspondant \`a sa composante connexe est un sch\'ema en groupes propre,
bien que pas forc\'ement r\'eduit. Si l'on consid\`ere alors la vari\'et\'e r\'eduite $(\Pic X_L)_\mathrm{red}$ associ\'ee \`a $\Pic X_L$, on voit
que sa composante connexe $(\Pic^0X_L)_\mathrm{red}$ est propre, connexe et r\'eduite. De plus, on voit bien que l'on a
$(\Pic X_L)_\mathrm{red}(L)=\Pic X_L(L)$ (voir par exemple \cite[Exp. VI$_\text{A}$, 0.2]{SGA3I}).

Comme $k$ est suppos\'e parfait, toute $k$-vari\'et\'e r\'eduite est g\'eom\'etriquement r\'eduite. La vari\'et\'e $(\Pic^0X_k)_\mathrm{red}$ \'etant
donc un sch\'ema en groupes sur $k$ g\'eom\'etriquement r\'eduit, on d\'eduit qu'elle est lisse sur $k$, ce qui nous dit que
$(\Pic^0X_L)_\mathrm{red}$ est lisse sur $L$ pour toute extension $L$ de $k$, notamment pour $L=\bar k$ ou $\bar K$. La vari\'et\'e
$(\Pic^0 X_L)_{\mathrm{red}}$ est alors une vari\'et\'e
ab\'elienne qui v\'erifie en plus que $(\Pic^0X_L)_\mathrm{red}(L)$ est sans $\ell$-torsion puisqu'il correspond \`a un sous-groupe de $\Pic X_L(L)$.
Le r\'esultat classique correspondant \`a la derni\`ere proposition dans \cite[II.6]{Mumford} nous dit que cette vari\'et\'e est alors de
dimension $0$, donc triviale. On voit alors que le sous-groupe $\Pic^0X_L(L)$ est trivial (bien que le sch\'ema $\Pic^0X_L$ pourrait ne pas l'\^etre).\\

Cela \'etant d\'emontr\'e, on voit donc que l'on a $\pic X_L=\ns X_L$. De plus, comme le sch\'ema $\underline{\ns} X_L$ est \'etale sur $L$
(cf. \cite[Exp. VI$_\text{A}$, 5.5]{SGA3I}), on voit que $\underline{\ns} X_L(L)=\underline{\ns} X_L(L^\al)=\underline{\ns} X_{L^\al}(L^\al)$.
Or, ce groupe \'etant invariant par extension de corps alg\'ebriquement clos (cf. \cite[Proposition 3.1]{MaulikPoonen}), on trouve que
\begin{multline*}
\pic X_{\bar k}=\Pic X_{\bar k}(\bar k)=\underline{\ns} X_{\bar k}(\bar k)=\underline{\ns} X_{k^\al}(k^\al)=\underline{\ns} X_{K^\al}(K^\al)\\
=\underline{\ns} X_{\bar K}(\bar K)=\Pic X_{\bar K}(\bar K)=\pic X_{\bar K}.
\end{multline*}
Enfin, il est un fait connu que le groupe de N\'eron-Severi est finiment engendr\'e (cf. par exemple le corollaire 2 au th\'eor\`eme 3 de
\cite[IV.19]{Mumford} pour le cas sur un corps alg\'ebriquement clos, ce qui suffit bien dans ce cas d'apr\`es les derni\`eres \'egalit\'es).
\end{proof}

\begin{proof}[D\'emonstration du Th\'eor\`eme \ref{theoreme invariance brnr}]
Notons d'abord que l'on a $\Gamma_k=\Gamma_{K}/\Gamma_{\bar kK}$. On a la suite de restriction-inflation correspondante \`a ce quotient pour le
$\Gamma_{K}$-module $\pic X_{\bar K}$ :
\begin{equation}\label{equation suite inflation restricition pour les pic}
0 \to H^1(\Gamma_k,(\pic X_{\bar K})^{\Gamma_{\bar kK}}) \to H^1(\Gamma_{K},\pic X_{\bar K}) \to H^1(\Gamma_{\bar kK},\pic X_{\bar K}).
\end{equation}

D'autre part, il est \'evident que le morphisme de changement de base $\pic X_{\bar k}\to\pic X_{\bar K}$ se factorise par
\[\pic X_{\bar k}\to(\pic X_{\bar K})^{\Gamma_{\bar kK}}\to\pic X_{\bar K}.\]
Or, le th\'eor\`eme \ref{theoreme pic invariant} nous dit que la composition des deux fl\`eches est un isomorphisme, et puisque la deuxi\`eme
fl\`eche est une injection, chacune de fl\`eches est aussi un isomorphisme. En particulier, on voit que $\pic X_{\bar K}$ est invariant par
$\Gamma_{\bar kK}$.\\

On trouve alors l'isomorphisme canonique
\[H^1(\Gamma_k,\pic X_{\bar k}) \xrightarrow{\sim} H^1(\Gamma_k,(\pic X_{\bar K})^{\Gamma_{\bar kK}}), \]
ce qui fait que la suite (\ref{equation suite inflation restricition pour les pic}) devient
\[ 0 \to H^1(\Gamma_k,\pic X_{\bar k}) \to  H^1(\Gamma_{K},\pic X_{\bar K}) \to H^1(\Gamma_{\bar kK},\pic X_{\bar K}). \]
D'autre part, le fait que $\pic X_{\bar K}$ soit $\Gamma_{\bar kK}$-invariant nous dit que
\[H^1(\Gamma_{\bar kK},\pic X_{\bar K})=\hom(\Gamma_{\bar kK},\pic X_{\bar K}).\]
Le th\'eor\`eme \ref{theoreme pic invariant} nous dit aussi que le groupe $\pic X_{\bar K}$ est sans $\ell$-torsion pour tout $\ell\neq p$. Le groupe
$\Gamma_{\bar kK}$ \'etant profini, on voit que $\hom(\Gamma_{\bar kK},\pic X_{\bar K})$ est un $p$-groupe. Finalement, puisque les groupes de la
suite sont de torsion, on peut consid\'erer leur partie de torsion premi\`ere \`a $p$. On retrouve alors
\[0 \to H^1(\Gamma_k,\pic X_{\bar k})\{p'\} \to H^1(\Gamma_{K},\pic X_{\bar K})\{p'\} \to 0.\]
On r\'e\'ecrit enfin cet isomorphisme comme
\[0 \to \brnral V\{p'\} \to \brnral V_{K}\{p'\} \to 0, \]
ce qui conclut.
\end{proof}

\section{Quelques r\'esultats sur les extensions ins\'eparables}\label{section inseparable}
L'\'etude du groupe de Brauer non ramifi\'e dans le cadre g\'en\'eral en caract\'eristique positive (i.e. sans forc\'ement avoir le droit \`a une
compactification lisse) entra\^\i ne l'\'etude des r\'esidus, donc des anneaux $A$ de valuation discr\`ete sur le corps de fonctions $K=k(V)$
de la vari\'et\'e $V$ consid\'er\'ee, cf.
\cite{ColliotSantaBarbara}. Le
th\'eor\`eme de Gabber \cite[Expos\'e X, Theorem 2.1]{ILO} nous permet de nous ramener au cas o\`u l'on compte avec une compactification lisse,
mais au prix de consid\'erer une extension finie (souvent ins\'eparable) de $K$. Les r\'esultats suivants, utilis\'es dans la preuve du th\'or\`eme
\ref{theoreme sans compactification lisse corps fini} nous montrent que l'on ne perd pas d'information par ce changement de base.

\begin{pro}\label{proposition extensions avd}
Soit $A$ un anneau de valuation discr\`ete de corps de fractions $K$ de caract\'eristique $p$. Soit $L/K$ une extension finie purement ins\'eparable
et soit $B$ la fermeture int\'egrale de $A$ dans $L$. Alors $B$ est un anneau de valuation discr\`ete et les degr\'es r\'esiduel et de ramification
$f$ et $e$ de l'extension $B/A$ v\'erifient $ef|[L:K]$.
\end{pro}

\begin{proof}
Puisque $L/K$ est purement ins\'eparable, on sait qu'il s'agit d'une extension radicielle, laquelle on peut d\'ecomposer en une suite d'extensions
radicielles d'ordre $p$. La multiplicativit\'e des degr\'es $e$ et $f$ \'etant un fait classique, il est facile de voir que l'on peut se restreindre au
cas o\`u $[L:K]=p$. Soit alors $x_0\in K$ tel que $L=K(x_0^{1/p})$. On sait que l'on peut supposer $x_0\in A$. On note $v$ la valuation sur $A$ et on
fixe aussi $\pi\in A$ tel que $v(\pi)=1$.\\

Supposons d'abord que $(v(x_0),p)=1$. Il existe alors un entier $n$ premier \`a $p$ tel que $v(x_0^n)\equiv 1\mod p$. Alors, quitte \`a changer
$x_0$ par $x_0^n$ divis\'e par une puissance convenable de $\pi^p$ (\'el\'ement dont la racine $p$-i\`eme engendre toujours l'extension $L/K$), on
peut supposer que $v(x_0)=1$. On peut alors d\'efinir une valuation discr\`ete $w$ sur $L$ par la formule suivante.

Pour $b=a_0+a_1x_0^{1/p}+\cdots +a_{p-1}x_0^{(p-1)/p}\in L$, on pose
\[w(b)=\min_{0\leq i\leq p-1} \{pv(a_i)+i\}.\]
Il est un exercice facile de v\'erifier que cette formule d\'efinit bien une valuation discr\`ete sur $L$ dont l'anneau correspondant est $A[x_0^{1/p}]$.
Une fois cela admis, il est \'evident que l'on a $B=A[x_0^{1/p}]$, car tout \'el\'ement $b\in L\smallsetminus K$ a un polyn\^ome minimal de la forme
$x^p-a$, avec $a\in K$, et alors $w(b)\geq 0\Leftrightarrow v(a)\geq 0$. Dans ce cas, puisque $B$ est un $A$-module de type fini, la th\'eorie
classique (cf. \cite[I.\S4]{SerreCorpsLocaux}) nous dit que l'on a $ef=p$ (et donc, plus pr\'ecis\'ement, on a $e=p$ et $f=1$).\\

Supposons maintenant qu'il n'existe pas d'\'el\'ement $x_0$ engendrant $L/K$ tel que $(v(x_0),p)=1$. Quitte \`a diviser encore une fois par une
puissance convenable de $\pi^p$, on peut supposer que l'on a $v(x_0)=0$. On affirme que $B$ est un anneau de valuation discr\`ete d'id\'eal maximal
$\pi B$, donc d'uniformisante $\pi$.

Pour montrer cela, on montre d'abord que $\pi B$ est le seul id\'eal maximal de $B$, i.e. que l'on a $B^*=B\smallsetminus \pi B$. En effet,
soit $b$ un \'el\'ement dans $B\smallsetminus \pi B$. Si $b\in A$, on sait alors qu'il appartient \`a $A\smallsetminus\pi A=A^*$ et il est donc
inversible. Sinon, on sait que $b^p\in A$ et que
$v(b^p)\equiv 0\mod p$ par hypoth\`ese. Or, si $v(b^p)\geq p$, on peut \'ecrire $b^p=\pi^pa$ avec $a\in A$, ce qui nous dit que $b=\pi a^{1/p}$ et
on a bien $a^{1/p}\in B$. On trouve que $b\in\pi B$, ce qui est absurde. Finalement, on a que $v(b^p)=0$, ce qui entra\^ine que
$v(b^{-p})=-v(b^p)=0$ et alors $b^{-p}\in A$, nous donnant enfin que $b^{-1}\in B$ et alors $b\in B^*$. On en d\'eduit que $B$ est un anneau local,
int\`egre, diff\'erent de $L$ et d'id\'eal maximal principal. D'apr\`es \cite[VI, \S3, Proposition 9]{BourbakiAlgComm1-7}, il suffit pour conclure de d\'emontrer que
$\bigcap_{n\in\n}\pi^n B=0$. Or, cela est tr\`es simple : soit $b\in B$ un \'el\'ement non nul. Alors $b^p\in A$ et il existe $n\in\n$ tel que
$\pi^{-np}b^p\not\in A$. On voit alors que $\pi^{-n}b\not\in B$, ce qui est la m\^eme chose que $b\not\in\pi^n B$.

Il suffit alors pour conclure de noter que dans ce cas on a clairement $e=1$, tandis que $f=1$ ou $p$, d\'ependant de la nature de $x_0$ dans le
corps r\'esiduel de $A$. Dans tous les cas, on a bien $ef|[L:K]$.
\end{proof}

\begin{rem}
Ce r\'esultat peut aussi \^etre obtenu \`a partir de la d\'emonstration de \cite[Chapter 4, Proposition 1.31]{Liu}.
\end{rem}

\begin{cor}\label{corollaire extensions avd}
Soit $A$ un anneau de valuation discr\`ete de corps de fractions $K$ de caract\'eristique $p$. Soit $\ell\neq p$ un nombre premier et soit $L/K$ une
extension de degr\'e premier \`a $\ell$. Alors il existe un anneau de valuation discr\`ete $B\supset A$ de corps de fractions $L$, contenant la fermeture
int\'egrale de $A$ dans $L$ et tel que les degr\'es r\'esiduel et de ramification $f$ et $e$ de l'extension $B/A$ v\'erifient $(ef,\ell)=1$.
\end{cor}

\begin{proof}
Consid\'erons la sous-extension s\'eparable maximale $K'/K$ de $L/K$. Soit $A'$ la fermeture int\'egrale de $A$ dans $K'$ et soient $P_i$ les id\'eaux
premiers au-dessus de l'id\'eal maximal de $A$. On sait que $\sum e_if_i=[K':K]$ et que ce degr\'e est premier \`a $\ell$. Il existe alors un $i_0$ tel que
$(e_{i_0}f_{i_0},\ell)=1$. En localisant $A'$ autour de l'id\'eal $P_{i_0}$, on obtient un anneau de valuation discr\`ete $A'_{i_0}$ de corps de fractions
$K'$. La proposition pr\'ec\'edente nous dit alors que la fermeture int\'egrale $B$ de $A'_{i_0}$ est un anneau de valuation discr\`ete de corps de
fractions $L$ et tel que les degr\'es $e'$ et $f'$ de l'extension $B/A'_{i_0}$ v\'erifient $e'f'=p^\alpha$ pour un certain entier positif $\alpha$. La
multiplicativit\'e des degr\'es nous dit alors que l'on a $ef=e_{i_0}f_{i_0}e'f'=[K':K]p^\alpha$ pour l'extension $B/A$, ce qui est bien un nombre premier
\`a $\ell$. Il est enfin \'evident que $B$ contient la fermeture int\'egrale de $A$, car il contient celle de $A'_{i_0}\supset A$.
\end{proof}

\begin{pro}\label{proposition isomorphisme des galois pour les ext inseparables}
Soit $K$ un corps de caract\'eristique $p>0$ et soit $L/K$ une extension finie purement ins\'eparable. Soient $\Gamma_L$ et $\Gamma_K$ les groupes
de Galois absolus respectifs de $L$ et $K$. Alors le morphisme naturel $\pi:\Gamma_L\to\Gamma_K$ est un isomorphisme.
\end{pro}

\begin{proof}
Puisque $L/K$ est purement ins\'eparable, on sait qu'elle correspond \`a une suite d'extensions radicielles de degr\'e $p$. Il est \'evident alors qu'il
suffit de d\'emontrer la proposition pour $L=K(\alpha^{1/p})$ pour certain $\alpha\in K$.\\

Soit $\bar L$ une cl\^oture s\'eparable de $L$. Il est facile de voir qu'elle contient une (unique) cl\^oture s\'eparable $\bar K$. Soit alors
$\sigma\in\Gamma_K=\aut_K(\bar K)$. Pour montrer la surjectivit\'e de $\pi$, il s'agit d'\'etendre $\sigma$ en un automorphisme $\sigma_L$
de $\bar L$ fixant $L$. Soit donc $\beta\in \bar L\smallsetminus \bar K$ et soit $P(x)=\sum_{i=0}^na_ix^i\in L[x]$ son polyn\^ome minimal.
On voit alors que $\beta^p\in \bar K.$ En effet, $P(\beta)=0$ entra\^\i ne
\[P(\beta)^p=\sum_{i=0}^na_i^p\beta^{pi}=\sum_{i=0}^na_i^p(\beta^p)^i=Q(\beta^p)=0,\]
et on voit que les co\'efficients de $Q$ sont dans $K$, car $a_i\in L=K[\alpha^{1/p}]$ donne clairement $a_i^p\in K$. On pose alors
\[\sigma_L(\beta):=\sigma(\beta^p)^{1/p}\quad\text{pour tout }\beta\in \bar L.\]
Il est facile de voir que $\sigma_L(\beta)$ est uniquement d\'efini (il n'existe qu'une racine $p$-i\`eme en caract\'eristique $p$) et que
$\sigma_L$ est bien un morphisme de corps dont la restriction \`a $\bar K$ co\"\i ncide avec $\sigma$. Il suffit alors de montrer que
$\sigma(\beta^p)^{1/p}\in \bar L$ et que $\sigma_L$ fixe $L$. Or, on a
\[P(\sigma_L(\beta))^p=\sum_{i=0}^na_i^p\sigma_L(\beta)^{pi}=\sum_{i=0}^na_i^p\sigma(\beta^p)^{i}=Q(\sigma(\beta^p)),\]
et $\sigma(\beta^p)$ et $\beta^p$ ont forc\'ement le m\^eme polyn\^ome minimal, d'o\`u $P(\sigma_L(\beta))^p=Q(\beta^p)=0$ et alors
$P(\sigma_L(\beta))=0$, ce qui nous dit que $\sigma_L(\beta)\in \bar L$. De plus, pour $\beta\in L$, on a $P(x)=x-\beta$ et alors
$P(\sigma_L(\beta))=0$ entra\^\i ne clairement que $\sigma_L(\beta)=\beta$.\\

Pour l'injectivit\'e, soit $\sigma\in \Gamma_L$ un automorphisme fixant $\bar K$ et soit $\beta\in \bar L\smallsetminus \bar K$. On a d\'ej\`a montr\'e que
$\beta^p\in \bar K$, d'o\`u $\sigma(\beta)^p=\sigma(\beta^p)=\beta^p$ et alors $\sigma(\beta)=\beta$ par l'unicit\'e des racines $p$-i\`emes en
caract\'eristique $p$, ce qui conclut.
\end{proof}

\begin{cor}\label{corollaire isomorphisme extensions purement inseparables}
Soit $K$ un corps de caract\'eristique $p>0$ et soit $L/K$ une extension finie purement ins\'eparable. Soit $G$ un $K$-groupe de torsion
(ou plus g\'en\'eralement, tel que $G(\bar K)=G(\bar L)$). Alors le morphisme de restriction en cohomologie galoisienne (ou \'etale)
\[H^1(K,G)\to H^1(L,G),\]
est un isomorphisme.
\end{cor}

\begin{proof}
En effet, puisque le morphisme $\Gamma_L\to\Gamma_K$ est un isomorphisme compatible avec l'action sur $G(\bar K)=G(\bar L)$, le morphisme
$H^1(\Gamma_K,G(\bar K))\to H^1(\Gamma_L,G(\bar L))$ est un isomorphisme du point de vue de la cohomologie de groupes profinis, d'o\`u le r\'esultat.
\end{proof}

\end{document}